\documentclass[10pt]{amsart}

\usepackage{amsmath}
\usepackage{amsthm}
\usepackage{amsfonts}
\usepackage{amssymb}
\usepackage[dvipsnames]{xcolor}
\definecolor{darkgreen}{rgb}{0,0.45,0}
\definecolor{darkred}{rgb}{0.75,0,0}
\definecolor{darkblue}{rgb}{0,0,0.6}
\usepackage[colorlinks,citecolor=darkgreen,linkcolor=darkblue,urlcolor=darkred]{hyperref}
\usepackage{tikz}
\usepackage{tikz-cd}
\usepackage{stmaryrd}
\usepackage{float}
\usepackage{bbm}
\usepackage{enumitem}

\newcommand{\one}{\mathbbm{1}}

\usepackage{ arydshln}

\usepackage{multicol}

\usepackage[normalem]{ulem}

\newcommand{\iso}{\cong}
\renewcommand{\to}{\rightarrow}

\newcommand{\Q}{\mathbb{Q}}
\newcommand{\C}{\mathbb{C}}
\newcommand{\R}{\mathbb{R}}
\newcommand{\G}{\mathbb{G}}
\newcommand{\Z}{\mathbb{Z}}

\newcommand{\A}{\mathbb{A}}

\newcommand{\p}{\mathfrak{p}}
\newcommand{\g}{\mathfrak{g}}

\newcommand{\M}{\mathcal{M}}

\renewcommand{\H}{\mathcal{H}}
\renewcommand{\O}{\mathcal{O}}
\newcommand{\E}{\mathcal{E}}

\newcommand{\Lie}{\operatorname{Lie}}

\newcommand{\End}{\operatorname{End}}
\newcommand{\Res}{\operatorname{Res}}
\newcommand{\Hom}{\operatorname{Hom}}
\newcommand{\Gal}{\operatorname{Gal}}

\newcommand{\dR}{\mathrm{dR}}
\newcommand{\SD}{\mathrm{SD}}

\newcommand{\SL}{\operatorname{SL}}
\newcommand{\GL}{\operatorname{GL}}

\newcommand{\GSp}{\operatorname{GSp}}
\newcommand{\Sp}{\operatorname{Sp}}
\newcommand{\GO}{\operatorname{GO}}

\newcommand{\Sym}{\operatorname{Sym}}
\newcommand{\St}{\operatorname{St}}

\newcommand{\sgn}{\operatorname{sgn}}

\newcommand{\ad}{\mathrm{ad}}
\newcommand{\pol}{\mathrm{pol}}

\newcommand{\can}{\mathrm{can}}
\newcommand{\sub}{\mathrm{sub}}

\newcommand{\cal}[1]{\mathcal{#1}}

\newcommand{\Asai}{{\operatorname{Asai}}}
\newcommand{\Ad}{\operatorname{Ad}}

\newcommand{\diag}{\operatorname{diag}}

\newcommand{\sigmab}{{\bar{\sigma}}}
\newcommand{\gt}{\widetilde{\gamma}}
\newcommand{\e}{\mathbf{e}}
\newcommand{\f}{\mathbf{f}}
\newcommand{\ba}{\boldsymbol{\alpha}}
\newcommand{\vb}{\mathbf{v}}

\newcommand{\D}{\mathcal{D}}
\newcommand{\PD}{\operatorname{PD}}
\newcommand{\CH}{\operatorname{CH}}

\theoremstyle{plain}
\newtheorem{theorem}{Theorem}[section]

\newtheorem{proposition}[theorem]{Proposition}
\newtheorem{lemma}[theorem]{Lemma}
\newtheorem{corollary}[theorem]{Corollary}

\newtheorem{conjecture}[theorem]{Conjecture}
\newtheorem{hypothesis}[theorem]{Hypothesis}

\newtheorem{thmx}{Theorem}

\newtheorem{conjx}[thmx]{Conjecture}

\theoremstyle{definition}
\newtheorem{definition}[theorem]{Definition}
\newtheorem*{definition*}{Definition}
\newtheorem{example}[theorem]{Example}
\newtheorem{examples}[theorem]{Examples}

\newtheorem{remark}[theorem]{Remark}

\numberwithin{equation}{section}

\begin{document}

\title{Motivic action for Siegel modular forms}
\author{Aleksander Horawa and Kartik Prasanna}
\date{\today}
\maketitle

\begin{abstract}
We study the coherent cohomology of automorphic sheaves corresponding to Siegel modular forms $f$ of low weight on $\GSp(4)$ Shimura varieties. Inspired by the work of Prasanna--Venkatesh on singular cohomology of locally symmetric spaces, we propose a conjecture that explains all the contributions of a Hecke eigensystem to coherent cohomology in terms of the action of a motivic cohomology group.  Under some technical conditions, we prove that our conjecture is equivalent to Beilinson's conjecture for the adjoint $L$-function of $f$. We also prove some unconditional results in special cases. For a lift $f$ of a Hilbert modular form $f_0$ to $\GSp(4)$, we produce elements in the motivic cohomology group for which the conjecture holds, using the results of Ramakrishnan on the Asai $L$-function of $f_0$. For a lift $f$ of a Bianchi modular form~$f_0$ to $\GSp(4)$, we show that our conjecture for $f$ is equivalent to the conjecture of Prasanna-Venkatesh for~$f_0$, thus establishing a connection between the motivic action conjectures for locally symmetric spaces of non-hermitian type and those for coherent cohomology of Shimura varieties.  	
\end{abstract}

\setcounter{tocdepth}{1}
\tableofcontents

\section{Introduction}


A recent conjecture of Venkatesh and the second-named author~\cite{Prasanna_Venkatesh} proposes a surprising relationship between the singular cohomology of locally symmetric spaces and higher Chow groups. For example,  the simplest instance of this conjecture predicts that the Hecke isotypic components of the singular cohomology of Bianchi modular threefolds are related to regulators of elements in $\CH^2(E\times E,1)$ for elliptic curves $E$ over imaginary quadratic fields. This prediction is rather mysterious since the locally symmetric spaces in this case are only real manifolds and have no underlying algebraic structure.

In this paper, we propose a similar relationship between the coherent cohomology of  automorphic sheaves on Siegel modular threefolds and regulators of elements in $\CH^2(A\times A,1)$ for abelian surfaces $A$ defined over $\Q$ (Conjecture~\ref{conj:A}).  The setting is at first sight quite different from the case considered in ~\cite{Prasanna_Venkatesh}: the underlying locally symmetric space $X$ is now hermitian symmetric, but the automorphic form is not cohomological, namely it does not contribute to the singular cohomology of $X$. On the other hand, it does contribute to the coherent cohomology of a suitable sheaf on $X$.  This situation has also been considered in the recent thesis of Oh~\cite{Oh}; we explain the difference between our work and that of Oh in Section~\ref{subsec:comp} below.

After stating the conjecture, we prove three main results:
\begin{enumerate}
	\item First,  in Theorem~\ref{thm:Siegel} below,  we prove  Conjecture~\ref{conj:A} conditional on Beilinson's conjecture for the adjoint $L$-value of the associated Siegel modular forms and other standard hypotheses. (In this case, Beilinson's conjecture is the statement that a relevant piece of $\CH^2(A\times A,1)$ is of rank one,  generated by a single element $\alpha$ and that the $L$-value in question is rational, modulo the regulator of $\alpha$ and other standard motivic periods.)
	\item Next, for an abelian surface $A=\Res_{F/\Q} (E)$ obtained by restrictions of scalars from a {\it real quadratic} field $F$, we verify our conjecture for a rank one subspace of $\mathrm{CH}^2(A \times A, 1)$ (see Theorem~\ref{thm:Hilbert} below). Assuming the rank prediction of Beilinson, this proves our conjecture.
	\item Finally,  for an elliptic curve $E$ over an {\it imaginary quadratic} field $F$,  we prove a compatibility between our conjecture for the abelian surface $A=\Res_{F/\Q} (E)$ and the conjecture of~\cite{Prasanna_Venkatesh} for $E/F$ (see Theorem~\ref{thm:Bianchi}). Assuming the rank prediction of Beilinson, this proves that the two conjectures are equivalent.
\end{enumerate}

This last result suggests that the motivic action conjectures for non-hermitian locally symmetric spaces should be connected to (and perhaps implied by) the similar conjectures for coherent cohomology of hermitian locally symmetric spaces.  It would be of interest to study this phenomenon in more generality. 

We will now elaborate on our conjecture and state these results precisely. We will prioritize getting to the precise statements and defer the comments about the proofs to Section~\ref{subsec:proofs}.

\subsection{Siegel modular forms}

Let $f$ be a non-endoscopic cuspidal holomorphic Siegel modular form of weight $(k,2)$ and paramodular level $N$, with coefficients in the number field $\Q$. If $X_\Q$ is the Siegel modular threefold of paramodular level $N$ defined over $\Q$ and $\E_{k,2}$ is the automorphic vector bundle on $X_\Q$ of weight $(k,2)$, then~$f$ defines a section:
\begin{equation*}
	[f] \in H^0(X_\Q, \E_{k,2}).
\end{equation*}
Let $\pi^h$ be the automorphic representation of $\GSp_4(\A)$ generated by $f$ and let $\pi_f$ be its finite part. Then there is a unique generic automorphic representation $\pi^g$ of $\GSp_4(\A)$ with the same finite part. We may then consider a Whittaker-normalized vector $f^W \in \pi^g$ (Definition~\ref{def:Whittaker-norm}), and its associated cohomology class:
\begin{equation*}
	[f^W] \in H^1(X_\Q, \E_{k,2}) \otimes \R =  H^1(X_\R, \E_{k,2})
\end{equation*}
(see Proposition~\ref{prop:[f^W]_is_real}). Then $[f]$ and $[f^W]$ span the two dimensional subspace  $H^*(X_\R, \E_{k,2})_{\pi_f}$ giving  the total contribution of $\pi_f$ to the cohomology of the automorphic sheaf $\E_{k,2}$ (Theorem~\ref{thm:reps_in_coherent_cohomology}). We sometimes write $H^*(X_\R, \E_{k,2})_{f}$ for $H^*(X_\R, \E_{k,2})_{\pi_f}$.

As in ~\cite[Sec.\ 4]{Prasanna_Venkatesh}, we assume there is an adjoint motive $M(\pi_f,\Ad)$ associated with $\pi_f$ in the category of Chow motives over $\Q$, whose Galois representation is the adjoint representation of the Galois representation associated with $\pi_f$, and that is well defined up to isomorphism in this category. In the special case when $f$ corresponds to an abelian surface $A$ over $\Q$ (so that $(k,2) = (2,2))$, the motive $M(\pi_f, \Ad)$ may be realized explicitly as $\Sym^2 H^1(A) (1)$.  

Consider the motivic cohomology group
\[
H^1_{\mathcal M} := H^1_{\mathcal M}(M(\pi_f,\Ad)_\Z, \Q(1)),
\]
where as usual, the subscript $\Z$ indicates classes that extend to an integral model. 
Now, Beilinson's conjecture predicts that the regulator map to Deligne cohomology:
\begin{equation}\label{eqn:intro_reg}
	r_{\mathcal D} \colon  H^1_{\mathcal M}\otimes \R \to H^1_{\mathcal D}(M(\pi_f,\Ad)_\R, \R(1))=: H^1_{\mathcal D}
\end{equation}
is an isomorphism of one-dimensional real vector spaces. We consider a certain {\em natural generator} $\delta^\vee \in (H^1_{\mathcal D})^\vee$ (Definition~\ref{def:dual_natural_generator}; see~\eqref{eqn:intro_deltavee} below for an explicit description in the case of abelian surfaces) and define an action $\tau$ of $(H^1_{\mathcal D})^\vee$ on the $\pi_f$-isotypic part of the cohomology of $\mathcal E_{k,2}$ over $\R$ by setting:
\begin{align*}
	H^0(X_\R, \E_{k,2})_{\pi_f} & \overset{\tau(\cdot)} \to H^1(X_\R, \E_{k,2})_{\pi_f} \\
	\tau(\delta^\vee) : [f] & \mapsto [f^W].
\end{align*}
Via the regulator isomorphism~\eqref{eqn:intro_reg}, this gives an action of $(H^1_{\mathcal M})^\vee$ on the cohomology of $\mathcal E_{k,2}$.

\begin{conjx}[Conjecture~\ref{conj:motivic_action}]\label{conj:A}
	The action of $(H^1_{\mathcal M})^\vee$ preserves the rational structure $H^\ast(X_\Q, \E_{k,2})_{\pi_f}$. Equivalently, given a non-zero element $\alpha \in H^1_{\mathcal M}$, we have that:
	$$\frac{[f^W]}{\delta^\vee( r_{\mathcal D}(\alpha))} \in H^1(X_\Q, \E_{k,2})_{\pi_f} \subseteq H^1(X_\R, \E_{k,2})_{\pi_f}.$$
\end{conjx}

Our first main theorem is the following.

\begin{thmx}[Theorem~\ref{thm:main}]\label{thm:Siegel}
	Conjecture~\ref{conj:A} is implied by Beilinson's conjecture for the adjoint $L$-function $L(\pi_f, \Ad,s)$ at $s=1$ and Deligne's conjecture for some quadratic character twists of the spin $L$-function $L(\pi_f, \psi_\pm, s)$.
\end{thmx}


\subsection{Explication for modular abelian surfaces}\label{sec:intro_explication}

The Brumer--Kramer Conjecture~\ref{conj:Brumer_Kramer} predicts that a rational abelian surface of conductor $N$ corresponds to a cuspidal Siegel modular form~$f$ of weight $(2,2)$ and paramodular level $N$. In this case, as mentioned earlier, the motive $M(\pi_f,\Ad)$ equals $\Sym^2 H^1(A) (1)$, a quotient of $H^1(A) \otimes H^1(A)(1) \subseteq H^2(A \times A)(1)$. Therefore, a motivic cohomology class $\alpha \in H^1_{\mathcal M}$ admits the following interpretation as a higher Chow element: $\alpha = \{(D_i, \varphi_i)\} \in H^3_{\mathcal M}(A \times A, \Q(2))$ where:
\begin{itemize}
	\item $D_i$ is an irreducible divisor on $A \times A$,
	\item $\varphi_i$ is a meromorphic function on $D_i$,
	\item $\sum\limits_i \mathrm{div}(\varphi_i) = 0$.
\end{itemize} 

For any $\Q$-basis $\omega_1, \omega_2$ of $F^1 H^1_{\mathrm{dR}}(A) = H^0(A, \Omega^1_A)$, we define a {\em natural generator} of $H^1_{\mathcal D}$ by:
\begin{equation*}
	\delta := (2 \pi i) \left((\omega_1 \otimes \overline{\omega_2} + \overline{\omega_2} \otimes \omega_1) - (\omega_2 \otimes \overline{\omega_1} + \overline{\omega_1} \otimes \omega_2)\right).
\end{equation*}
Changing the basis of $F^1 H^1_{\mathrm{dR}}(A)$ rescales $\delta$ by the determinant of the change of basis. See Remark~\ref{rmk:rational_structure_on_Deligne} for a more canonical description of $\delta$ in terms of a natural rational structure on Deligne cohomology.

Since it is the dual of the Deligne cohomology group that acts on coherent cohomology, we also want to consider a dual of $\delta$. For that, we fix a polarization of the abelian surface $A$, consider the associated pairing:
$$\langle -, - \rangle_\pol \colon  \Sym^2 H^1(A) \times \Sym^2 H^1(A) \to \Q(-2),$$
and define
\begin{align}
	\delta^\vee := \frac{\pi^4}{\sqrt{\Delta_{\Ad(f)}}} \langle \delta, - \rangle_{\pol} \in (H^1_{\mathcal D})^\vee\label{eqn:intro_deltavee},
\end{align}
where $\Delta_{\Ad(f)}$ is the adjoint conductor of $f$, i.e.\ the conductor of the adjoint Galois representation associated with $f$. Since $H^1_\mathcal D$ is one-dimensional, $\delta^\vee$ only depends on the choice of polarization up to scalars in $\Q^\times$. (See Remark \ref{rmk:polarization}.)



Given $\alpha$ as above, this leads to the explicit expression:
\begin{align*}
	\delta^\vee(r_{\mathcal D}(\alpha)) & = \frac{\pi^4}{\sqrt{\Delta_{\Ad(f)}}} \frac{1}{(2 \pi i)^3} \sum_i \int\limits_{D_i(\C)} \log|\varphi_i| \cdot \delta \cup \omega_{\pol}  \in \R^\times,
\end{align*}
where $\eta_{\rm \pol} \in H^{1,1}(A) \cap H^2(A, \Z)$ is the cohomology class associated with the polarization on $A$ and  $\omega_{\rm pol} := \eta_{\rm pol} \boxtimes \eta_{\rm pol} \in H^{2,2}(A \times A) \cap H^4(A \times A, \Z)$.

Concretely,  Conjecture~\ref{conj:A} then predicts that the element
\begin{equation}\label{eqn:explicit_prediction}
	\frac{[f^W]}{ \frac{2 \pi i}{\sqrt{\Delta_{\Ad(f)}}} \displaystyle \sum_i \int\limits_{D_i(\C)} \log|\varphi_i| \cdot \delta \cup \omega_{\pol}} \in  H^1(X_\R, \E_{2,2}) 
\end{equation}
is rational in coherent cohomology, namely lives in the $\Q$-subspace $H^1(X_\Q, \E_{2,2})$.  This gives a subtle relationship between the Siegel modular form and the conjectural abelian surface associated with it.

\subsection{Special case: Hilbert modular forms}

Suppose $F$ is a real quadratic field and let $f$ be the (non-endoscopic) Yoshida lift of a Hilbert modular form $f_0$ of weight $(2,2)$. Using a theorem of Ramakrishnan~\cite{Ramakrishnan} on special values of Asai $L$-functions for Hilbert modular forms, we prove the following theorem.

\begin{thmx}[Theorem~\ref{thm:real_quadratic}]\label{thm:Hilbert}
	Suppose $f$ is the Yoshida lift of a Hilbert modular form $f_0$. Then there is an explicit rank one subspace of $(H^1_{\mathcal M})^\vee$ which acts rationally on coherent cohomology. Therefore, assuming the rank prediction of Beilinson's conjecture (Hypothesis~\ref{hyp:reg_is_isom}), Conjecture~\ref{conj:A} is true in this case.
\end{thmx}

This is one of the first unconditional results towards the motivic action conjectures. As far as we know, the only other known case is dihedral weight one modular forms (c.f.\ ~\cite{Horawa} over $\C$ and \cite{DHRV, lecouturier2022triple, zhang2023harris} mod $p^n$).

\subsection{Special case: Bianchi modular forms}
Suppose $F$ is now an imaginary quadratic field and $f$ is a Yoshida lift of a Bianchi modular form $f_0$ of weight $(2,2)$. In this case, we prove a compatibility between our conjecture for $f$ and the conjecture of~\cite{Prasanna_Venkatesh} for $f_0$. 

There are explicit Eichler--Shimura maps~\cite[Section 5.1]{tilouine2022integral}
\begin{align*}
	\omega^i \colon S_{2,2}^F(\mathfrak N) & \to H^i(X_0, \R) \\
	f_0 & \mapsto \omega_{f_0}^i
\end{align*}
for $i = 1,2$. Given a Whittaker-normalized Bianchi modular form $f_0$, we define the period $u^1(f_0) \in \R^\times$ to satisfy $\frac{\omega_{f_0}^1}{u^1(f_0)} \in H^1(X_0, \Q)_{f_0}$. Then the action $\tau_{\rm PV}$ of $ H^1_{\mathcal D}(M(f_0, \Ad)_\R, \R(1))^\vee$ on the $f_0$-isotypic part of the cohomology can be described by defining a natural dual generator $\eta^\vee \in H^1_{\mathcal D}(M(f_0, \Ad)_\R, \R(1))^\vee$ (Definition~\ref{def:nat_gen_for_Bianchi}) and setting:
\begin{align*}
	H^1(X_0, \R)_{f_0} & \overset{\tau_{\rm PV}(-)}\to H^2(X_0, \R)_{f_0} \\
	\tau_{\rm PV}(\eta^\vee) \colon \frac{\omega_{f_0}^1}{u^1(f_0)}  & \mapsto \omega_{f_0}^2.
\end{align*}
The main conjecture of~\cite{Prasanna_Venkatesh} in this case asserts that the resulting action of the dual motivic cohomology group is rational. It has an explicit form similar to \eqref{eqn:explicit_prediction} (see, for example, equation~\eqref{eqn:explicit_Bianchi_regulator}).

We may define maps:
\begin{align*}
	H^1(X_0, \Q)_{f_0} & \overset{\theta_1}\to H^0(X_\Q, \E_{2,2})_{f} \\
	H^2(X_0, \Q)_{f_0} & \overset{\theta_2} \to H^1(X_\Q, \E_{2,2})_f
\end{align*}
which are normalized to preserve the rational structures. We then prove the following compatibility of our conjecture with the conjecture of \cite{Prasanna_Venkatesh}.

\begin{thmx}[Theorem~\ref{thm:HP_implies_PV}]\label{thm:Bianchi}
	There is a natural isomorphism:
	$$d^\vee \colon H^1_{\mathcal D}(M(f, \Ad), \R(1))^\vee \to H^1_{\mathcal D}(M(f_0, \Ad)_\R, \R(1))^\vee$$
	under which the diagram:
	\begin{center}
		\begin{tikzcd}
			H^1(X_0, \Q)_{f_0} \otimes \R \ar[d, "\tau_{\rm PV} \circ d^\vee(-)"] \ar[r, "\theta_1"]\ar[d, swap, "\text{\cite{Prasanna_Venkatesh}}"] & H^0(X_\Q, \E_{2,2})_f \otimes \R  \ar[d, "\tau(-)"]  \ar[d, swap, "\textnormal{our action}"] \\
			H^2(X_0, \Q)_{f_0} \otimes \R \ar[r, "\theta_2"] & H^1(X_\Q, \E_{2,2})_f \otimes \R
		\end{tikzcd}
	\end{center}
	commutes, up to $\Q^\times$. In particular, assuming the rank prediction of Beilinson's conjecture (Hypothesis~\ref{hyp:reg_is_isom}), our Conjecture~\ref{conj:A} is equivalent to the main conjecture of~\cite{Prasanna_Venkatesh} in this setting.
\end{thmx}

\begin{remark}
The original construction of Galois representations attached to cohomological Bianchi modular forms is due to Taylor~\cite{Taylor} and Harris-Soudry-Taylor \cite{HST}. It proceeds via first making the functorial transfer to a (low-weight) Siegel modular form that is not cohomological (but that does contribute to coherent cohomology) and then using congruences to establish the existence of a Galois representation, analogous to the Deligne-Serre method for weight one modular forms. Thus it seems a natural question to compare the motivic action conjectures for the Bianchi (cohomological case) and the Siegel (coherent cohomology) case.  More generally,  it seems a natural question to study how the motivic action conjectures interact with Langlands functoriality. 
\end{remark}

\begin{remark}
	Note that Theorems~\ref{thm:Hilbert} and~\ref{thm:Bianchi} corroborate our choice of natural generator $\delta$ of the Deligne cohomology group $H^1_{\mathcal D}$, without appealing to Beilinson's conjecture. In the first, we prove a rationality result which is not conditional on the conjecture. In the second, we confirm that the action we defined is compatible with the completely canonical action defined by Prasanna--Venkatesh~\cite{Prasanna_Venkatesh} in the Bianchi case.
\end{remark}

\subsection{Overview of the proofs}\label{subsec:proofs}

The difficulty in defining the motivic action is that there is no canonical way to normalize the contributions to cohomology of Siegel modular forms. In the singular cohomology setting~\cite{Prasanna_Venkatesh}, a single automorphic representation contributes to all the cohomological degrees of singular cohomology via Eichler--Shimura maps. However, as discussed above in the Siegel case, the contributions to different degrees of the cohomology of an automorphic vector bundle come from different members of an archimedean $L$-packet. In order to define a motivic action, one has to normalize the automorphic embeddings for every element of the $L$-packet. 

The idea of the present paper is to normalize $[f] \in H^0(X_\C, \mathcal E_{k,2})_{\pi_f}$ to be rational in coherent cohomology and $[f^W] \in H^1(X_\C, \mathcal E_{k,2})_{\pi_f}$ using the Whittaker model. Then, we pick an explicit generator of the dual Deligne cohomology group $\delta^\vee \in H^1_{\mathcal D}(M(f, \Ad)_\R, \R(1))^\vee$ which should act by $[f] \mapsto [f^W]$.

The proof of Theorem~\ref{thm:Siegel} relies on three key ingredients:
\begin{enumerate}
	\item The relationship between $\langle f^W, f^W \rangle$ and the adjoint $L$-value which was proved by Chen--Ichino~\cite{Chen_Ichino} and Chen~\cite{chen2022algebraicity}.
	\item The relationship between the {\em Whittaker period} $c^W(f)$ and $c^+(M(f)) c^-(M(f))$ (a product of Deligne periods) which is proved in Theorem~\ref{thm:cW(f)=c+c-}, assuming Deligne's conjecture. A similar relationship features in~\cite{LPSZ} and \cite{Oh}. To obtain this relationship, we need to know that some twists of the spin $L$-function of $f$ do not vanish, which was recently proved in~\cite{radziwill2023nonvanishing}.
	\item An explicit version of Beilinson's conjecture which is proved to be equivalent to Beilinson's conjecture in Theorem~\ref{thm:Beilinson_for_Sym2}. This is the {\em non-critical} analogue of Yoshida's period relation~\cite{Yoshida} for $L$-values which are {\em critical} in the sense of Deligne~\cite{Deligne_Special_values}.
\end{enumerate}

The special cases in Theorems~\ref{thm:Hilbert} and~\ref{thm:Bianchi} both rely on the factorization of $L$-functions:
\begin{equation}\label{eqn:intro_fact}
	L(f, \Ad, s) = L(f_0, \Ad, s) \cdot L(f_0, \Asai, s+1).
\end{equation}
The essential ingredient is a dichotomy for real and imaginary quadratic fields which we summarize in
Table~\ref{table:dichotomy} below.

\begin{table}[h]
	\begin{tabular}{c|c|c}
		quadratic field $F$ & $L$-function & value at $s = 1$ \\ \hline
		real & $L(f_0, \Ad, s)$ & critical \\
		& $L(f_0, \Asai, s+1)$ & non-critical \\ \hline
		imaginary & $L(f_0, \Ad, s)$ & non-critical \\
		& $L(f_0, \Asai, s+1)$ & critical
	\end{tabular}
	\caption{For the $L$-values within the factorization~\eqref{eqn:intro_fact}, we indicate which are critical and non-critical in the sense of Deligne~\cite{Deligne_Special_values} for real and imaginary quadratic fields $F$.}
	\label{table:dichotomy}
\end{table}
In both cases, we have a good understanding of the critical $L$-values: due to Shimura~\cite{Shimura_HMF} in the real case, and Cremona--Whitley~\cite{cremona1994periods} and Ghate \cite{ghate1996critical} in the imaginary case. 

In the real case, the proof of Theorem~\ref{thm:Hilbert} then relies on the fact that Beilinson's conjecture for the Asai $L$-function was proved by Ramakrishnan~\cite{Ramakrishnan}. The key point is that in this case:
$$H^3_{\mathcal M}(M(f_0, \Asai)_\Z, \Q(2)) \iso H^1_{\mathcal M}(M(f, \Ad)_\Z, \Q(1))$$
and the Asai motive $M(f_0, \Asai)$ is (conjecturally) realized directly in the cohomology of the Hilbert modular surface $X_0$. Ramakrishnan~\cite{Ramakrishnan} constructed explicit classes in $H^3_{\mathcal M}(H^2(X_0)_\Z, \Q(2))$ using Hirzebruch--Zagier divisors and modular units on them, and related their regulators to the value of the Asai $L$-function. As in the case of weight one modular forms (c.f.\ ~\cite{Horawa} over $\C$ and \cite{DHRV, lecouturier2022triple, zhang2023harris} mod $p^n$), it seems that one can only prove instances of the motivic action conjectures in the presence of explicit motivic cohomology classes.

In the imaginary case, Theorem~\ref{thm:Bianchi} is equivalent to an explicit period relationship given in Theorem~\ref{thm:Bianchi_case_non_critical_parts}. To prove this theorem, we use again the relationship between adjoint $L$-values and Petersson norms to reduce the statement about the non-critical parts of the periods to a statement about the critical parts.

\subsection{Comparison with previous work}
\label{subsec:comp}

The idea that motivic cohomology should be related to the cohomology of locally symmetric spaces was first introduced in~\cite{Prasanna_Venkatesh, Venkatesh:Derived_Hecke, Galatius_Venkatesh}. Subsequently, Harris and Venkatesh~\cite{Harris_Venkatesh} proposed a similar conjecture for coherent cohomology associated with weight one modular forms by defining an action modulo powers of a prime. 

As discussed above, the difficulty of generalizing the action of~\cite{Prasanna_Venkatesh} over the complex numbers to coherent cohomology is that there seems to be no way to consistently normalize the automorphic realizations of different elements of an $L$-packet. One exception is the case of Hilbert modular forms: there are natural maps between the elements of the $L$-packet for $\SL_2(\R)^d$ coming from partial complex conjugation operators~\cite{Shimura_HMF, Harris_periods_I}. These ideas were used by the first-named author~\cite{Horawa} to define a motivic action for (partial) weight one Hilbert modular forms. For more general Shimura varieties, the elements of the $L$-packet have different ``sizes'' (see Figure~\ref{fig:K-types} for the Siegel case), so it seems difficult to define natural maps between them. 

A different approach is taken by Gyujin Oh~\cite{Oh}. Instead of normalizing the contributions to cohomology and defining an action that is conjecturally rational, he chooses {\em metrics} on the various cohomology groups and conjectures there is an isomorphism of graded metrized $\overline \Q$-vector spaces: 
$$\bigwedge (H^1_{\mathcal M}(M(\Pi, \Ad)_\Z, \overline \Q(1))) \otimes H^{i_{\min}}(X_{\overline \Q}, \mathcal E)_{\Pi} \iso \bigoplus_{i=i_{\min}}^{i_{\max}} H^{i}(X_{\overline \Q}, \mathcal E)_{\Pi},$$
where $X$ is a Shimura variety, $\mathcal E$ is an automorphic vector bundle, and $\Pi$ is an automorphic representation satisfying some technical assumptions.

The advantage of Oh's approach is that his conjecture applies to any automorphic vector bundle on any Shimura variety.  Oh proves that his conjecture is implied by Beilinson's conjecture in two cases: the case of Siegel threefolds and Picard modular surfaces, under similar hypothesis to ours.
However, it seems difficult to extract an explicit rationality statement from it; indeed, rescaling an automorphic form by a complex number of norm one does not change its volume, but it does change its rationality properties. 


%

\subsection*{Acknowledgements}

We thank Gyujin Oh for sharing an early draft of his preprint with us, and Akshay Venkatesh for his interest and helpful discussions. 
KP is especially grateful to Venkatesh for everything he learnt from him during the collaboration \cite{Prasanna_Venkatesh},  and
for allowing us to include in (Section~\ref{sec:imag_quad} of) this paper some unpublished computations from the $\GL_2$ case that had helped motivate some preliminary versions of the main conjecture 
in \cite{Prasanna_Venkatesh}. 
We are also grateful to Shih-Yu Chen, Atsushi Ichino, Tasho Kaletha, Kai-Wen Lan, Francesco Lemma, and James Newton both for asking us useful questions and for answering ours at various stages of the project.
Finally, we would like to thank the referee for carefully reading the paper and making several suggestions which greatly improved its content.

AH was supported by NSF grant DMS-2001293 and UK Research and Innovation
grant MR/V021931/1. KP was supported by NSF grant DMS-2001293. For the purpose of Open Access, the authors have applied a CC BY public copyright licence to any Author Accepted Manuscript (AAM) version arising from this submission.

\section{Preliminaries: Cohomology of automorphic vector bundles}


Let $G = \GSp_4$ be the symplectic group defined with respect to the matrix $J = \begin{pmatrix}
	0 & I_2 \\
	- I_2 & 0
\end{pmatrix},$ i.e.
$$\GSp_4(R) = \{g \in \GL_4(R) \ | \ {^t g} J g = \mu(g) \cdot J \text{ for some }\mu(g) \in R^\times  \}.$$

\subsection{Siegel modular threefolds and automorphic vector bundles}

Let $K_f$ be a neat open subgroup of $\GSp_4(\A_f)$. The open Shimura variety of level $K_f$, associated with $G$, has a canonical model $Y_{G, \Q}$ over $\Q$ such that
$$Y_{G, \Q}(\C) \iso G(\Q)_+ \backslash [\H_2 \times G(\A_f)]/K_f \iso \coprod_{i} \Gamma_i \backslash \H_2$$
where the subgroups $\Gamma_i \subseteq \Sp_4(\Z)$ corresponds to $K_f$. The Siegel modular threefold $Y_{G, \Q}$ can be identified with the moduli space of abelian surfaces with principal polarization and $K_f$-level structure.

According to~\cite{Faltings_Chai}, $Y_{G, \Q}$ has a toroidal compactification $X_{G, \Q}$ defined over $\Q$ associated to a choice of rational polyhedral cone decomposition $\Sigma$. We may choose $\Sigma$ so that $X_{G, \Q}$ is smooth and the boundary $D = \partial X_{G, \Q}$ is a simple normal crossings divisor. We fix this choice once and for all and suppress it from notation. 

We follow the exposition in \cite[Section VI.4]{Faltings_Chai} to define automorphic vector bundles on $Y_{G, \Q}$ and $X_{G, \Q}$. Recall that $P \subseteq G$ is the Siegel parabolic with Levi $M_P \iso \GL_2 \times \GL_1$. We write $P^\circ \subseteq G^\circ = \Sp_4$ for the kernels of the similitude character $\nu$. The maximal compact subgroup $K^\circ \subseteq G^\circ(\R)$ has complexification $K_\C^\circ \subseteq G^\circ(\C)$ and there exists $g \in G^\circ(\C)$ such that $g M_{P^\circ, \C} g^{-1 } = K_\C^\circ$. Note that $g \not\in G^\circ(\R)$; explicitly, we may take $g = \begin{pmatrix}
	-i I_2 & iI_2 \\
	I_2 & I_2
\end{pmatrix} \in \GSp_4(\C)$ (the significance of this detail is explained in Appendix \ref{subsection:Weylgroup}). This extends to $K = K^\circ \R_{>0}\subseteq G(\R)$.

The symmetric space $\H_2 \iso G^\circ(\R)/K^\circ \iso G(\R)^+ /K$ is naturally embedded in the compact dual $\H_2^\vee = (G/P)(\C)$ as a $G(\R)^+$-invariant open subset of the $G(\C)$-homogeneous space $\H_2^\vee$.

Consider an arithmetic subgroup $\Gamma$ of $G(\Q)$ and let $Y_\Gamma$ be the locally symmetric space $\Gamma \backslash \H_2$. We may construct an automorphic vector bundle on $Y_\Gamma$ as follows. For a finite-dimensional rational representation $\varrho \colon P \to \GL(V_\varrho)$, we define a $G(\C)$-equivariant vector bundle
$$G(\C) \times_{P(\C), \varrho} V_\varrho \text{ on }\H_2^\vee = (G/P)(\C).$$
Restricting it to $\H_2$ and quotienting by $\Gamma$ defines a vector bundle
$$\cal E_\varrho \text{ on }\Gamma \backslash \H_2.$$
Harris~ \cite{Harris:AVBI, Harris:AVBII} and Milne~\cite{Milne:AVB} showed that these bundles are in fact defined over the canonical model $Y_{G, \Q}$ of the Shimura variety for $G $. 

These vector bundles have two natural extensions to the toroidal boundary~\cite{Harris_delbar}, the canonical and subcanonical extension. We summarize this discussion in a theorem.

\begin{theorem}[{\cite[Theorem VI.4.2]{Faltings_Chai}}]
	\leavevmode
	\begin{enumerate}
		\item The automorphic vector bundle $\cal E$ on $Y_{G, \Q}$ has a canonical extension $\cal E^{\can}$ to the toroidal compactification $X_{G, \Q}$.
		
		\item This defines an exact functor
		$$\left\{ \begin{minipage}{0.24\textwidth}
			\begin{center}
				finite-dimensional $P(\C)$-representations
			\end{center}
		\end{minipage} \right\} \to \left\{ \begin{minipage}{0.2\textwidth}
			\begin{center}
				Hecke-equivariant vector bundles over~$X_{G, \Q}$
			\end{center}
		\end{minipage} \right\}$$
		which commutes with tensor products and duals. 
	\end{enumerate}
\end{theorem}

\begin{remark}
	Note that \cite[Theorem VI.4.2]{Faltings_Chai} is stated for $P^\circ$ instead of $P$. Indeed, the authors of loc.\ cit.\ work with $G^\circ$, $P^\circ$, and $M^\circ$; however, on page 222, they clarify that ``Whether one uses $\mathbf G$, $\mathbf Q$, $\mathbf M$ or $\mathbf G^\circ$, $\mathbf Q^\circ$, $\mathbf M^\circ$ will have no substantial consequence in this section: in any case, the only difference is the central $\G_m$-action.''. We decided to include the action of the center in the definition of the vector bundles as in \cite{Harris_Kudla:GSp(2), LPSZ}. Given a finite-dimensional representation $\varrho$ of $P(\C)$, the underlying vector bundle over $X_{G, \Q}$ only depends on the restriction $\varrho|_{P^\circ(\C)}$; however, the Hecke action does depend on the representation of $P(\C)$.
\end{remark}

If $\cal I(D)$ is the invertible sheaf on $X_{G, \Q}$ defining the divisor $D$, then we define the subcanonical extension $\cal E^\sub$ of $\cal E$ to $X_{G, \Q}$ to be $\cal E^\can \otimes \cal I(D)$.

Finally, given a finite-dimensional complex representation of $K = K^\circ \R_{>0}$, we may extend it to $M_{P}(\C)$ by analytic continuation and then inflate it to $P(\C)$. Therefore, we may also associate vector bundles to representations of $K$. In fact, any holomorphic $G(\R)^+$-equivariant vector bundle on $\H_2$ is $C^\infty$-isomorphic to a pullback of one associated to a representation of $K$.

\begin{definition}\label{def:sheaf_ass_to_rep}
	Let $\varrho$ be a representation of $M_P$ of highest weight $(k_1, k_2; m)$ for integers $k_1 \geq k_2 \geq 0$ and $m$. We then write:
	\begin{align}
		V_{k_1, k_2; m} = V_\varrho, \label{eqn:Vk1k2} \qquad	\cal E_{k_1, k_2; m}^\ast & = \cal E^\ast_\varrho \text{ for }\ast \in \{\can, \sub\}.
	\end{align}
\end{definition}
We fix the standard choice of basis for $V_{k_1, k_2; m}$, following \cite[pp.\ 902]{Moriyama}.

Recall that $M_P = \GL_2 \times \GL_1$. Let $\St$ be the standard representation of $\GL_2$, $\Sym^k$ the $k$th symmetric power of the standard representation,  $\det$ the determinant representation of $\GL_2$,  and $\mu$ be the identity representation of $\GL_1$. Then:
$$(\Sym^k \otimes {\det}^j) \boxtimes \mu^m \iso V_{k +j, j; m}.$$
For example, $\dim V_{k_1, k_2; m} = 1$ if and only if $k_1 = k_2$.

\begin{remark}\label{rmk:highest_wt_flipped}
	Due to an unfortunate clash in standard notation, the highest weight of the restriction of $V_{k_1, k_2; m}$ to $K \R_{>0}$ is $(-k_2, -k_1; m)$ and not $(k_1, k_2; m)$.
\end{remark}

\begin{examples}
	We describe several of the automorphic vector bundles in terms of the Siegel modular threefold $Y = Y_{G, \Q}$ and its toroidal compactification $X = X_{G, \Q}$, and the universal abelian surface $A$ over $Y$ and its semi-abelian extension to $X$. We write $D$ for the boundary divisor $\partial X$. See~\cite[pp.\ 258]{Faltings_Chai} for a discussion about the action of the center in these geometric examples.
	
	\begin{enumerate}
		\item When $\varrho = \St \boxtimes \mu$ is the representation $M = \GL_2 \times \GL_1$ of weight $(1, 0; 1)$, then
		\begin{align*}
			\cal E_{(1,0;1)} &  \iso \cal T_{A/Y}^\ast \iso \Omega^1_{A/Y}, \\
			\cal E_{(1,0;1)}^\can & \iso \cal T_{A/X}^\ast \iso \Omega^1_{A/X}. 
		\end{align*}
		\item When $\varrho = \Sym^2 \St \boxtimes \one$ is the representation of $M$ of weight $(2, 0; 0)$,
		\begin{align*}
			\cal E_{(2,0;0)} & \iso \Omega^1_Y, \\
			\cal E_{(2, 0; 0)}^\can & \iso \Omega^1_X(\log D), 
		\end{align*}
		\item When $\varrho = \det \boxtimes \mu^2$ is the representation of $M$ of weight $(1, 1; 2)$, then 
		\begin{align*}
			\cal E_{(1,1;2)} & \iso \omega_{A/Y} \iso \det \Omega^1_{A/Y}, \\
			\cal E_{(1,1;2)}^\can & \iso \omega_{A/X}(\log D) \iso \det \Omega^1_{A/X} (\log D), \\
			\cal E_{(1,1;2)}^\sub & \iso \omega_{A/X} \iso \det \Omega^1_{A/X}.
		\end{align*}
				
		\item When $\varrho = \det^{3} \boxtimes \mu^6$  is the representation of $M$ of weight $(3,3; 6)$, then
		\begin{align*}
			\cal E_{(3,3;6)} & = \cal K_Y \iso \omega_{A/Y}^3 & \text{ canonical bundle on $Y$}, \\
			\cal E_{(3,3;6)}^\can & \iso \cal K_X(\log D)  \iso \omega_{A/X}^3, \\
			\cal E_{(3,3;6)}^\sub & = \cal K_X, & \text{canonical bundle on $X$ \cite[Prop. (2.2.6)]{Harris_delbar}}.
		\end{align*}
	\end{enumerate}
\end{examples}

Henceforth, we omit $G$ from the notation and write $Y = Y_{G, \Q}$ and $X = X_{G, \Q}$. Although $X$ is defined over~$\Q$, we sometimes write~$X_\Q$ to emphasize that we mean the variety over~$\Q$ instead of $\C$.

By definition, a Siegel modular form $f$ of level $\Gamma$ and weight $\varrho$ is a section of the vector bundle $\cal E_\varrho$ over $Y_{\C}$, i.e.\ $f \in H^0(Y_{\C}, \cal E_\varrho)$. When $\varrho = \det^k$, $f$ is scalar-valued and $V_{\varrho} = V_{k,k;0}$.

Let $N \geq 1$ be an integer. The {\em principal level subgroup of level $N$} is:
$$\Gamma(N) = \begin{pmatrix}
	1_2 + NM_2(\Z) & NM_2(\Z) \\
	NM_2(\Z) & 1 + NM_2(\Z)
\end{pmatrix} \cap \Sp_4(\Z).$$

However, we will be more interested  in the following level structures.

\begin{definition}\label{def:paramodular}
	Let $N \geq 1$ be an integer. The {\em paramodular group} $K(N)$ is defined as:
	$$K(N) = \begin{pmatrix}
		\Z & N\Z & \Z & \Z \\
		\Z & \Z &  \Z & N^{-1}\Z \\
		\Z & N\Z & \Z& \Z \\
		N\Z & N\Z & N\Z & \Z
	\end{pmatrix} \cap \Sp_4(\Q).$$
\end{definition}

The reader can consult~\cite{Roberts_Schmidt_paramodular} for a detailed discussion of Siegel modular forms with this level structure. They develop a theory of newforms for these level structures.

A Siegel modular form admits a $q$-expansion by pulling back to the formal completion along boundary strata~\cite[{\S 7.1.2}]{lan_book}. Explicitly, in the case of Siegel modular forms, see~\cite[Section V.1]{Faltings_Chai} for full level and $\varrho = \det^k$, ~\cite{ichikawa2014vector} for principal level $N$ and general $\varrho$, and~\cite{florit2021abelian} for paramodular level $N$ and general~$\varrho$.

\begin{theorem}
	Let $X$ be a Siegel modular variety and let $A$ be a $\Q$-algebra. Then the following properties hold:
	\begin{enumerate}
		\item (Koecher principle) The restriction map
		$$H^0(X_{A}, \cal E^{\can}) \to H^0(Y_{A}, \cal E)$$
		is an isomorphism.
		\item (Higher Koecher principle) The restriction map
		$$H^1(X_{A}, \cal E^{\can}) \to H^1(Y_{A}, \cal E)$$
		is injective.
		\item ($q$-expansion principle) Let $f$ be a cuspidal Siegel modular form of level $\Gamma(N)$ with Fourier coefficients in a $\Q$-algebra $A$ and $B \subseteq A$ be a subalgebra. Then $f \in H^0(Y_{B}, \cal E_\varrho)$ if and only if the Fourier coefficients of $f$ at all cusps are in $B$.
		\item There is an isomorphism between the space of cuspidal Siegel modular forms with coefficients in $A$ and the image of the natural map
		$$H^0(X_{A}, \cal E^\sub) \to H^0(X_{A}, \cal E^\can).$$
	\end{enumerate}
\end{theorem}
\begin{proof}
	The Koecher principle is due to Faltings--Chai~\cite[Lemma V.1.5]{Faltings_Chai} (for full level and $\varrho = \det^k$), which was generalized to vector bundle and to higher degree  cohomology by Lan~\cite{Lan_higher_Koecher}. For the $q$-expansion principle, see more generally~\cite[Prop.~7.1.2.14]{lan_book}. Finally, (4) is \cite[Prop. (5.4.2)]{Harris_delbar}.
\end{proof}


\subsection{Higher cohomology of automorphic vector bundles}

We are interested in the coherent cohomology of the vector bundles $\cal E^\sub$ and $\cal E^\can$. We first note that the cohomology groups are independent of the choice of toroidal compactification and, as a consequence, the Hecke algebra acts on them.

\begin{theorem}[Harris {\cite{Harris_delbar},\cite{Blasius_Harris_Ramakrishnan}}]
	\leavevmode
	\begin{enumerate}
		\item The cohomology groups $H^i(X, \cal E^\can)$ and $H^i(X, \cal E^\sub)$ are independent of the choice of toroidal compactification, up to canonical isomorphism.
		\item There is a natural action of the Hecke algebra $\cal H = \cal H(G(\A_f), K_f)$ relative to level $K_f$ on $H^i(X, \cal E^\can)$ and on $H^i(X, \cal E^\sub)$.
	\end{enumerate}
\end{theorem}

Harris~\cite{Harris_delbar} and Su~\cite{Su} expressed the coherent cohomology of the vector bundle $\cal E^\can$ in terms of Dolbeault classes, i.e.\ relative Lie algebra cohomology. Before stating their results, we introduce some relevant notation. The Lie algebra $\mathfrak g$ of $G$ admits a Cartan decomposition $\mathfrak g = \mathfrak k \oplus \mathfrak p$ with $\mathfrak k = \Lie(K)$ and the Shimura cocharacter determines a further decomposition $\mathfrak g_\C = \mathfrak k_\C \oplus \mathfrak p_+ \oplus \mathfrak p_-$, where $\p_+$ (resp.\ $\p_-$) is identified with the holomorphic (resp.\ antiholomorphic) tangent space of $\H_2$ at $i I_2$. We identify the Lie algebra $\mathfrak P$ of the Siegel parabolic $P$ with $\mathfrak k \oplus \mathfrak p_-$. Moreover, we let $\mathfrak P^\circ = \mathfrak P \cap \mathfrak{sp}_4$. Note that $\mathfrak P^\circ/\mathfrak k^\circ = \mathfrak p_-$ via the above identifications.

\begin{theorem}[Harris, Su]\label{thm:Su}
	Let $\cal E_\varrho$ be an automorphic vector bundle over $Y$ associated to the representation $(\varrho, V_\varrho)$ of $K$ and $\cal E^{\mathrm{can}}$ be its canonical extension to $X$. Then there is a natural Hecke-equivariant isomorphism
	\begin{equation}\label{eqn:Su}
		H^i(X_{\C}, \cal E_\varrho^{\mathrm{can}}) \iso H^i(\frak P^\circ, K^\circ; \cal A(G)^{K_f} \otimes V_\varrho),
	\end{equation}
	where $V_\varrho$ has a non-trivial Hecke action denoted by $V_\circ$ in \cite[pp.\ 9]{Su}.
\end{theorem}

We now wish to describe the contributions of a holomorphic Siegel modular form of weight $(k_1, k_2; m)$ and level $\Gamma$ to coherent cohomology of automorphic sheaves. Let $\pi$ be the automorphic representation of $\GSp_4(\A)$ generated by $f$, which is non-CAP and cuspidal. Its infinite component $\pi_\infty$ is a holomorphic (limit of) discrete series $X_{\lambda; m}^1$ with Harish--Chandra parameter $(\lambda_1, \lambda_2; m) := (k_1 - 1, k_2 - 2; m)$. The $L$-packet at infinity also contains the generic (large) discrete series $X_{\overline \lambda; m}^2$ with Harish--Chandra parameter $(\overline \lambda, m) := (\lambda_1, -\lambda_2; m)$. See Appendix~\ref{app:GSp(4,R)} for the precise definitions and more details about the representation theory of $\GSp_4(\R)$. 

Writing $\pi_f$ for the finite part of $\pi$, the following lemma describes the $L$-packet $\Pi$ associated to such $\pi_f$ and the automorphic multiplicities.

\begin{lemma}\label{lemma:multiplicities}
	Let $\Pi$ be the $L$-packet of $\pi_f$ and consider the finite $L$-packet:
	$$\Pi_f = \{\pi_f' \text{ nearly equivalent to }\pi_f \ | \ \pi_f' \otimes \pi_\infty \text{ is automorphic for some }\pi_\infty \}.$$
	\begin{enumerate}
		\item If $\Pi$ is non-endoscopic, then 
		$$\Pi = \{\pi_f' \otimes X_{\lambda; m}^1, \pi_f' \otimes X_{\overline \lambda; m}^2  \ | \ \pi_f' \in \Pi_f \}$$ 
		consists of two elements for each $\pi_f' \in \Pi_f$:
		\begin{align*}
			\pi^h & = \pi_f' \otimes X_{\lambda; m}^1 \qquad \text{(holomorphic representation)}, \\
			\pi^g & = \pi_f' \otimes  X_{\overline \lambda; m}^2  \qquad \text{(generic representation)},
		\end{align*}
		and all of them occur in the automorphic spectrum with multiplicity one. 
		
		\item If $\Pi$ is weakly endoscopic {\em (Yoshida type)}, associated to an automorphic representation $\sigma$ of $\GL_2 \times_{\GL_1} \GL_2$, then there is a bijection:
		\begin{align*}
			\{ S \subseteq \{v \text{ finite place} \ | \ \sigma_v\text{ is discrete series}  \} \} & \to \Pi \\
			S & \mapsto \pi_S = \bigotimes (\pi_S)_v \nonumber
		\end{align*}
		where $(\pi_S)_v$ is generic if and only if $v \not \in S$ and
		\begin{equation*}
			(\pi_S)_\infty 	= \begin{cases}
				X_{\lambda; m}^1 & \text{$|S|$ is odd}, \\
				X_{\overline \lambda; m}^2 & \text{$|S|$ is even}.\\
			\end{cases}
		\end{equation*} 
		Each $\pi_S$ occurs in the automorphic spectrum with multiplicity one. 
		
		More specifically, $\pi_S$ is constructed as a theta lift from a definite or indefinite quaternion algebra over $\Q$ ramified at the finite places $S$.
	\end{enumerate}
\end{lemma}
\begin{proof}
	This follows from Arthur's classification~\cite{Arthur:Aut_reps, Gee_Taibi}. More specifically, see \cite[Theorem~10.1.3]{LSZ17:Euler-systems} for (1) and \cite{Roberts_Global_L-packets_GSp(2)_theta_lifts} or \cite[Corollary 5.4]{Weissauer} for (2).
\end{proof}

\begin{remark}
	The local Langlands correspondence for $\GSp(4)$~\cite{Gan_LLC_GSp(4)} together with the Arthur multiplicity formula give a complete description of the automorphic contributions of $\Pi_f$, but we decided not to spell this out here. In the case of Yoshida lifts from Hilbert modular surfaces, see also results of Roberts~\cite{Roberts_Global_L-packets_GSp(2)_theta_lifts}. 
	
	Later, we will focus on the unique element $\pi_f' \in \Pi_f$ which is generic at all places, because we will be interested in representations of paramodular level.
\end{remark}




Based on these results and Theorem~\ref{thm:Su}, we can describe the contributions of this $L$-packet $\Pi$ to coherent cohomology. Following~\cite{Prasanna_Venkatesh}, we consider the $\Pi$-isotypic component of the cohomology groups.

\begin{definition}
	Let $\pi_f$ be the finite part of a representation $\pi$ of $\GSp_4(\A)$ such that $\pi_f^{K_f} \neq 0$, and let $\Pi$ be the $L$-packet of $\pi_f$ as above. Consider the spherical Hecke algebra $\mathcal H^{K_f}$ away from the places where $K_{f, v}$ is not hyperspecial, and let $\chi \colon \cal H^{K_f} \to \Q(\pi_f)$ be the character associated with $\pi_f$. Then the {\em $\Pi$-isotypic component} of $H^\ast(X, \cal E^\bullet)$ for $\bullet \in \{\sub, \can\}$ is:
	$$H^\ast(X, \cal E^\bullet)_\Pi = \{\eta \in H^\ast(X, \cal E^\bullet) \ | \ T \eta = \chi(T) \eta \text{ for all }T \in \cal H^{K_f}  \}.$$
	Assuming the elements of $\Pi$ are cuspidal, we consider:
	$$H^\ast(X, \cal E)_\Pi = \text{Im}( H^\ast(X, \cal E^{\text{sub}})_{\Pi} \to H^\ast(X, \cal E^{\can})_{\Pi}).$$
\end{definition}

We will use the same normalization of contributions to cohomology as~\cite{LPSZ}, i.e.\ we set $m = k_1 + k_2 - 6$ and assume that $\pi_f \otimes \| \nu \|^{m/2}$ is unitary.

The next goal is to compute $H^\ast(X, \cal E)_{\Pi}$ for the $L$-packets $\Pi$ described in Lemma~\ref{lemma:multiplicities}.

\begin{theorem}[{\cite[Theorem 5.2]{LPSZ}}]\label{thm:reps_in_coherent_cohomology}
	Let $\pi$ be a cuspidal, non-CAP automorphic representation of $\GSp_4( \A)$ whose component at $\infty$ is a (limit of) discrete series representation associated to the parameter $(\lambda; m)$. Let $\Pi$ be the $L$-packet associated with $\pi_f$.  
	
	Consider the four vector bundles associated to the following four representations:
	\begin{align}
		\cal E_0 & = \cal E_{(k_1, k_2; m)} = \cal E_{(\lambda_1 + 1, \lambda_2 + 2;m)},  & V_0 & = V_{k_1, k_2;m},  \label{eqn:E0} \\
		\cal E_1 & = \cal E_{(k_1, 4-k_2; m)} = \cal E_{(\lambda_1 + 1, -\lambda_2 + 2; m)}, & V_1 & = V_{k_1, 4-k_2;m}, \label{eqn:E1}  \\
		\cal E_2 & = \cal E_{(k_2 - 1, 3 - k_1;m)}  = \cal E_{(\lambda_2 + 1, -\lambda_1 + 2; m)}, & V_2 & = V_{k_2-1, 3-k_1;m}, \label{eqn:E2} \\
		\cal E_3 & = \cal E_{(3 - k_2, 3 - k_1;m)} = \cal E_{(-\lambda_2 +1, - \lambda_1 + 2; m)}, & V_3 & = V_{3-k_2, 3-k_1;m}.  \label{eqn:E3}
	\end{align}
	with notation as in equation~(\ref{eqn:Vk1k2}) and $\lambda = (\lambda_1, \lambda_2) = (k_1 - 1, k_2 - 2)$.

	\begin{enumerate}
		\item Suppose first that $\Pi$ is non-endoscopic. Then:
		\begin{align*}
			H^i(X_{\C}, \E_i)_{\Pi} & \iso \bigoplus_{\pi_f' \in \Pi_f}  \Hom_{K^\circ}\left({ \bigwedge}^i \mathfrak p_- \otimes V_i^\vee, X_{\lambda;m}^1\right) \otimes (\pi_f')^{K_f}  & i=0,3, \\
			H^i(X_{\C}, \E_i)_{\Pi} & \iso \bigoplus_{\pi_f' \in \Pi_f}  \Hom_{K^\circ}\left({ \bigwedge}^i \mathfrak p_- \otimes V_i^\vee, X_{\overline \lambda;m}^2\right) \otimes (\pi_f')^{K_f}  & i=1,2.
		\end{align*}
		
		\item Suppose next that $\Pi$ is endoscopic associated to $\sigma$. Using the notation of Lemma~\ref{lemma:multiplicities}, we have that:
		\begin{align*}
			H^i(X_{\C}, \E_i)_{\Pi} & \iso  \Hom_{K^\circ}\left({ \bigwedge}^i \mathfrak p_- \otimes V_i^\vee, X_{\lambda;m}^1\right) \otimes \bigoplus_{\substack{S \subseteq S(\sigma) \\ |S| \text{ odd}}} \pi_{S, f}^{K_f}  & i=0,3, \\
			H^i(X_{\C}, \E_i)_{\Pi} & \iso  \Hom_{K^\circ}\left({ \bigwedge}^i \mathfrak p_- \otimes V_i^\vee, X_{\overline \lambda;m}^2\right) \otimes \bigoplus_{\substack{S \subseteq S(\sigma)\\ |S| \text{ even}}} \pi_{S, f}^{K_f}  & i=1,2.
		\end{align*}
		In particular,  in this case,  a finite representation $\pi_f \in \Pi_f$ contributes to either $H^0$ and $H^3$ or $H^1$ and $H^2$ but not both.
	\end{enumerate}
	
	Finally, each of the spaces $\Hom_{K^\circ}(\bigwedge^i \mathfrak p_- \otimes V_i^\vee, X_{\mu;m}^j)$ above is one-dimensional. Moreover, for all $i$ and all automorphic sheaves $\mathcal E = \mathcal E_\varrho^\can$, we have:
	$$H^i(X_{\C}, \cal E)_\Pi = 0 \qquad \text{if $\E \neq \E_i$}.$$
\end{theorem}

Figure~\ref{fig:coherent cohomology} gives a graphical interpretation of these result.

\begin{figure}[h]
	\begin{tikzpicture}
		\fill[fill=blue!10]    (0,0) -- (4.5, 0) -- (4.5, 4.5) -- (0,0);
		
		\fill[fill=green!10]    (0,0) -- (4.5, 0) -- (4.5, -4.5) -- (0,0);
		
		\fill[fill=darkgreen!10]    (0,0) -- (4.5, -4.5) -- (0, -4.5) -- (0,0);
		
		\fill[fill=darkblue!10]    (0,0) -- (0, -4.5) -- (-4.5, -4.5) -- (0,0);
		

		
		
		
		\foreach \x in {-4, -3, ..., 4} {
			\foreach \y in {-4, -3, ..., 4} {
				\fill (\x, \y) circle(1pt);
		} }
		
		\draw[very thick, blue, ->] (0,0) -- (2,0);
		\draw[very thick, blue, ->] (0,0) -- (-2,0);
		\draw[very thick, blue, ->] (0,0) -- (0,2);
		\draw[very thick, blue, ->] (0,0) -- (0,-2);

		\draw[very thick, blue, ->] (0,0) -- (1,1);
		\draw[very thick, blue, ->] (0,0) -- (-1,-1);
		\draw[very thick, blue, ->] (0,0) -- (-1,1);
		\draw[very thick, blue, ->] (0,0) -- (1,-1);
		
		
		
		\fill[darkred, very thick] (3,2) circle(3pt) node[left] {$\lambda$};
		\fill[darkred, very thick] (-2,-3) circle(3pt) node[above] {$\lambda'$};
		
		\fill[darkgreen, very thick] (4,4) circle(3pt) node[right] {$\Lambda$};
		\fill[darkgreen, very thick] (-4,-4) circle(3pt) node[below] {$\Lambda'$};
		
		\draw[darkgreen, dotted, thick] (3,2) -- (4,2) -- (4,4);
		\draw[darkgreen, dotted, thick] (-2,-3) -- (-2,-4) -- (-4,-4);
		
		\draw[darkblue, very thick] (4,4) circle(4pt) node[left] {$\cal E_0$};
		\draw[darkblue, thick, dashed]  (3,2) -- (4,2) -- (4,4);
		
		\draw[darkblue, very thick] (-1,-1) circle(4pt) node[left] {$\cal E_3$};
		\draw[darkblue, thick, dashed] (-2, -3) -- (-1, -3) -- (-1, -1);
		
		\draw[darkred, <->] (2.8,1.6) .. controls (0.8, -0.8) .. (-1.6,-2.8);
		
		\fill[darkred, very thick] (3,-2) circle(3pt) node[left] {$\overline \lambda$};
		\fill[darkred, very thick] (2,-3) circle(3pt) node[above] {$\overline \lambda'$};
		
		\fill[darkgreen, very thick] (4,-2) circle(3pt) node[right] {$\overline \Lambda$};
		\fill[darkgreen, very thick] (2,-4) circle(3pt) node[below] {$\overline \Lambda'$};
		
		\draw[darkgreen, dotted, thick] (3,-2) -- (4,-2);
		\draw[darkgreen, dotted, thick] (2,-3) -- (2,-4);
		
		\draw[darkblue, very thick] (4,0) circle(4pt) node[left] {$\cal E_1$};
		\draw[darkblue, thick, dashed]  (3,-2) -- (4,-2) -- (4,0);
		
		\draw[darkblue, very thick] (3,-1) circle(4pt) node[left] {$\cal E_2$};
		\draw[darkblue, thick, dashed]  (2,-3) -- (3,-3) -- (3,-1);
		
		\draw[darkred, <->] (2.9,-2.3) .. controls (2.7, -2.7) .. (2.3,-2.9);
		
		\draw (4.5, 2) node[right] {$H^0(\cal E_0)$};
		
		\draw (4.5, 0) node[right] {$H^0(\cal E_{0,1})$ and $H^1(\cal E_{0,1})$};
		\draw (4.8, -0.5) node[right] {$\cal E_{0,1} = \cal E_0 = \cal E_1$};
		
		\draw (4.5, -2) node[right] {$H^1(\cal E_1)$};
		
		\draw (4, -4.8) node[right] {$H^1(\cal E_{1,2})$ and $H^2(\cal E_{1,2})$};
		\draw (4.3, -5.3) node[right] {$\cal E_{1,2} = \cal E_1 = \cal E_2$};
		
		\draw (2.5, -4.5) node[below] {$H^2(\cal E_2)$};
		
		\draw (0, -4.5) node[below] {$H^2(\cal E_{2,3})$};
		\draw (0, -5) node[below] {and};
		\draw (0, -5.5) node[below] {$H^3(\cal E_{2,3})$};
		\draw (0, -6) node[below] {$\cal E_{2,3} = \cal E_2 = \cal E_3$};
		
		\draw (-2.5, -4.5) node[below] {$H^3(\cal E_3)$};
	\end{tikzpicture}
	\caption{Contributions to coherent cohomology according to the Harish--Chandra parameter $\lambda$. In this example, $\lambda = (3, 2)$. The Harish--Chandra $\lambda$ parameters are labeled by red dots, the Blattner parameters $\Lambda$, given  by Table~\ref{table:min_K-types}, are labeled by green dots, and the $K$-types for the vector bundles $\cal E_i$ are labeled by blue circles. Serre duality corresponds to reflection about the line $\lambda_1 = -\lambda_2$.}
	\label{fig:coherent cohomology}
\end{figure}
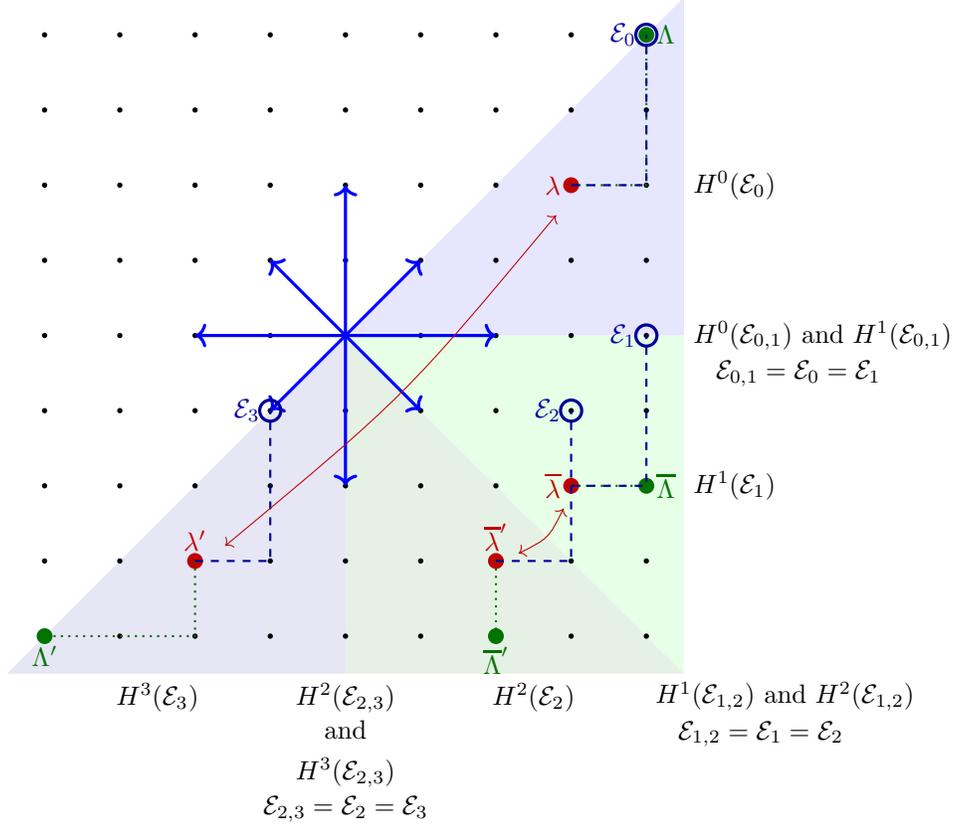

Before giving a proof of the theorem, we give a lemma. Recall that $(\frak P^\circ, K^\circ)$-cohomology is the cohomology of the complex with $j$-th term
$$\Hom_{K^\circ} \left({\bigwedge}^j \mathfrak p_- \otimes V_\varrho^\vee, \cal A(G)^{K_f}  \right).$$
We observe that this complex degenerates as follows.

\begin{lemma}\label{lemma:PK}
	For $k_1 \geq k_2 \geq 2$ with $m \equiv k_1 + k_2 \mod 2$, let $(\lambda_1, \lambda_2) = (k_1 -1, k_2 - 2)$ and consider $V_0, V_1, V_2, V_3$ as in equations~\eqref{eqn:E0}--\eqref{eqn:E3}. Then:
	\begin{align*}
		\dim \Hom_{K^\circ}\left({\bigwedge}^j \mathfrak p_- \otimes V_i^\vee,  X_{\lambda;m}^1\right) & = \begin{cases}
			1 & \text{if }i = j = 0 \text{ or }i = j = 3, \\
			0 & \text{otherwise},
		\end{cases} \\
		\dim \Hom_{K^\circ}\left({\bigwedge}^j \mathfrak p_- \otimes V_i^\vee,  X_{\overline \lambda; m}^2 \right) & = \begin{cases}
			1 & \text{if }i = j = 1\text{ or }i = j = 2, \\
			0 & \text{otherwise}.
		\end{cases}
	\end{align*}
\end{lemma}
\begin{proof}
	Recall that $\mathfrak p_- \iso \Sym^2(\mathrm{St})$ as a $K^\circ$-representation (c.f.~\cite[pp.\ 85]{Harris:special_values}). By Remark~\ref{rmk:highest_wt_flipped}, the highest weight in the $K^\circ$-representation $V_{r_1, r_2;m}^\vee|_{K^\circ}$ is $(r_1, r_2)$. Therefore, the highest weight of ${\bigwedge}^j \mathfrak p_-$ is $(0,0)$ when $j = 0$, $(0, -2)$ when $j = 1$, $(-1, -3)$ when $j = 2$, $(-3, -3)$ when $j = 3$. 
	
	The lemma then follows by recalling the $K^\circ$-types occurring in $X_\lambda^1$ and $X_{\overline \lambda}^2$, as described in \cite[{Sec.\ 2.2}]{Schmidt} (see also \cite[Lemma 6.1]{Muic} and Figure~\ref{fig:K-types}). This is summarized in Figure~\ref{fig:proof_of_Lemma}.
\end{proof}

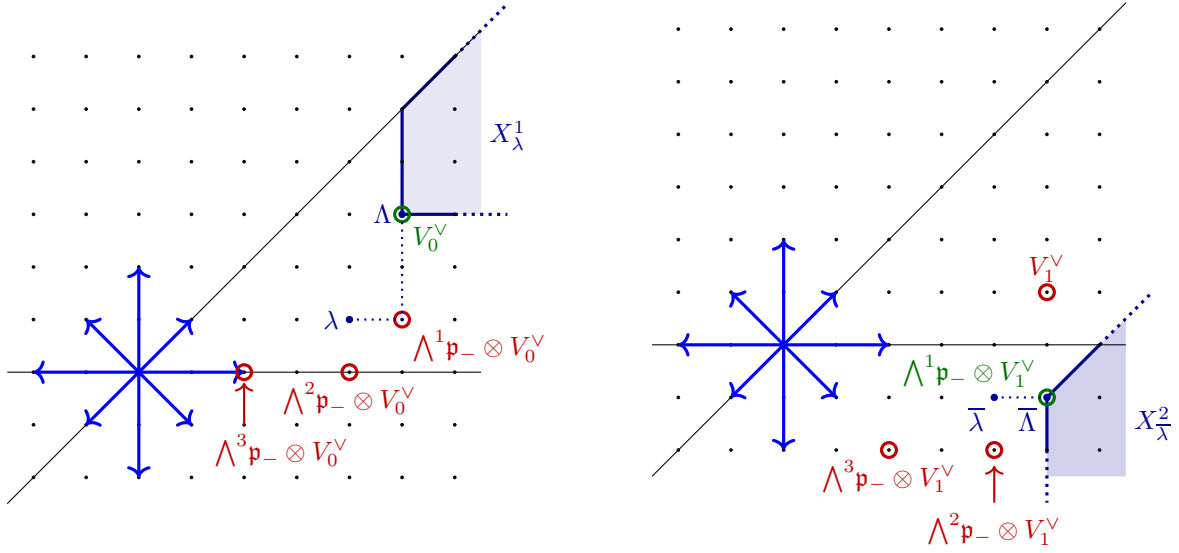
\begin{figure}[h]
	\begin{multicols}{2}
		\begin{tikzpicture}[scale = 0.7]
			\fill[fill=darkblue!10]    (6.5, 3) -- (5,3) -- (5, 5) -- (6.5, 6.5);
			
			
			\draw[darkblue, very thick]    (6, 3) -- (5,3) -- (5, 5) -- (6, 6);
			
			\draw[darkblue, dotted, very thick] (6, 6) -- (7, 7);
			\draw[darkblue, dotted, very thick] (6, 3) -- (7, 3);
			
			
			%
			
			
			
			\draw (-2.5, -2.5) -- (6.5, 6.5);
			
			\draw (-2.5, 0) -- (6.5, 0);
			
			\foreach \x in {-2, -1, ..., 6} {
				\foreach \y in {-2, -1, ..., 6} {
					\fill (\x, \y) circle(1pt);
			} }
			
			\draw[very thick, blue, ->] (0,0) -- (2,0);
			\draw[very thick, blue, ->] (0,0) -- (-2,0);
			\draw[very thick, blue, ->] (0,0) -- (0,2);
			\draw[very thick, blue, ->] (0,0) -- (0,-2);

			\draw[very thick, blue, ->] (0,0) -- (1,1);
			\draw[very thick, blue, ->] (0,0) -- (-1,-1);
			\draw[very thick, blue, ->] (0,0) -- (-1,1);
			\draw[very thick, blue, ->] (0,0) -- (1,-1);
			
			
			
			\fill[darkblue, very thick] (4,1) circle(2pt) node[left] {$\lambda$};
			
			\fill[darkblue, very thick] (5,3) circle(2pt) node[left] {$\Lambda$};
			
			\draw[darkgreen, very thick] (5,3) circle(4pt) node[below right] {$V_0^\vee$};
			\draw[darkred, very thick] (5,1) circle(4pt) node[below right] {${\bigwedge}^1\mathfrak p_- \otimes V_0^\vee$};
			\draw[darkred, very thick] (4,0) circle(4pt) node[below] {${\bigwedge}^2\mathfrak p_- \otimes V_0^\vee$};
			\draw[darkred, very thick] (2,0) circle(4pt);
			\draw[darkred, thick, ->] (2, -1) node[below] {\hspace{1cm}${\bigwedge}^3\mathfrak p_- \otimes V_0^\vee$} -- (2,-0.2);
			
			\draw[darkblue, dotted, thick] (4,1) -- (5,1) -- (5,3);
			
			
			
			
			\draw (7, 4.5) node[darkblue] {$X_\lambda^1$};
			
			
			
			\draw (0, -4) node {};	
		\end{tikzpicture}
		
		\begin{tikzpicture}[scale = 0.7]
			
			
			
			
			\fill[fill=darkblue!60, opacity = 0.3]  (5,-2.5) -- (5, -1) -- (6.5, 0.5) -- (6.5, -2.5);
			
			\draw[darkblue, very thick]  (5,-2) -- (5, -1) -- (6, 0);
			\draw[darkblue, dotted, very thick]  (6, 0) -- (7, 1);
			\draw[darkblue, dotted, very thick]  (5, -2) -- (5, -3);
			
			
			
			\draw (-2.5, -2.5) -- (6.5, 6.5);
			
			\draw (-2.5, 0) -- (6.5, 0);
			
			\foreach \x in {-2, -1, ..., 6} {
				\foreach \y in {-2, -1, ..., 6} {
					\fill (\x, \y) circle(1pt);
			} }
			
			\draw[very thick, blue, ->] (0,0) -- (2,0);
			\draw[very thick, blue, ->] (0,0) -- (-2,0);
			\draw[very thick, blue, ->] (0,0) -- (0,2);
			\draw[very thick, blue, ->] (0,0) -- (0,-2);

			\draw[very thick, blue, ->] (0,0) -- (1,1);
			\draw[very thick, blue, ->] (0,0) -- (-1,-1);
			\draw[very thick, blue, ->] (0,0) -- (-1,1);
			\draw[very thick, blue, ->] (0,0) -- (1,-1);
			
			
			
			
			
			\draw[darkred, very thick] (5,1) circle(4pt) node[above] {$V_1^\vee$};
			\draw[darkgreen, very thick] (5,-1) circle(4pt) node[above left] {${\bigwedge}^1\mathfrak p_- \otimes V_1^\vee$};
			\draw[darkred, very thick] (4,-2) circle(4pt);
			\draw[darkred, very thick] (4,-3) node[below] {${\bigwedge}^2\mathfrak p_- \otimes V_1^\vee$};
			\draw[darkred, thick, ->] (4,-3) -- (4,-2.3);
			\draw[darkred, very thick] (2,-2) circle(4pt) node[below] {${\bigwedge}^3\mathfrak p_- \otimes V_1^\vee$}; (2,-0.2);
			
			
			\fill[darkblue, very thick] (4,-1) circle(2pt) node[below left] {$\overline \lambda$};
			
			\fill[darkblue, very thick] (5,-1) circle(2pt) node[below left] {$\overline \Lambda$};
			
			\draw[darkblue, dotted, thick] (4,-1) -- (5,-1);
			
			
			\draw (7, -1.5) node[darkblue] {$X_{\overline \lambda}^2$};
			
			
		\end{tikzpicture}
	\end{multicols}
	
	
	\caption{This diagram indicates the proof of Lemma~\ref{lemma:PK} for the representation $V_0$ on the left hand side and $V_1$ on the right hand side when $\lambda = (4,1)$. The shaded regions represent the $K^\circ$-types occurring in $X_\lambda^1$ and $X_{\overline \lambda}^2$ and their minimal $K^\circ$-types are indicated by $\Lambda$ and $\overline \Lambda$. We also label the representations $\bigwedge^j \mathfrak P \otimes V_i^\vee$ for $j = 0, 1, 2, 3$ and $i = 0, 1$. One obtains similar diagrams for $i = 2,3$ by reflecting over the $\lambda_2 = -\lambda_1$ axis.}
	\label{fig:proof_of_Lemma}
\end{figure}

\begin{proof}[Proof of Theorem~\ref{thm:reps_in_coherent_cohomology}]
	The proof is a computation based on Theorem~\ref{thm:Su}. Taking $\Pi$-isotypic components of both sides of~\eqref{eqn:Su}, we obtain:
	\begin{align*}
		H^i(X_{\C}, \cal E_\varrho^\can)_{\Pi} & = \bigoplus_{\pi_f \otimes \pi_\infty \subseteq \cal A(G)} H^i(\frak P^\circ, K^\circ; \pi_\infty \otimes V_\varrho) \otimes \pi_f^{K_f} \\
		& = \bigoplus_{\pi_f \in \Pi_f} \left( (H^i(\frak P^\circ, K^\circ; X_{\lambda;m}^1 \otimes V_\varrho) \otimes \pi_f^{K_f})^{\oplus m_1(\pi_f)} \oplus (H^i(\frak P^\circ, K^\circ; X_{\overline \lambda; m}^2 \otimes V_\varrho) \otimes \pi_f^{K_f})^{\oplus m_2(\pi_f)}\right) \label{eqn:holo_gen_contributions}
	\end{align*}
	where constants $m_1(\pi_f)$ and $m_2(\pi_f)$ are appropriate multiplicities in the automorphic spectrum. Then the multiplicity results of Lemma~\ref{lemma:multiplicities} together with  Lemma~\ref{lemma:PK} complete the proof.
\end{proof}

The theorem allows us to discuss the contributions to cohomology of vectors in automorphic representations. We fix once and for all the standard bases of $\bigwedge^i \mathfrak p_-$ and the representations $V_i$ as in \cite[\S1.2]{Moriyama}.  This gives a choice of a highest weight vector in $\bigwedge^i \mathfrak p_- \otimes V_i^\vee$ which we denote by $v_i$.

\begin{definition}\label{def:[f]}
	Let $\pi$ be a non-CAP cuspidal automorphic representation of $G$ such that $\pi_\infty$ is a (non-degenerate limit of) discrete series representation. Let $f = \bigotimes f_v \in \pi$ be a factorizable vector such that:
	\begin{enumerate}
		\item $f_v \in \pi_v^{K_{f, v}}$ for each finite $v$,
		\item $f_\infty$ is a highest weight vector in a minimal $K$-type of $\pi_\infty$.
	\end{enumerate} 
	The {\em contribution of $f$ to cohomology} is the associated element $[f] \in H^i(X_{\C}, \E_i)_\Pi$ defined by the isomorphism from Theorem~\ref{thm:reps_in_coherent_cohomology}:
	\begin{align*}
		\Hom_{K^\circ}\left(\bigwedge^i \mathfrak p_- \otimes V_i^\vee, \pi_\infty \right) \otimes \pi_f^{K_f} & \overset\iso\to H^i(X_{\C}, \E_i)_\Pi \\
		\left( [v_i \mapsto f_\infty] \otimes \left[\bigotimes\limits_{v < \infty} f_v \right ] \right) & \mapsto [f],
	\end{align*}
	where $i$ is determined by $f_\infty$ and $\pi_\infty$.
\end{definition}

\begin{remark}
	In other words, a vector in an automorphic representation will give a contribution to cohomology once we choose an automorphic embedding of the representation. The difficulty of defining a motivic action for coherent cohomology is that the contributions in different degrees are associated to vectors in different elements of the $L$-packet and at the moment,  there seems to be no canonical way to rigidify the choice of automorphic embedding across an $L$-packet. 
	
	The only exception seems to be Hilbert modular forms where partial complex conjugation operators give a way to go between different elements of the $L$-packet; this idea goes back to Shimura~\cite{Shimura_HMF} and Harris~\cite{Harris:aut_forms_and_coh_of_vb, Harris_periods_I}. This gives a way to define the motivic action in this case~\cite{Horawa}.
\end{remark}

We record the parameters $\lambda$ which can lead to packets $\Pi$ making contributions to multiple degrees of cohomology of the same sheaf.

\begin{corollary}\label{cor:deg_of_cohomology}
	If $(k_1, k_2) = (k, 2)$ for $k \geq 2$, i.e.\ $\lambda = (k-1, 0)$, then
	$$H^i(X, \cal E_{(k,2; m)})_\Pi \neq 0 \text{ for } i = 0, 1,$$
	and
	$$H^i(X, \cal E_{(1, 3-k; m)})_\Pi \neq 0 \text{ for } i = 2, 3.$$
\end{corollary}
The goal of the paper is to study the rationality of these multiple contributions from the point of view of algebraic cycles, using a certain motivic cohomology group. 

\begin{remark}
	There is also a different class of automorphic representations for which there is a different degeneracy of contributions to coherent cohomology. Consider a Harish--Chandra parameter $\lambda = (p, -p)$ and the associated limit of discrete series $X_{\lambda;m}^\times$ for some $m$, described in Appendix~\ref{app:GSp(4,R)}. Then the archimedean $L$-packet of $X_{\lambda; m}^\times$ is the singleton $\{X_{\lambda;m}^\times\}$. 
	
	Let $\pi$ be a non-CAP cuspidal automorphic representation of $\GSp_4(\A)$ whose component at $\infty$ is the limit of discrete series $X_{\lambda; m}^\times$. Then the $L$-packet of $\Pi$ of $\pi$ contains no representations whose components at infinity are holomorphic (limits of) discrete series, i.e.\ the $L$-packet is not associated with any holomorphic Siegel modular form.
	
	In this case, all the contributions of $\Pi$ to cohomology are given by:
	$$H^i(X, \cal E_{p+1,2-p; m})_{\Pi} \neq 0 \text{ for } i = 1,2.$$	
	It would be interesting to understand if there is a motivic action $H^1(X, \cal E_{p+1,2-p; m})_{\Pi} \to H^2(X, \cal E_{p+1,2-p; m})_{\Pi}$ in this case, but we decided not to pursue this here.
\end{remark}

As in~\cite[(5.1)]{LPSZ}, we may use Serre duality~\eqref{eqn:SD} to define a $\GSp_4(\A_f)$-equivariant pairing:
\begin{equation*}
	\langle -, - \rangle_{\mathrm{SD}} \colon H^{i}(X, \E_{i}) \times H^{3-i}(X, \E_{3-i}) \to \Q\{m\},
\end{equation*}
where $\Q\{m\}$ is the $m$th power of the similitude character of $\GSp_4(\A_f)$. Explicitly:
\begin{align}
	\langle [f], [f'] \rangle_{\SD} & = \int_{Z_G(\A) G(\Q) \backslash G(\A)/ K_\infty K}  f(g) f'(g) \, dg \sim_{\Q^\times} \frac{1}{\pi^3} \langle f, f' \rangle, \label{eqn:SD}
\end{align}
where we use the automorphic normalization of Petersson norm
\begin{equation*}
	\langle f, f' \rangle = \int_{Z_G(\A) G(\Q) \backslash G(\A)} f(g) f'(g) \, dg,
\end{equation*}
and $dg$ is the Tamagawa measure as in \cite{Chen_Ichino}.

\section{Motives and Beilinson's conjecture}\label{sec:motives}


We give a brief summary of motives, motivic cohomology, Deligne's conjecture, and Beilinson's conjecture. We then discuss several examples directly related to our paper. We follow the same conventions as~\cite[Section 2]{Prasanna_Venkatesh} and \cite{jannsen1988withdrawn} and the reader is encouraged to consult these references for details.

Unless specified otherwise, we will always use the word {\em motive} to mean a Chow motive over $\Q$, as in~\cite[2.1.2]{Prasanna_Venkatesh}. Associated to a Chow motive $M$ and any $i$, we shall consider the associated $L$-function $L^i(M,s)$ which obtained by considering the action of Frobenius on the \'etale realizations $H^i_{\text{\'et}}(M_{\overline \Q})$. To ensure that the motive $M$ is determined by its $L$-function, we make two assumptions:
\begin{enumerate}
	\item the Tate conjecture, which implies that a pure Grothendieck motive of weight $i$ is determined by its $L$-function,
	\item Beilinson's filtration conjectures~\cite[\S2.1.10]{Prasanna_Venkatesh}, which imply that the Chow motive is determined up to isomorphism by its associated Grothendieck motive.
\end{enumerate}
These assumptions seem inevitable to give a statement of the motivic action conjectures; see \cite[\S2.1.10, \S2.1.11, Remark 9]{Prasanna_Venkatesh} for further discussion.

For any subring $R \subseteq \C$, we write $R(j) = (2 \pi i)^j R \subseteq \C$. We have a natural isomorphism $\C \iso \R(n) \oplus \R(n-1)$ and we write $\pi_{n-1} \colon \C \to \R(n-1)$ for the projection to $\R(n-1)$.

Let us assume for simplicity that $n \geq \frac{i}{2} + 1$. For a motive $M$ over $\Q$, we will be interested in the Beilinson short exact sequences~\cite[(2.1.10), (2.10.11)]{Prasanna_Venkatesh}:
\begin{equation}
	0 \to F^n H^i_{\rm dR} (M_\R) \overset{\widetilde \pi_{n-1}} \to H^i_B(M_\R, \R(n-1)) \to H^{i+1}_{\cal D} (M_\R, \R(n)) \to 0, \label{eqn:Beilinson_ses}
\end{equation}
\begin{equation}
	0  \to H_B^i(M_\R, \R(n)) \to H^i_{\dR}(M_\R)/F^n H^i_{\dR}(M_\R)  \to H^{i+1}_{\cal D} (M_\R, \R(n)) \to 0. \label{eqn:Beilinson_ses2}
\end{equation}

\begin{definition}\label{def:Hodge_RS}
	We define two {\em Hodge rational structures} on $\det H^{i+1}_{\cal D} (M_\R, \R(n))$:
	\begin{itemize}
		\item  $\cal R(M, i, n)$ is obtained from the short exact sequence~\eqref{eqn:Beilinson_ses} and the rational structures  $F^n H_{{\rm dR}}^i(M)$ and $H_B^i (M_\R, \Q(n-1))$,
		\item $\cal{DR}(M, i, n)$ is obtained from the short exact sequence~\eqref{eqn:Beilinson_ses2} and the rational structures $H_B^i(M_\R, \Q(n))$ and $H_{{\rm dR}}^i(M)/F^n H_{{\rm dR}}^i(M)$.
	\end{itemize}
\end{definition}

The two rational structures are related by the equations:
\begin{align*}
	\mathcal{DR}(M, i, n) & = (2 \pi i)^{-d^-(M, i, n)} \cdot \delta(M, i, n) \cdot \mathcal R(M, i, n), \\
	d^-(M, i, n) & = \dim H_B^i(M_\C, \Q(n))^-, \\
	\delta(M, i, n) & = \det \left( H_B^i(M_\C, \Q(n)) \otimes \C \overset\iso\to H_{\dR}^i(M, n) \otimes \C \right).
\end{align*}


Beilinson~\cite{Beilinson} defined a regulator map:
\begin{equation}\label{eqn:Beilinson_reg}
	r_{\cal D} \colon H^i_{\cal M} (M, \Q(j)) \otimes \R \to H^i_{\cal D}(M_\R, \R(j))
\end{equation}
which gives a different rational structure
$$\det r_{\cal D}( H^{i+1}_{\cal M} (M, \Q(n))) \subseteq \det H^{i+1}_{\cal D} (M_\R, \R(n)).$$

Scholl~\cite{scholl2000integral} showed that there is a unique $\Q$-subspace $H^{i+1}_{\mathcal M}(M_\Z, \Q(n)) \subseteq H^{i+1}_{\cal M} (M, \Q(j))$ such that if $X$ is a variety which has a regular integral model $\mathcal X$, then:
$$H^{i+1}_{\mathcal M}(h(X)_\Z, \Q(n)) = \mathrm{Im} \left( H^{i+1}_{\mathcal M}(h(\mathcal X), \Q(n)) \to H^{i+1}_{\mathcal M}(h(X), \Q(n) ) \right).$$

Beilinson's conjecture predicts that the difference between the two rational structures is given by an appropriate $L$-value.

\begin{conjecture}[Beilinson]\label{conj:Beilinson}
	Suppose $n \geq \frac{i}{2} + 1$ and if $n = \frac{i}{2} + 1$, that $H_{\textnormal{\'et}}^{i} (M_{\overline \Q}, \Q_\ell(i/2))^{G_{\Q}} = 0$. Then: 
	\begin{enumerate}[itemsep = 0.3cm]
		\item the Beilinson regulator $r_{\cal D} \colon H^{i+1}_{\cal M} (M_\Z, \Q(n)) \otimes \R \to H^{i+1}_{\cal D}(M_\R, \R(n))$
		is an isomorphism,
		\item we have, equivalently:
		\begin{align*}
			r_{\cal D} (\det H^{i+1}_\M(M_\Z, \Q(n))) & = L^{-i}(M^\vee, 1-n)^\ast \cdot \mathcal R(M, i, n), \\
			r_{\cal D} (\det H^{i+1}_\M(M_\Z, \Q(n))) & = L^{i}(M, n) \cdot \mathcal{DR}(M, i, n).
		\end{align*}
	\end{enumerate}
\end{conjecture}

We make the following hypothesis which we keep throughout the paper.

\begin{hypothesis}\label{hyp:reg_is_isom}
	We assume that Beilinson's regulator is an isomorphism, i.e.\ part (1) of Beilinson's Conjecture~\ref{conj:Beilinson}.
\end{hypothesis}

Deligne~\cite{Deligne_Special_values} defined the notion of a {\em critical value} $j \in \Z$ and periods $c^\pm(H^i(M(j))) \in \C^\times/\Q^\times$ such that conjecturally $L^i(M, j) \sim_\Q c^+(H^i(M)(j))$. The $L$-values $L^i(M, n)$ and $L^{-i}(M^\vee, 1-n)^\ast$ are critical if and only if $H^{i+1}_{\mathcal D}(M_\R, \R(n)) = 0$. In this case, the first map in the short exact sequence~\eqref{eqn:Beilinson_ses} is an isomorphism and we can express Deligne's period in terms of its determinant:
\begin{align}\label{eqn:Deligne_period_general}
	c^+(H^i(M)( i+1-n )) = \det (F^n H^i_{\mathrm dR}(M_\R) \overset{\widetilde \pi_{n-1}}\to H^i_B(M_\R, \R(n-1))  ).
\end{align}

\subsection{Motives associated to Siegel modular forms}

Let $f$ be a holomorphic cuspidal Siegel modular form of weight $(k_1, k_2)$ for $k_2 \geq k_2 \geq 2$. For simplicity, assume that the central character of $f$ is trivial and the $f$ has coefficients in $\Q$. Recall that $(k_1, k_2) = (\lambda_1 + 1, \lambda_2 + 2)$, where $(\lambda_1, \lambda_2)$ is the Harish--Chandra parameter for the holomorphic (limit of) discrete series representation.

\subsubsection{The motive $M(f)$}\label{sssec:M(f)}

Conjecturally,  there exists a pure Chow motive $M(f)$ of rank 4 and weight 
$$w = \lambda_1 + \lambda_2 = k_1 + k_2 - 3$$ 
associated with $f$. Its Betti realization has a Hodge decomposition:
\begin{equation}\label{eqn:HodgeM(f)}
	H_B^{\lambda_1 + \lambda_2}(M(f)) \otimes \C \iso H^{\lambda_1 + \lambda_2, 0} \oplus H^{\lambda_1, \lambda_2} \oplus H^{\lambda_2, \lambda_1} \oplus H^{0, \lambda_1 + \lambda_2}.
\end{equation}
When $\lambda_2 > 0$, the associated Galois representation can be found in the \'etale cohomology of the Siegel threefold, and hence one can construct a Grothendieck motive associated to $f$ this way.  However, for $\lambda_2 = 0$, the form $f$ is not cohomological and hence the motive cannot be constructed this way. In fact, the Galois representation is constructed using congruences with higher weight forms~\cite{Taylor}.

Let $\psi$ be a Hecke character of finite order and consider the spin $L$-function $L(f, \psi, s) = L(M(f)(\psi), s)$.

\begin{lemma}\label{lemma:Beil_for_spin}
	The critical values of $L(f, \psi, s)$ are:
	$$\{n \in \Z \ | \ \lambda_2 + 1 \leq n \leq \lambda_1\}.$$
	In particular, since $\lambda_1 > \lambda_2$, $L(f, \chi, s)$ always has a critical value. For these $n$, according to Deligne's conjecture:
	$$\frac{L(f, \psi, n)}{(2\pi i)^{2n} g(\psi)^2 c^\pm(H^w(M(f)))} \in \Q(f, \psi)$$
	where $\pm  1 = (-1)^{n+e(\psi_\infty)}$ and $e(\psi_\infty) = 0$ if $\psi_\infty$ is trivial and 1 otherwise.
\end{lemma}
\begin{proof}
	This is a standard computation with $\Gamma$-factors which can be described using the Hodge decomposition~\eqref{eqn:HodgeM(f)}. For example, see \cite{Harris:occult}.
\end{proof}

The Hodge filtration on the de Rham realization $H_{\mathrm{dR}}^w(M(f))$ has four steps:
\begin{equation*}
	H_{\mathrm{dR}}^w(M(f)) = F^{0} \supseteq F^{\lambda_2} \supseteq F^{\lambda_1} \supseteq F^{\lambda_1 + \lambda_2}.
\end{equation*}
These containments are all strict unless $\lambda_2 = 0$.

We choose a basis $\{\omega_i\}$ for $H_{\mathrm{dR}}^{w}(M(f))$ compatible with the filtration, i.e.\
\begin{align}
	F^{\lambda_1 + \lambda_2} & = {\rm span} \{\omega_1 \},  \\
	F^{\lambda_1 } & =  {\rm span} \{\omega_1, \omega_2 \},  \label{eqn:Flambda_1}\\
	F^{\lambda_2}  & = {\rm span} \{\omega_1, \omega_2, \omega_3 \}, \\
	F^{0} & = {\rm span} \{\omega_1, \omega_2, \omega_3, \omega_4 \}  \label{eqn:F0}
\end{align}
In the degenerate case $\lambda_2 = 0$, the basis only satisfies equations~\eqref{eqn:Flambda_1} and~\eqref{eqn:F0}.

Moreover, we have a decomposition into 2-dimensional eigenspaces for the action of the involution $F_\infty$:
$$H_B^{w}(M(f)_\C, \Q) = H_B^w(M(f)_\C, \Q)^+ \oplus H_B^w(M(f)_\C, \Q)^-.$$
We choose bases $\{v_1^\pm, v_2^\pm\}$ of $H_B^w(M(f)_\C, \Q)^\pm$. Suppose that under the de Rham-Betti comparison $c_{\mathrm{dR}, B}$:
\begin{equation}\label{eqn:c_DR_B}
	c_{\mathrm{\mathrm{dR}, B}} (\omega_i) = c_{i1}^+ v_1^+ +  c_{i2}^+ v_2^+ +  c_{i1}^- v_1^- +  c_{i2}^- v_2^- \qquad \text{for $i=1,2$ and some $c_{ij}^\pm \in \C$}.
\end{equation}

\begin{lemma}\label{lemma:c+-(M(f))}
	In the notation of equation~\eqref{eqn:c_DR_B}, we have that:
	$$c^{\pm(-1)^{\lambda_2}}(M(f)(\lambda_1)) =  (2 \pi i)^{-2\lambda_2} \det (c_{ij}^{\pm})_{1 \leq i,j \leq 2} = (2 \pi i)^{-2\lambda_2} (c_{11}^\pm c_{22}^\pm - c_{12}^\pm c_{21}^\pm).$$
\end{lemma}
\begin{proof}
	From equation~\eqref{eqn:Deligne_period_general}, we have that:
	$$c^\pm(M(f)(\lambda_1)) = \det(F^{\lambda_2 + 1} H^{w}_{\dR}(M(f))_\R \to H^w_{B}(M(f)_\R, \R( \lambda_2))^\pm ) $$
	and a rational basis of the de Rham cohomology group is $\omega_1, \omega_2$, while a rational basis of the Betti cohomology group is $(2 \pi i)^{\lambda_2} v_i^{\pm}$.
\end{proof}

\subsection{The motives $\Sym^2 M(f)$ and $M(f, \Ad)$}

We refer to \cite[Definition 4.2.1]{Prasanna_Venkatesh}  for the definition of the adjoint motive associated to an automorphic representation $\pi$. In the case where $\pi$ is associated with a Siegel modular form $f$, we define $M(f, \Ad) = \Ad M(\pi)$ to be the conjectural adjoint motive of $\pi$ of rank 10 and pure weight 0.

It turns out that the motive $M(f, \Ad)$ has a simple description in term of the motive $M(f)$ of $f$ described above. We consider the symmetric square motive $\Sym^2 M(f)$ which has rank~10 and weight $2w = 2(\lambda_1 + \lambda_2)$.  Then:
$$L(\Sym^2M(f), s) = L(M(f, \Ad), s - w)$$
and hence $M(f, \Ad) \iso \Sym^2 M(f)(w)$. Therefore, it is enough to consider $H^{2w}(\Sym^2 M(f))$.
The Hodge numbers of $\Sym^2 M(f)$ are:
$$\{\{(p,2w-p), (2w-p, p) \} \ | \ p = 2(\lambda_1 + \lambda_2), 2\lambda_1 + \lambda_2, 2 \lambda_1, \lambda_1 + 2\lambda_2, \lambda_1 + \lambda_2 \}.$$

\begin{lemma}\label{lemma:critical_for_Sym2}
	The critical values for the symmetric square $L$-function $L(\Sym^2 M(f), s)$ are:
	$$\{ n \textnormal{ odd} \ | \ \lambda_1 + 1 \leq n \leq \lambda_1 + \lambda_2   \} \cup \{ n \textnormal{ even} \ | \ \lambda_1 + \lambda_2 + 1 \leq n \leq \lambda_1 + 2\lambda_2\}.$$
	In particular, the $L$-function $L(\Sym^2 M(f), s)$ has a critical value if and only if $\lambda_2 > 0$.
\end{lemma}
\begin{proof}
	The $\Gamma$-factors are:
	$$\Gamma_\C(s)\Gamma_\C(s - \lambda_2) \Gamma_\C(s-2\lambda_2) \Gamma_\C(s - \lambda_1) \Gamma_\R(s + 1 -(\lambda_1 + \lambda_2))^2$$
	and the proof is a standard computation.
\end{proof}

\begin{lemma}[{Yoshida~\cite[(4.15) and (4.16)]{Yoshida}}]
	Suppose $f$ is a holomorphic scalar-valued Siegel modular forms of weight $k \geq 3$, i.e.\ $(\lambda_1, \lambda_2) = (k-1, k-2)$ and $\lambda_2 > 0$ . For simplicity, assume again that $f$ has rational Fourier coefficients. Then:
	\begin{eqnarray}
		c^+(\Sym^2(M(f))) & = (2\pi i)^{12 - 6k} c^+(M(f)) c^-(M(f)) \langle f, f \rangle, \label{eqn:Yoshida1} \\
		c^-(\Sym^2(M(f))) & = (2\pi i)^{6 - 2k} c^+(M(f)) c^-(M(f)) \langle f, f \rangle.\label{eqn:Yoshida2}
	\end{eqnarray}
	Moreover, the Petersson inner products $\langle f, f \rangle$ are explicitly related to a critical value of the standard $L$-function of $f$.
\end{lemma}

A similar formula for general weights $(\lambda_1, \lambda_2)$ with $\lambda_2 > 0$ can be derived from Yoshida's formula~\cite[(3.4)]{Yoshida} via the results of Kozima~\cite{Noritomo}.

When $\lambda_2  = 0$, we will derive an analogous formula for the non-critical $L$-value $L(\Sym^2 M(f), k)$, or equivalently the derivative $L'(\Sym^2 M(f), k-1)$, in Theorem~\ref{thm:Beilinson_for_Sym2} from Beilinson's conjecture. 

\subsection{Motives associated to abelian surfaces}\label{subsec:abelian_surfaces}

Let $A$ be an abelian surface over $\Q$. A precise statement of the modularity conjecture for abelian surfaces was stated by Brumer and Kramer~\cite{Brumer_Kramer}.

\begin{conjecture}[{Brumer--Kramer~\cite{Brumer_Kramer}}]\label{conj:Brumer_Kramer}
	Let $A$ be an abelian surface over $\Q$ of conductor $N$ with $\End_\Q(A) = \Z$. There exists a holomorphic cuspidal Siegel modular form $f$ for $\GSp_4$ with Harish--Chandra parameter $\lambda = (1,0)$ and paramodular level $K(N)$ such that:
	$$L(A, s) = L(f, s).$$
\end{conjecture}

See ~\cite{Brumer_et_al_paramodularity, Calegari_etal_modularity_examples} for evidence for the conjecture.

\begin{remark}
	In this extended remark, we elaborate on this conjecture and explain a few special cases.
	
	First, note that we only stated one direction of the conjecture in~\cite{Brumer_Kramer}. The converse requires two modifications:
	\begin{itemize}
		\item one excludes Siegel modular forms which are {\em Gritsenko lifts} (a variant of the Saito--Kurokawa lift; see loc. cit.\ and the references therein for details)
		\item as pointed out by Frank Calegari, to get a bijection, one also needs to account for abelian fourfolds with quaternionic multiplication: a modified version of the conjecture that account for this was given in~\cite{Brumer_Kramer_Corrigendum}.
	\end{itemize}
	
	Next, we explain how the following examples fit in the framework of the conjecture:
	\begin{enumerate}
		\item $A = E \times E$ for an elliptic curve $E$ over $\Q$, 
		\item $A = E_1 \times E_2$ for two non-isogenous elliptic curves $E_1$ and $E_2$ over $\Q$,
		\item $A = R_{F/\Q} E$  for an elliptic curve $E$ over a quadratic field $F$.  
	\end{enumerate}
	The first example is excluded by the assumption that $\End_\Q(A) = \Z$; indeed, $\End_\Q(E \times E) \supseteq M_2(\Z)$. On the side of Siegel modular forms, one could consider the weight two modular form $f_0$ associated with $E$. Then there is two natural ways to obtain a Siegel modular form from $f_0$ via theta lifting: 
	\begin{itemize}
		\item let $f$ be the Yoshida lift of $f_0 \times f_0$; however, then $f$ is not cuspidal,
		\item let $f$ be the Saito--Kurokawa lift of $f_0$; however, then $f$ is presumably a Gritsenko lift (although we have not verified this).
	\end{itemize}
	Therefore, case (1) is ruled out on both sides of the conjecture.
	
	Case (2) is also ruled out by the assumption that $\End_\Q(A) = \Z$, because $\End(E_1 \times E_2) \supseteq \Z \times \Z$. Nonetheless, one can consider the associated weight two modular forms $f_1$, $f_2$ and the Yoshida lift of $f_1 \times f_2$. As explained in Lemma~\ref{lemma:multiplicities}, one always has the automorphic representation~ $\pi_\emptyset$ of $\GSp_4(\A)$ which is generic at all places and hence paramodular. This means that there is a \textit{generic} paramodular automorphic form with Harish--Chandra parameter $(1,0)$ associated with $f_1$ and $f_2$. 
	
	However, we have excluded this case on the side of Siegel modular forms by requiring the $f$ is holomorphic and paramodular. Indeed, if $S$ is an odd finite set of finite places such that the local representations associated with $f_1$ and $f_2$ are discrete series at these places (such a set may not exist in general), then the associated automorphic representation $\pi_S$ is not generic at $v \in S$ and hence not paramodular. In particular, there is no {\em holomorphic} paramodular Siegel  modular form~$f$ associated with $f_1$ and $f_2$. In fact, this is the reason we assume that the Siegel modular form is non-endoscopic in our work.
	
	Finally, we consider $A = R_{F/\Q} E$. In this case, $E$ corresponds to an automorphic form for the group $\GL_{2, F}$, i.e. a Hilbert or Bianchi modular form. The conjecture has been verified in both cases ~\cite{Johnson_Schmidt, BergerDembelePacettiSengun} using non-endoscopic Yoshida lifts from $\GL_{2, F}$: there is a holomorphic paramodular Siegel modular form~$f$ (as well as a generic one) associated with $A$. Sections~\ref{sec:Yoshida} and~\ref{sec:imag_quad} of the present work are devoted to these special cases.
	
	The key difference between case (2) and case (3) is that the $L$-packet on $\GSp_4$ in the latter case is stable, even though the automorphic representation can be constructed using theta lifting. This is explained by the following observation. For a real quadratic field $F$, there is a quaternion algebra $B$ over $F$ ramified at the two infinite places of $F$ and no finite places. However, when $F = \Q \oplus \Q$, a quaternion algebra $B$ over $F = \Q \oplus \Q$ is $B = B_1 \times B_2$ for quaternion algebras $B_i$ over $\Q$, and hence if we want $B$ to be ramified at the two infinite places, we need $B$ to also be ramified at some finite place.
	
	In summary, one could extend Conjecture~\ref{conj:Brumer_Kramer} as follows:
	\begin{enumerate}
		\item If $A = E_1 \times E_2$ for non-isogenous elliptic curves $E_1$, $E_2$ over $\Q$, then there exists only a \textit{generic} automorphic form with Harish--Chandra parameter $\lambda = (1,0)$ and paramodular level $K(N)$ whose spin $L$-function agrees with $L(A, s)$.
		\item If $A$ is a simple abelian surface with $\End_\Q(A) = \Z$, then there exists a {\em holomorphic} Siegel modular form $f$ (as well as a \textit{generic} one) with Harish--Chandra parameter $\lambda = (1,0)$ and paramodular level $K(N)$ whose spin $L$-function agrees with $L(A, s)$.
	\end{enumerate}
	In fact, case (1) is settled: the desired automorphic form is the generic Yoshida lift of $f_1 \times f_2$. 
\end{remark}

Together with the Tate conjecture,  Conjecture \ref{conj:Brumer_Kramer} implies that the motive $H^1(A)$ is the motive $M(f)$ for an appropriate Siegel modular form $f$ of weight $(2,2)$. While the $L$-function $L(f, s)$ has a critical value in this case~(Lemma~\ref{lemma:c+-(M(f))}), we see that $L(\Sym^2 M(f), s)$ does not~(Lemma~\ref{lemma:critical_for_Sym2}). 

Finally, the constants $c_{i,j}^\pm$ in Lemma~\ref{lemma:c+-(M(f))} may be interpreted in terms of a dual basis under the tautological pairing
\begin{equation*}
	\langle -, - \rangle \colon H^1(A(\C), \Q) \times H_1(A(\C), \Q) \to \Q
\end{equation*}
given by $\langle \omega, \gamma \rangle = \int\limits_\gamma \omega$. We write $\gamma_i^\pm$ for the basis of $H_1(A(\C), \Q)^\pm$ dual to $v_i^\pm$, i.e.\
$$\langle v_i^\pm, \gamma_j^\pm \rangle = \delta_{ij}.$$ 
Then:
\begin{align*}
	c_{i, j}^\pm & = \langle \omega_i, \gamma_j^\pm \rangle = \int\limits_{\gamma_j^\pm}\omega_i.
\end{align*}
Moreover, Poincar\'e duality defines a pairing
\begin{align*}
	\langle -, - \rangle_{\mathrm{PD}} \colon H^1(A(\C), \Q) \times H^3(A(\C), \Q(2))  \to \Q 
\end{align*}
explicitly given after extending coefficients to $\R$ by:
\begin{align*}
	\langle \eta_1, \eta_2 \rangle_{\mathrm{PD}} = \frac{1}{(2\pi i)^2} \int\limits_{A(\C)} \eta_1 \wedge \eta_2.
\end{align*}
If $w_i^\pm$ is a basis of $H^3(A(\C), \Q(2))$ dual to $v_i^{\pm}$, then we also record that:
\begin{align*}
	\langle \omega_i, w_j^\pm \rangle_{\mathrm{PD}}  = \left\langle \sum_{i,k} c_{i,k}^\pm v_k^\pm, w_j^{\pm} \right\rangle_{\mathrm{PD}} = c_{i,j}^\pm.
\end{align*}

\section{Definition of the motivic action}


Let $f$ be a non-endoscopic holomorphic Siegel modular form of weight $(k, 2)$ for $k \geq 2$ even, of paramodular level $K(N)$, with coefficients in $\Q$ and trivial character. The assumption that the character of $f$ is trivial is necessary to use the theory of newforms of~\cite{Roberts_Schmidt}, and this implies that the weight $k$ is even.

For example, if $A$ be an abelian surface over $\Q$ with $\End(A) = \Z$ and conductor $N$, we can let $f$ be the Siegel modular form of weight~$2$ associated with $A$ via the Brumer--Kramer Conjecture~\ref{conj:Brumer_Kramer}. The assumption that~$f$ is non-endoscopic amounts to the assumption that $A$ is simple (i.e not the product of two elliptic curves).


Let $\pi$ be the automorphic representation associated with $f$, $\pi_f$ be its finite component, and $\Pi$ be the associated $L$-packet. Consider the Siegel modular variety $X$ of paramodular level $K(N)$. Then, according to Corollary~\ref{cor:deg_of_cohomology}, the $\Pi$-isotypic component of $H^\ast(X, \cal E_{k,2})$ is:
\begin{equation}\label{eqn:decomp}
	H^\ast(X, \cal E_{k,2})_\Pi \iso H^0(X, \cal E_{k,2})_\Pi \oplus H^1(X, \cal E_{k,2})_\Pi,
\end{equation}
where the dimension of each summand is equal to $\sum\limits_{\pi_f' \in \Pi_f} \dim ((\pi_f')^{K(N)} )$. Moreover, we will see in the next section that $(\pi_f')^{K(N)}$ only for the unique generic element $\pi_f$ of the $L$-packet, and in this case $\dim (\pi_f')^{K(N)} = 1$. Therefore, under our assumptions, each of the summands in equation~\eqref{eqn:decomp} is one-dimension. The goal of this section is to explain this degeneracy in contributions using a motivic action.

\subsection{Theory of local newforms and the Whittaker rational structure on $\pi_f$}\label{subsec:local_newforms}

Roberts and Schmidt~\cite{Roberts_Schmidt} developed a theory of local newforms for $\GSp_4$ for paramodular level structures. We summarize these results in this section.

Let $F$ be a non-archimedean local field of characteristic 0, and let $\p$ be the prime ideal of the ring of integers of $F$. Let $K(\p^n)$ be the paramodular group of level $\p^n$:
\begin{equation}
	K(\p^n) = \begin{pmatrix}
		\O & \p^n  & \O & \O \\
		\O & \O & \O & \p^{-n} \\
		\O & \p^n & \O & \O \\
		\p^n & \p^n & \p^n & \O
	\end{pmatrix}.
\end{equation}

\begin{remark}
	Roberts and Schmidt~\cite{Roberts_Schmidt} use $J = \begin{pmatrix}
		& & & 1 \\
		& & 1 \\
		& -1 \\
		-1
	\end{pmatrix}$ instead of $J = \begin{pmatrix}
		& & 1 \\
		& & & 1 \\
		-1 \\
		& -1
	\end{pmatrix}$ which accounts for the discrepancy in the definitions of the paramodular groups.
\end{remark}

\begin{definition}
	An irreducible, admissible representation $\pi$ of $\GSp(4, F)$ is  {\em paramodular} if $\pi^{K(\p^n)} \neq 0$ for some $n \gg 0$. If $n \geq 0$ is minimal with the property, then $K(\p^n)$ is called the {\em minimal paramodular level}.
\end{definition}

In particular, if $f$ is a Siegel modular form of paramodular level, then by definition any local component $\pi_v$ for finite $v$ of the associated automorphic representation is paramodular.

\begin{theorem}[{Roberts--Schmidt~\cite[Theorem 7.5.1]{Roberts_Schmidt}}]\label{thm:Roberts_Schmidt}
	Let $(\pi, V)$ be an irreducible, admissible representation of $\GSp(4, F)$ with trivial central character. If $\pi$ is paramodular and $K(\p^n)$ is its minimal paramodular level, then $\dim V^{K(\p^n)} = 1$.
\end{theorem}

This allows us to put a rational structure on the representation $\pi_f$ of $G(\A_f)$ associated to a Siegel modular form with paramodular level structure.

\begin{corollary}\label{cor:rational_stru}
	Suppose $(\pi_f, V)$ be an irreducible, admissible representation of  $G(\A_f)$ such that $\pi_f = \bigotimes_v \pi_v$ and each $\pi_v$ is paramodular with trivial central character. Then there is a rational structure on $(\pi_f, V)$, i.e.\ a non-zero $G(\A_f)$-stable subspace $V_0 \subseteq V$ defined over $\Q$.
	
	Moreover, the choice of $V_0$ is equivalent to the choice of a vector in the one-dimensional vector space $V_v^{K(\p^n)}$ for each finite place $v$ and corresponding prime $\p$, where $K(\p^n)$ is the minimal paramodular level of $(\pi_v, V_v)$.
\end{corollary}
\begin{proof}
	The result follows from Theorem~\ref{thm:Roberts_Schmidt} and~\cite[Lemma I.1]{Waldspurger_rational} applied to $H = K(\p^n)$, $\chi = 1$, because we have assumed that $\Q(\pi_f) = \Q$. 
\end{proof}

\begin{remark}
	This rational structure is denoted by $\pi_f'$ in \cite{LPSZ}.
\end{remark}

\begin{remark}
	Let $(\pi, V)$ be an irreducible, admissible representation of $\GSp(4, F)$ with trivial central character. Robert--Schmidt~\cite[Theorems 7.5.1, 7.5.8]{Roberts_Schmidt} also prove that:
	\begin{itemize}
		\item if $\pi$ is paramodular, then $\pi$ is generic,
		\item if $\pi$ is tempered, then $\pi$ is paramodular if and only if it is generic.
	\end{itemize}
	Therefore, our paper deals exactly with the automorphic representations whose local components are tempered and generic.
\end{remark}

\subsection{The rational structure on $\pi_\infty$}

Let $\pi$ be an automorphic representation of $\GSp_4(\R)$ associated to a holomorphic Siegel modular form $f$ with paramodular level structure. Then $\pi = \pi_f \otimes X_{\lambda;m}^1$ for a holomorphic (limit of) discrete series representation $X_{\lambda;m}^1$ for some $\lambda$ and $m$. The $L$-packet at infinity of $X_{\lambda;m}^1$ consists of two representations:
\begin{enumerate}
	\item $X_{\lambda; m}^1$, holomorphic (limit of) discrete series,
	\item $X_{\overline \lambda; m}^2$, generic (limit of) discrete series,
\end{enumerate}
(Lemma~\ref{lemma:multiplicities}). Moreover, each such $\pi_f \otimes \pi_\infty$ occurs in the automorphic spectrum with multiplicity one. The goal of this section is to put a rational structure on $\pi_\infty$ in each case. 

Let $f = \bigotimes\limits_v f_v \in \pi$ be a non-zero cusp form satisfying:
\vspace{-0.3cm}
\begin{enumerate}
	\item for each finite $v$, let $K(v^n)$ be the minimal paramodular level of $\pi_v$, and let $f_v \in \pi_v^{K(v^n)}$,
	\item if $\pi_\infty = X_\lambda^1$, $f_\infty$ is any vector in the minimal $K_\infty$-type $V_{(\lambda_1 + 1, \lambda_2 + 2)}$,
	\item if $\pi_\infty = X_{\overline \lambda}^2$, $f_\infty$ is any vector in the minimal $K_\infty$-type $V_{(\lambda_1+1, -\lambda_2)}$.
\end{enumerate}
(See Appendix~\ref{app:GSp(4,R)} for a more detailed discussion of the representation theory of $\GSp_4(\R)$ and Figure~\ref{fig:coherent cohomology} for the location of the minimal $K$-types of $\pi_\infty$.) By multiplicity one, such a vector $f$ is characterized by these properties up to a constant. 

\subsubsection{The generic representation $\pi_\infty = X_{\overline \lambda}^2$}

For the generic representation $\pi_\infty = X_{\overline \lambda}^2$, we note that $\pi$ is globally generic, and hence we can use a global Whittaker model to normalize the choice of $f$. 

Let $W$ be the Whittaker functional associated to $f$ and the choice of character $\psi_U$ of the unipotent radical~$U$ of the standard Borel (chosen as in~\cite[Section 2.1]{Chen_Ichino}):
\begin{equation}
	W(g) = \int\limits_{U(\Q) \backslash U(\A)}  f(ug) \overline{\psi_U(u)} \, du.
\end{equation}
We decompose $W = \prod\limits_v W_v$ as a product of local Whittaker functionals $W_v$ for $\pi_v$ with respect to $\psi_{U, v}$. 

\begin{definition}\label{def:Whittaker-norm}
	A vector $f = \bigotimes\limits_v f_v \in \pi$ is {\em Whittaker-normalized} if:
	\begin{enumerate}
		\item for $v$ finite, $W_v$ is chosen so that $W_v(1) = 1$,
		\item for $v = \infty$, we normalize $W_\infty$ so that $W_\infty(w_\infty)$ is given by the expression \cite[(1.2)]{Chen_Ichino}, where $w_\infty = \mathrm{diag}(1,1,-1,-1) \in \GSp_4(\R)$.
	\end{enumerate}
	We write $f^W = \bigotimes f_v^W$ for a Whittaker-normalized vector in $\pi$.
\end{definition}

\begin{remark}
	Comparing our Whittaker-normalized vector $f^W$ to the one used by Chen--Ichino~\cite{Chen_Ichino}, $f^{\mathrm{CI}}$, we see that $f^W = w_\infty(f^{\rm CI})$, since we chose a vector with $f_\infty \in V_{\lambda_1 +1, - \lambda_2}$ while they choose a vector with $f_\infty \in V_{\lambda_2, -\lambda_1-1}$. See also~\cite{Moriyama} for a discussion of these normalizations. With these choices, the resulting cohomology class $[w_\infty(f^W)] \in H^2(X_{\C}, \E_2)_f$ is called {\em Whittaker-$\Q$-rational} in the terminology of~\cite{LPSZ}.
\end{remark}

\begin{remark}
	We use the normalization of the Haar measures in Chen--Ichino~\cite{Chen_Ichino}, i.e.\ the Haar measure on $F_v^\times$ is chosen so that $\O_v^\times$ has volume $1$. This differs from~\cite{Roberts_Schmidt} who use $d^\times x = dx / |\cdot|$.  In any case, the normalization $W_v(1) = 1$ in our notation agrees with the normalization $Z(W_v, s) = L(\pi_v, s)$ used by Roberts--Schmidt~\cite{Roberts_Schmidt}. 
\end{remark}

Finally, we compare the Whittaker normalized vector to the rational structure on coherent cohomology. Recall that in Definition~\ref{def:[f]}, given a factorizable vector $f = \bigotimes f_v \in \pi$ with $f_v \in \pi_v^{K(\p^n)}$ and $f_\infty \in \pi_\infty$ a highest weight vector in a minimal $K$-type, we defined its {\em contribution to cohomology} $[f] \in H^i(X_\C, \E_i)$ with~$i$ determined by $f_\infty \in \pi_\infty$ via Theorem~\ref{thm:reps_in_coherent_cohomology} (see Figure~\ref{fig:coherent cohomology}). In particular, the assumption that $f_\infty \in V_{(\lambda_1 + 1, -\lambda_2)}$ implies that:
\begin{align}
	[f^W] & \in H^1(X_\C, \E_1)_\Pi.
\end{align}

We check this class is in fact defined over $\R$.

\begin{proposition}\label{prop:[f^W]_is_real}
	The Whittaker-normalized vector $f^W$ gives rise to a real coherent cohomology class, i.e.\ $[f^W] \in H^1(X_\R, \E_1)_\Pi$.
\end{proposition}

To prove this, we describe the action of complex conjugation on coherent cohomology. Let $Y$ be a proper algebraic variety over $\R$ and $\mathcal E_0$ be a locally free $\O_Y$-module. Write $Y_\C$ and $\E$ for the base change of $Y$ and $\E_0$ to $\C$. Then there is a complex conjugation map $c_Y \colon Y_\C \to Y_\C$ and we have that $c_Y^\ast \E \iso \E$. Using base change for coherent cohomology, we define complex conjugation $c$ as follows:
\begin{center}
	\begin{tikzcd}
		H^i(Y_\R, \mathcal E_0) \otimes \C \ar[rr, "{1 \otimes \overline{(-)}}"] \ar[d, "\iso"] & &  H^i(Y_\R, \mathcal E_0) \otimes \C \ar[d, "\iso"] \\
		H^i(Y_\C, \mathcal E) \ar[rr, "c"] \ar[rd, swap, "c_Y^\ast"] & & H^i(Y_\C, \mathcal E) \\
		& H^i(Y_\C, c_Y^\ast \mathcal E) \ar[ru, swap, "c_Y^\ast \E \iso \E"]
	\end{tikzcd}
\end{center}
Recall that coherent cohomology may be computed using the Dolbeault complex $\mathcal A^{0, \bullet} \otimes \mathcal E$, i.e.\
$$H^i(Y_\C, \mathcal E) \iso H^i\left( \Gamma(Y(\C),  \mathcal A^{0, \bullet} \otimes \mathcal E)\right).$$
One can then check that the action of complex conjugation on the left-hand side corresponds to the action on the complex $(\mathcal A^{0, \bullet} \otimes \E, \overline{\partial})$ which is locally given by
\begin{equation}\label{eqn:complex_conj_Dolb}
	f(x) d\overline{x_1} \wedge \cdots \wedge d \overline{x_i} \mapsto \overline{f(\overline x)} d\overline{x_1} \wedge \cdots \wedge d\overline{x_i}.
\end{equation}
Using this, we can describe the action of complex conjugation on the class $[f]$ as follows.

\begin{lemma}\label{lemma:cc_on_coherent_cohomology}
	Let $c \colon H^i(X_\C, \mathcal E_i) \to H^i(X_\C, \mathcal E_i)$ be complex conjugation on coherent cohomology and $[f] \in H^i(X_\C, \mathcal E_i)$ be a class associated with an appropriate automorphic form as in Definition~\ref{def:[f]}. Then $c[f]$ is the coherent cohomology class associated with the automorphic form 
	$$g \mapsto \overline{f(g \cdot w_\infty)}$$
	for $w_\infty = \diag(1,1,-1,-1) \in \GSp_4(\R)$.
\end{lemma}
\begin{proof}
	To prove this lemma, we have to examine the proof of Su's Theorem~\ref{thm:Su} in \cite{Su}, through which we defined the contribution to cohomology $[f]$ of an automorphic form $f$ in Definition~\ref{def:[f]}. For the following discussion, we hence adopt the notation of his paper. The cohomology of the sheaf $\widetilde V$ on the (open) Shimura variety $\mathrm{Sh}_{\mathbb K}$ is computed in Section 1.1.2 using Dolbeault cohomology. In particular, equation~(1.9) gives an explicit isomorphism
	\begin{equation}\label{eqn:Dolbeault_vs_CE_complex}
		\varphi \colon (\widetilde V \otimes_{\mathcal O_{\mathrm{Sh}_{\mathbb K}}} \mathcal A^{0,i}_{\mathrm{Sh}_{\mathbb K}} )(U) \overset\iso\to \Hom_{K_\circ}\left( {\bigwedge}^i (\p_0/\mathfrak k_0), C^\infty(\pi_\circ^{-1}(U)) \otimes V  \right)
	\end{equation}
	which induces an isomorphism of complexes between the Dolbeault complex and the Chevalley-Eilenberg complexes. Given $\varpi \in  (\widetilde V \otimes_{\mathcal O_{\mathrm{Sh}_{\mathbb K}}} \mathcal A^{0,i}_{\mathrm{Sh}_{\mathbb K}} )(U)$, equation~(1.8) gives the formula
	$$\varphi(\varpi)(Y_1 \wedge \cdots \wedge Y_i)(y) := \varphi(d \pi_{\circ, y} Y_1 \wedge \cdots \wedge d \pi_{\circ, y} Y_i)(y) \in V.$$
	Hence formula~\eqref{eqn:complex_conj_Dolb} for complex conjugation on the Dolbeault complex, yields the following formula on the right-hand side of isomorphism~\eqref{eqn:Dolbeault_vs_CE_complex}: a $K_\circ$-equivariant map $\psi \colon  {\bigwedge}^i (\p_0/\mathfrak k_0) \to C^\infty(\pi_\circ^{-1}(U)) \otimes V$ is sent to
	\begin{equation}\label{eqn:complex_conj_CE_complex}
		c(\psi)(Y_1 \wedge \cdots \wedge Y_i)(y) := \overline{\psi(Y_1 \wedge \cdots \wedge Y_i)(\overline y)}.
	\end{equation}
	Next, if $j \colon \mathrm{Sh}_{\mathbb K} \hookrightarrow \mathrm{Sh}_{\mathbb K, \Sigma}$ is an admissible toroidal compactification, Definition 2.8 of loc.\ cit.\ defines the sheaf $\mathcal I^i_V$ as the sheafification of the subsheaf
	\begin{align*}
		U \mapsto \Hom_{K_{\circ}}\left({\bigwedge}^i (\mathfrak p_0/\mathfrak k_0), C^\infty_{\mathrm{dmg}}(\pi_\circ^{-1}(U)) \otimes V \right)
	\end{align*}
	of the sheaf
	\begin{align*}
		U \mapsto \Hom_{K_{\circ}}\left({\bigwedge}^i (\mathfrak p_0/\mathfrak k_0), C^\infty(\pi_\circ^{-1}(U)) \otimes V \right)
	\end{align*}
	The operator $\psi \mapsto c(\psi)$ preserves the growth condition, and hence the subsheaf. Finally, Proposition 2.10 identifies the global sections of this sheaf with the Chevalley-Eilenberg complex with the appropriate growth condition:
	\begin{equation}
		\mathcal I_V^i(\mathrm{Sh}_{\mathbb K, \Sigma}) = \Hom_{K_\circ} \left( {\bigwedge}^i (\mathfrak p_0/ \mathfrak k_0), C^\infty_{\mathrm{dmg}}(G)^{\mathbb K} \otimes V \right),
	\end{equation}
	and, accordingly, the action of complex conjugation of the final complex is also given by equation~\eqref{eqn:complex_conj_CE_complex}. 
		
	It remains to observe that:
	\begin{itemize}
		\item by~\cite[II.7]{Milne:AVB}, the complex conjugation map $X_\C \to X_\C$ is given on the group variable by $g \mapsto g \cdot w_\infty$;
		\item by~\cite[Theorem 5.1~(b)]{Milne_Canonical_models}, the complex conjugation action on $(\E_i)_\C$ is given on sections by $s \mapsto \overline{s}$.
	\end{itemize}
	This completes the proof of the lemma.
\end{proof}

\begin{remark}
	This is analogous to Shimura's complex conjugation $f \mapsto f^\varrho$ on classical modular forms, given by $f^\varrho(z) = \overline{f(-\overline z)}$.
\end{remark}

\begin{proof}[Proof of Proposition~\ref{prop:[f^W]_is_real}]
	Since $\pi_f$ is defined over $\Q$, $\overline{\pi_f} \iso \pi_f$, and hence the finite part of $f^W$ is conjugation-invariant. By Lemma~\ref{lemma:cc_on_coherent_cohomology}, it remains to check that $f^W_\infty(g_\infty) = \overline{f^W_\infty(g_\infty \cdot w_\infty)}$. 
	
	Recall from Appendix~\ref{app:GSp(4,R)} that $X_{\overline{\lambda}}^2|_{\Sp_4(\R)} = X_{\overline{\lambda}}^2 \oplus X_{\overline{\lambda}'}^3$ (with a slight abuse of notation), and $w_\infty$ defines a map $X_{\overline{\lambda}}^2 \to X_{\overline{\lambda}'}^3$. Similarly, using the explicit construction of these representation in Theorem~\ref{thm:Muic_Klingen}, we deduce that $\overline{X_{\overline{\lambda}}^2} = X_{\overline{\lambda}'}^3$ because $\overline{D_k^+} = D_k^-$. This shows that $g_\infty \mapsto \overline{f^W_\infty(g_\infty \cdot w_\infty)}$ is a vector in the same representation of $\Sp_4(\R)$ as $f^W$. Since the two maps also act the same way on $K^\circ$-types, we see that both these vectors also belong to the same $K^\circ$-type. To check that they are the same vector, it remains to observe that the value used for the Whittaker-normalization in~\cite[(1.2)]{Chen_Ichino} is real.
\end{proof}

Next, note that $w_\infty(f_\infty) \in V_{(\lambda_2, -\lambda_1 - 1)}$, and hence:
\begin{align}
	[w_\infty(f^W)] & \in H^2(X_\C, \E_2)_\Pi.
\end{align}
(In fact, one can again check that $[w_\infty(f^W)] \in H^2(X_\R, \E_2)_\Pi$ as well.) Since $\dim H^2(X_\C, \E_2)_\Pi = 1$, we can rescale this cohomology class to be rational.

\begin{definition}\label{def:Whit_period}
	The {\em Whittaker period} $c^W(\pi_f)$ associated with $\pi_f$ is a complex number, well-defined up to~$\Q^\times$, such that
	$$\frac{[w_\infty(f^W)]}{c^W(\pi_f)} \in H^2(X_\Q, \E_2)_\Pi$$
	is rational in coherent cohomology. We will also write $c^W(f)$ for $c^W(\pi_f)$.
\end{definition}

\begin{remark}
	This agrees with the Whittaker period $\Omega^W(f)$ in \cite[Section 10.1]{LPSZ}. The idea to define periods using contributions to higher coherent cohomology goes back to Harris~\cite{Harris:occult}, who used Bessel models instead of Whittaker models.
\end{remark}

\subsubsection{The holomorphic representation $\pi_\infty = X_\lambda^1$}

When $\pi_\infty = X_\lambda^1$, the representation $\pi$ is not globally generic. We instead give a rational structure on $\pi_\infty$ using coherent cohomology. This idea goes back to~\cite{Blasius_Harris_Ramakrishnan} who used it to prove that there is a rational structure on $\pi_f$. Since we already put a rational structure on $\pi_f$, we will instead use it to define a rational structure on $\pi_\infty$.

Let $f = \bigotimes f_v \in \pi$ and assume that for every finite $v$, $f_v$ is Whittaker-normalized according to Definition~\ref{def:Whittaker-norm}. To be more precise, by Lemma~\ref{lemma:multiplicities}, the $L$-packet of $\pi_f$ consists of $\pi^h = \pi_f \otimes X_\lambda^1$ and $\pi^g = \pi_f \otimes X_{\overline\lambda}^2$ and we choose $f_v$ for finite $v$ so that the resulting vector in $\pi^g$ is Whittaker-normalized. We then consider the associated cohomology class $[f] \in H^0(X, \cal E_0)_\Pi$ for $f$ as above.

\begin{definition}
	A vector $f_\infty \in X_\lambda^1$ is {\em rational} if the resulting vector $f = \bigotimes f_v \in \pi$ gives a $\Q$-rational section $[f] \in H^0(X, \cal E_0)_\Pi$.
\end{definition} 

In other words, we get a rational structure on $X_\lambda^1$ by using the isomorphism:
$$H^0(X_\C, \E_0)_\Pi \iso \Hom_K(V_0^\vee, X_{\lambda;m}^1) \otimes \pi_f^{K_f},$$ 
and the rational structures on $\pi_f^{K_f}$ defined in Corollary~\ref{cor:rational_stru} and $H^0(X_\C, \E_0)_\Pi$ given by the rational coherent cohomology $H^0(X, \cal E_0)_\Pi$.

\subsection{The motivic action}

We are finally ready to define the motivic action and state our main theorem. As above, $f$ is a holomorphic cuspidal Siegel modular form of weight $(k,2)$ for $k \geq 2$ even, paramodular level $N$, with trivial central character, and defined over $\Q$.



We will be interested in the adjoint motive $M(f, \Ad)$ and we recall that $M(f, \Ad) = \Sym^2 M(f) (k-1)$. We record the appropriate Beilinson short exact sequence~\eqref{eqn:Beilinson_ses}:
\begin{equation}\label{eqn:ses_Sym2}
	0 \to \underbrace{F^{k} H_{\mathrm{dR}}^{2k-2} (\Sym^2 M(f)_\R)}_{\dim_\R  = 3} \overset{\tilde \pi_{k-1}}\to \underbrace{H^{2k-2}_B (\Sym^2 M(f)_\R, \R(k-1)) }_{\dim_\R = 4} \to \underbrace{H^{2k-1}_{\cal D} (\Sym^2 M(f)_\R, \R(k))}_{\dim_\R = 1} \to 0,
\end{equation}
where $H^{2k-1}_{\cal D} (\Sym^2 M(f)_\R, \R(k)) \iso H^1_{\cal D}(M(f, \Ad)_\R, \R(1))$, etc.

Because $\dim H^{2k-1}_{\cal D} (\Sym^2 M(f)_\R, \R(k)) = 1$, the Deligne cohomology group is naturally identified with its determinant.

As in~\cite{Prasanna_Venkatesh, Horawa} , we want to define the action of the dual Deligne cohomology group $H^1_{\cal D}(M(f, \Ad)_\R, \R(1))^\vee$ on $H^\ast(X_\C, \E_{k, 2})_\Pi$ and conjecture that the resulting action of the rational structure coming from motivic cohomology descends to the rational structure on coherent cohomology.

We will define a natural generator of $H^1_{\cal D}(M(f, \Ad)_\R, \R(1))$. First, note that:
\begin{align*}
	H^1_{\mathcal D}(M(f, \Ad)_\R, \R(1)) & \iso H^{2k-1}_{\cal D} (\Sym^2 M(f)_\R, \R(k)) \\
	& \iso H^{2k-2}_B (\Sym^2 M(f)_\R, \R(k-1)) / \widetilde \pi_{k-1}  & \text{by \eqref{eqn:ses_Sym2}} \\
	& \iso (H^{k-1,k-1})^{c_B = (-1)^{k-1}, F_\infty = (-1)^{k-1}} & \text{\cite[Sec.\ 2.2.4]{Prasanna_Venkatesh}}
\end{align*}
Given a basis $\omega_1, \omega_2 \in F^{k-1} H^{k-1}_{\mathrm{dR}}(M(f))$, a natural basis of $ F^{k} H_{\mathrm{dR}}^{2k-2} (\Sym^2 M(f))$ is given by
$$\omega_1 \otimes \omega_1, \omega_1 \otimes \omega_2 + \omega_2 \otimes \omega_1, \omega_2 \otimes \omega_2$$
and we define a natural basis of $(H^{k-1,k-1})^{c_B = (-1)^{k-1}, F_\infty = (-1)^{k-1}}$ below.

\begin{definition}\label{def:natural_gen}
	A {\em natural generator} of the Deligne cohomology group $H^{2k-1}_{\cal D} (\Sym^2 M(f)_\R, \R(k))$ is
	$$\delta' = (\omega_1 \otimes \overline{\omega_2} + \overline{\omega_2} \otimes \omega_1) + (-1)^{k-1} (\omega_2 \otimes \overline{\omega_1} + \overline{\omega_1} \otimes \omega_2)  \in (H^{k-1,k-1}(\Sym^2 M(f)))^{c_B = (-1)^{k-1}, F_\infty = (-1)^{k-1}}.$$
	Therefore, a {\em natural generator} of the Deligne cohomology group $H^1_{\mathcal D}(M(f, \Ad)_\R, \R(1))$ is:
	$$\delta = \delta(\omega_1, \omega_2) = (2\pi i)^{k-1} \delta' \in (H^{0,0}(M(f, \Ad)))^{c_B = 1, F_\infty = 1}.$$
\end{definition}

Both $\delta$ and $\delta'$ are well-defined up to a rational constant; specifically, 
$$\delta(a\omega_1 + b \omega_2, c \omega_1 + d\omega_2) = (ad - bc) \delta(\omega_1, \omega_2)$$ and similarly for $\delta'$.

\begin{remark}\label{rmk:rational_structure_on_Deligne}
	Of course, the choice of $\delta'$ defines a rational structure $V_\Q$ on the one-dimensional real vector space $V := H^{2k-1}_{\cal D} (\Sym^2 M(f)_\R, \R(k))$. This rational structure may be defined abstractly, without picking bases, as follows.
	
	Using Beilinson's short exact sequence~\eqref{eqn:ses_Sym2}, $V = H^{2k-1}_{\cal D} (\Sym^2 M(f)_\R, \R(k))$ may be identified with the complement of $F^{k} H_{\mathrm{dR}}^{2k-2} (\Sym^2 M(f)_\R)$ inside $H^{2k-2}_B (\Sym^2 M(f)_\R, \R(k-1))$. Since $\Sym^2 M(f)$ is a quotient of $M(f) \otimes M(f)$, we have a commutative diagram:
	\begin{center}
		\begin{tikzcd}
			F^{k} H_{\mathrm{dR}}^{2k-2} (\Sym^2 M(f)_\R) \ar[r, hook] & H^{2k-2}_B (\Sym^2 M(f)_\R, \R(k-1)) \ar[r, two heads] &  V \\
			F^{k-1} H^1_{\mathrm{dR}}(M(f))_\R \otimes F^{k-1} H^1_{\mathrm{dR}}(M(f))_\R \ar[u, two heads] \ar[r, "\pi_0 \otimes \pi_{k-1}"] & H^{k-1}_B(M(f)_\R, \R) \otimes H^{k-1}_B(M(f)_\R, \R(k-1)) \ar[ru, out = 0, in = -90] \ar[u, two heads] 
		\end{tikzcd}
	\end{center}
	This description leads to a natural rational structure $V_\Q \subseteq V$ on Deligne cohomology: the complement of $F^{k} H_{\mathrm{dR}}^{2k-2} (\Sym^2 M(f)_\R) $ is generated by the image of $H^{k-1,0}(M(f)) \otimes H^{0,k-1}(M(f))$, and we define $V_\Q$ as the image of the rational structure $F^{k-1} H^{k-1}_{\rm dR}(M(f)) \otimes \overline{F^{k-1} H^{k-1}_{\rm dR}(M(f))} \subseteq H^{k-1,0} \otimes H^{0,k-1}$. In terms of the above diagram, this is summarized as follows:
	\begin{center}
		\begin{tikzcd}
			F^{k-1}H^{k-1}_{\rm dR}(M(f)) \otimes \overline{F^{k-1} H^{k-1}_{\rm dR}(M(f))} \ar[d, hook] \ar[rr, "\iso"] & & V_\Q \ar[d, hook] \\
			H^{k-1,0}(M(f)) \otimes H^{0,k-1}(M(f)) \ar[r] & H^{k-1}_B(M(f)_\R, \R) \otimes H^{k-1}_B(M(f)_\R, \R(k-1)) \ar[r, two heads] & V.
		\end{tikzcd}
	\end{center} 
	One can then easily check that the natural generator $\delta' \in V$ defined above belongs to the rational structure~$V_\Q$.
\end{remark}


As in \cite{Prasanna_Venkatesh}, it is the dual to $H^1_{\mathcal D}(M(f, \Ad)_\R, \R(1))$ that should act on cohomology. We will use a polarization to identify $H^1_{\mathcal D}(M(f, \Ad)_\R, \R(1))$ with its dual, following~\cite[\S 2.2]{Prasanna_Venkatesh}.

\begin{definition}
	Let $\langle -, - \rangle_{\pol} \colon M(f) \otimes M(f) \to \Q(-(k-1))$ be a polarization of the motive $M(f)$ of weight $k-1$, i.e.\ a perfect $(-1)^{k-1}$-symmetric pairing whose Betti realization $H_B(s)$ gives a polarization of the rational Hodge structure $H_B(M(f))$.
	
	Associated with $\langle -, - \rangle_{\pol}$, there is a symmetric perfect pairing
	$$\langle -, - \rangle_{\pol} \colon M(f, \Sym^2) \otimes M(f, \Sym^2) \to \Q(-2(k-1)),$$
	inducing a polarization of the rational Hodge structure $H_B(M(f, \Sym^2)_\C, \Q)$. Using $M(f, \Sym^2)(k-1) = M(f, \Ad)$, we also get an associated polarization
	$$\langle -, - \rangle_{\pol} \colon M(f, \Ad) \otimes M(f, \Ad) \to \Q,$$
\end{definition}

\begin{definition}\label{def:dual_natural_generator}
	The {\em dual natural generator} $\delta^\vee$ of $H^1_{\mathcal D}(M(f, \Ad), \R(1))^\vee$ associated with a natural generator~$\delta$ and a polarization is defined by 
	$$\delta^\vee := \frac{\pi^{2k}}{\sqrt{\Delta_{\Ad(f)}}} \langle \delta, - \rangle_{\pol}.$$
\end{definition}

\begin{remark}\label{rmk:polarization}
	Even though we have chosen a polarization in order to define the dual natural generator, it is well-defined up to $\Q^\times$. Indeed, $\delta^\vee$ only depends on the restriction of the polarization to a pairing on the one-dimensional vector space $H^1_{\mathcal D}(M(f, \Ad), \R(1))$, which is unique up to $\Q^\times$.
	
	When $(k,2) = (2,2)$, $f$ corresponds to an abelian surface $A$ with $\End(A)$ according to the Brumer--Kramer conjecture~\ref{conj:Brumer_Kramer}. Then the polarization $\langle - , - \rangle_{\pol}$ on $M(f) = H^1(A)$ is associated with a polarization of $A$, which is unique up to rational scalars. Indeed, if $i \colon A \to A^\vee$ and $j \colon A \to A^\vee$ are two polarizations, then $j^\vee \circ i \colon A \to A$ is an isogeny, i.e. an element of $\End(A) \otimes \Q = \Q$. This shows that $i$ is a scalar multiple of $j$.
 \end{remark}

\begin{remark}
	While the appearance of the factor $\frac{\pi^{2k}}{\sqrt{\Delta_{\Ad(f)}}}$ may be surprising at this stage, it seems to be needed for the eventual rationality statement. More specifically, it is related to factors appearing in the functional equation for the adjoint $L$-function. Perhaps there is a different phrasing of the conjecture which avoids using the functional equation, but we decided not to pursue this here and stuck to our original approach for computational convenience.%
\end{remark}

\begin{definition}\label{def:motivic_action}
	Let $f^W \in \pi_f \otimes X_{\overline \lambda}^2$ be a Whittaker-normalized vector and $f \in \pi_f \otimes X_{\lambda}^1$ be the associated $\Q$-rational vector. We define an action of $ H^1_{\mathcal D}(M(f, \Ad)_\R, \R(1))^\vee$ on $H^\ast(X_\R, \E_{k,2})_\Pi$ by letting a dual natural generator $\delta^\vee \in H^1_{\mathcal D}(M(f, \Ad)_\R, \R(1))^\vee$ act by:
	\begin{align*}
		H^0(X_\R, \E_{k,2})_\Pi & \to H^1(X_\R, \E_{k,2})_\Pi  \\
		[f] & \mapsto [f^W].
	\end{align*}
	Note that this action is well-defined up to $\Q^\times$.
\end{definition}

Recall that we have a Beilinson regulator~\eqref{eqn:Beilinson_reg} map:
\begin{equation}\label{eqn:our_Beilinson_reg}
	r_{\cal D} \colon H^1_{\cal M}(M(f, \Ad), \Q(1)) \otimes \R \to H^1_{\cal D}(M(f, \Ad)_\R, \R(1))
\end{equation}
which is an isomorphism by Hypothesis~\ref{hyp:reg_is_isom} (i.e. part~(1) of Beilinson's Conjecture~\ref{conj:Beilinson}). Under this hypothesis, we define a degree-shifting action of
$H^1_{\mathcal M}(M(f, \Ad)_\Z, \Q(1))^\vee \otimes \R$ on $H^\ast(X_\C, \mathcal E_{k,2})_{\Pi}$ using the Beilinson regulator~\ref{eqn:our_Beilinson_reg} and Definition~\ref{def:motivic_action}. Our main conjecture is that the resulting motivic action preserves the rational structure on coherent cohomology.

\begin{conjecture}\label{conj:motivic_action}
	The action of $H^1_{\mathcal M}(M(f, \Ad)_\Z, \Q(1))^\vee \otimes \R$ on $H^\ast(X_\R, \mathcal E_{k,2})_{\Pi}$ descends to rational structures, i.e.\ the action of $H^1_{\mathcal M}(M(f, \Ad), \Q(1))^\vee$ preserves the rational structure $H^\ast(X, \E_{k, 2})_\Pi \subseteq H^\ast(X_\R, \mathcal E_{k,2})_{\Pi}$.
\end{conjecture}

We can make this explicit in the following way. Given a non-zero motivic cohomology class $\alpha \in H^1_{\mathcal M} (M(f, \Ad)_\Z, \Q(1))$, we have that:
$$\frac{\delta^\vee}{ \delta^\vee ( r_{\mathcal D}(\alpha))} \in H^1_{\mathcal M}(M(f, \Ad)_\Z, \Q(1))^\vee.$$ 
According to Definition~\ref{def:motivic_action}, this element acts by:
\begin{align}
	H^0(X_\R, \mathcal E_{k,2})_\Pi & \to H^1(X_\R, \mathcal E_{k,2})_\Pi \nonumber \\
	[f] & \mapsto \frac{\sqrt{\Delta_{\Ad(f)}}}{\pi^{2k}} \frac{[f^W]}{ \langle r_{\mathcal D}(\alpha), \delta \rangle_{\mathrm{pol}}}.
\end{align}
Therefore, Conjecture~\ref{conj:motivic_action} is equivalent to the following rationality statement.

\begin{conjecture}\label{conj:explicit}
	For a non-zero motivic cohomology class $\alpha \in H^1_{\mathcal M} (M(f, \Ad)_\Z, \Q(1))$, the coherent cohomology class
	$$ \frac{\sqrt{\Delta_{\Ad(f)}}}{\pi^{2k}} \frac{[f^W]}{ \langle r_{\mathcal D}(\alpha), \delta \rangle_{\mathrm{pol}}} \in  H^1(X_\R, \mathcal E_{k-2})_\Pi$$
	is rational, i.e. belongs to $H^1(X, \mathcal E_{k-2})_\Pi \subseteq  H^1(X_\R, \mathcal E_{k-2})_\Pi$.
\end{conjecture}

\begin{theorem}\label{thm:main}
	Assume:
	\begin{enumerate}
		\item Deligne's conjecture for $L(M(f), \psi_{\pm}, s)$ at the critical point $s = k-1$,
		\item Beilinson's conjecture for $L(M(f, \Ad), s)$ at the point $s = 1$.
	\end{enumerate}
	Then Conjecture~\ref{conj:explicit} holds (and hence so does Conjecture~\ref{conj:motivic_action}).
\end{theorem}



The proof of Theorem~\ref{thm:main} will occupy the next section. For the reader's convenience, we give an outline here.  The goal is to prove the rationality of the cohomology class
$$\omega =  \frac{\sqrt{\Delta_{\Ad(f)}}}{\pi^{2k}} \frac{[f^W]}{ \langle r_{\mathcal D}(\alpha), \delta \rangle_{\mathrm{pol}}} \in H^1(X_\R, \E_{k,2})_\Pi.$$ 
We consider a rational class $\eta \in H^2(X, \E_{1,3-k})_\Pi$ and compute that they pair to a non-zero rational number under the Serre duality pairing $\langle \omega , \eta \rangle_{\rm SD} \in \Q^\times$. Since the space $H^1(X, \cal E_{k,2})_\Pi$ is one-dimensional, this will prove the theorem.

We temporarily write $a \approx b$ if $a = \pi^n \sqrt{c} b$ for $n \in \Z$ and $c \in \Q^\times$, and present the main steps of the argument up to $\approx$. More precise statements are given in the following sections.

\begin{enumerate}
	\item Recall the Whittaker-normalized vector $w_\infty(f^W)$, which has an associated cohomology class $[w_\infty(f^W)] \in H^2(X_\C, \E_{1,3-k})_f$, and the Whittaker period $c^W(f)$,  which was defined so that 
	$$\eta = \frac{[w_\infty(f^W)]}{c^W(f)} \in H^2(X, \cal E_{1,3-k})_f$$
	is rational (Definition~\ref{def:Whit_period}).
	
	\item  In Theorem~\ref{thm:cW(f)=c+c-}, we will check that the period $c^W(f)$ is related to Beilinson periods for $M(f)$ as follows:
	$$c^W(f) \approx c^+(M(f)) \cdot c^-(M(f)).$$
	
	\item Chen--Ichino~\cite{Chen_Ichino} prove that:
	$$\langle f^W, f^W \rangle \approx L(M(f, \Ad), 1).$$
	
	\item In Theorem~\ref{thm:Beilinson_for_Sym2}, we check that Beilinson's conjecture is equivalent to the statement:
	$$L(M(f, \Ad), 1) \approx c^+(M(f)) \cdot c^-(M(f)) \cdot \langle r_{\cal D}(\alpha), \delta \rangle_{\mathrm{pol}}$$
	for $\alpha \in H^1_{\cal M}(\Ad(M), \Q(1))$.
	
	\item Finally, we compute that:
	\begin{align*}
		\langle \omega, \eta \rangle_{\rm SD} & \approx \frac{\langle f^W, f^W \rangle}{\langle r_{\cal D}(\alpha), \delta \rangle_{\mathrm{pol}} \cdot c^W(f)} \\
		& \approx \frac{L(M(f, \Ad), 1)}{\langle r_{\cal D}(\alpha), \delta \rangle_{\mathrm{pol}} \cdot c^W(f) } & \text{by (3)} \\
		& \approx \frac{c^+(M(f)) c^-(M(f))}{c^W(f)} & \text{by (4)} \\
		& \approx 1 & \text{by (2)}
	\end{align*}
	which completes the proof up to $\approx$.
\end{enumerate}

In Section~\ref{sec:Yoshida}, we also prove Conjecture~\ref{conj:motivic_action}, unconditionally on the rationality statement of Beilinson's conjecture, in the case of Yoshida lifts associated with real quadratic fields by using results of Ramakrishnan~\cite{Ramakrishnan}. In Section~\ref{sec:imag_quad}, we consider Yoshida lifts associated with imaginary quadratic fields and show that our conjecture implies the conjecture of Prasanna--Venkatesh~\cite{Prasanna_Venkatesh} in this case. In both of these special cases, we only consider the case $k = 2$.

\section{Proof of Theorem~\ref{thm:main}}


\subsection{The Whittaker period}

Recall that we defined the Whittaker period $c^W(\pi_f)$ associated to the Whittaker-normalized vector $f^W$ (Definition~\ref{def:Whit_period}). Inspired by the work of Loeffler--Pilloni--Skinner--Zerbes~\cite{LPSZ}, we relate it to Deligne's conjecture for spin $L$-functions.

\begin{theorem}[Loeffler, Pilloni, Skinner, Zerbes]\label{thm:c^W_and_spin_L_values}
	Let $f$ be a holomorphic Siegel modular form of weight $(k_1,k_2)$, paramodular level $N$, and coefficients in $E$. Then for Dirichlet characters $\psi_+, \psi_-$ such that $\psi_{\pm}(-1) = \pm 1$, we have that:
	$$c^W(f) \sim_{E(\psi_+, \psi_-)^\times} \Lambda\left(f, \psi_+, k_1 - 1\right) \Lambda\left(f, \psi_-,  k_1 - 1\right).$$
\end{theorem}
\begin{proof}
	This is Proposition~10.3 in \cite{LPSZ} (stated in the case $k_1 \geq k_2 \geq 3$; for generalization to the case $k_2 = 2$, see~\cite{Loeffler2020bloch}). Note that we use the motivic normalization of $L$-functions, instead of the automorphic one.
\end{proof}

We stated the theorem for the right-most critical value, according to Lemma~\ref{lemma:Beil_for_spin}. There are also analogous statements for the other critical values.

Next, we want to express $c^W(f)$ in terms of Deligne periods, so we need to make sure that there exist non-vanishing twists of the spin $L$-function.

\begin{theorem}\label{thm:cW(f)=c+c-}
	Let $f$ be a holomorphic Siegel modular form of weight $(k_1, k_2)$, paramodular level $N$, and rational coefficients. Assuming Deligne's conjecture, we have that:
	$$c^W(f) \sim_{\Q^\times}  \pi^{- 4 (k_1-1) + 2(k_2 - 2)} c^{+}(M(f)(k_1 - 1)) c^-(M(f)(k_1 - 1)),$$
	where $c^\pm(M(f)(k_1))$ are the Deligne periods associated to the motive $M(f)(k_1)$ of $f$. 
\end{theorem} 
\begin{proof}
	We have that:
	$$\Lambda(f, \psi_\pm, s) = L(f, \psi_\pm, s) \cdot L_\infty(f, \psi_\pm, s)$$
	where
	$$L_\infty(f, \psi_\pm, s) = \Gamma_\C(s) \Gamma_\C(s + 2 - k_2),$$
	and hence
	$$L_\infty\left(f, \psi_\pm, k_1 - 1\right) = \Gamma_\C\left( k_1  - 1 \right) \Gamma_\C\left( s + k_1 + 1  - k_2 \right) \sim_{\Q^\times} \pi^{- 2 (k_1-1) + (k_2 - 2)}.$$
	Altogether, Theorem~\ref{thm:c^W_and_spin_L_values} shows that according to Beilinson's conjecture:
	\begin{align*}
		c^W(f) \sim_{\Q(\psi_+, \psi_-)^\times} \pi^{- 4 (k_1-1) + 2(k_2 - 2)} g(\psi_+)^2 g(\psi_-)^2 c^+(M(f)(k_1 - 1)) c^-(M(f)(k_1 - 1)),
	\end{align*}
	as long as there exists $\psi_+$, $\psi_-$ such that $L(f, \psi_\pm, k_2 - 1) \neq 0$. Consider the automorphic representation $\Pi$ of $\GL_4(\A)$ associated with $f$ so that $L(f, s) = L(\Pi, s)$. If $f$ is non-endoscopic, then $\Pi$ is cuspidal, so such $\psi_+$ and $\psi_-$ exist by the recent result of \cite{radziwill2023nonvanishing}. On the other hand, if $f$ is endoscopic, then $\Pi = \pi_1 \boxplus \pi_2$ for automorphic representations $\pi_1$, $\pi_2$ of $\GL_2(\A)$, so such $\psi_+$ and $\psi_-$ exist by \cite{friedberg1995nonvanishing}.
	
	Altogether, this proves that
	$$
	\frac{c^W(f)}{\pi^{- 4 (k_1-1) + 2(k_2 - 2)} c^+(M(f)(k_1 - 1)) c^-(M(f)(k_1 - 1))} \in \Q^{\rm ab}.
	$$
	Moreover, since $f^\sigma = f$ for any $\sigma \in \Gal(\Q^{\rm ab}/\Q)$, we conclude that this quotient in fact lies in $\Q$.
\end{proof}

\begin{remark}
	The strategy behind the proof goes back to Harris' occult periods~\cite{Harris:occult} who uses Bessel periods instead of Whittaker periods; see \cite[Remark 6.7]{LPSZ} for the difference between the two approaches. 
\end{remark}

\subsection{The adjoint $L$-value and Petersson inner products}

We recall a theorem of Chen and Ichino which relates the square of the Petersson norm of $f^W$ to the adjoint $L$-value, and hence verifies the conjecture of Lapid--Mao~\cite{Lapid_Mao} in this case. Note that our notation is slightly different; the precise relationship is $(\lambda_1 + 1, -\lambda_2) = (\lambda_1^{\mathrm{CI}}, \lambda_2^{\mathrm{CI}})$.

\begin{theorem}[Chen--Ichino~\cite{Chen_Ichino}, Chen~\cite{chen2022algebraicity}]
	Let $\pi$ be a globally generic cuspidal automorphic representation of $\GSp_4(\A)$ of square-free paramodular level $N$ such that the Harish--Chandra parameter of $\pi_\infty$ is $(\lambda_1, -\lambda_2)$. Let $f^W \in \pi$ be a Whittaker-normalized vector. Then:
	$$\langle f^W, f^W \rangle = 2^c \cdot \frac{  \Lambda(1, \pi, \Ad)}{\Lambda(2) \Lambda(4)} \cdot \prod_{v|N\infty} C_v$$
	where $c = 2$ if $\pi$ is endoscopic and $c = 1$ otherwise, $\Lambda(s)$ is the completed Riemann $\zeta$-function, and $C_v$ are explicitly described constants, satisfying:
	$$\prod_{v|N\infty} C_v \sim_{\Q^\times} \pi^{3\lambda_1 + \lambda_2 + 8}.$$
	
	When the level is not square-free, then the same formula holds up to an unknown factor in $E^\times$, where $E$ is the field of definition of $\pi_f$.
\end{theorem}

\begin{corollary}\label{cor:Chen--Ichino}
	If $f$ is a holomorphic Siegel modular form of weight $(k,2)$ with coefficients in $E$, and $f^W$ is the associated Whittaker-normalized generic Siegel modular form, then:
	$$\langle f^W, f^W \rangle \sim_{E^\times} \pi^{3(k-1)+5} \cdot \Lambda(f, \Ad, 1).$$
\end{corollary}
\begin{proof}
	Recall that $(\lambda_1, \lambda_2) = (k-1,0)$, so $\pi^{3 \lambda_1 + \lambda_2 + 8} = \pi^{3(k-1)+8}$. Also, we have that:
	\begin{align*}
		\Lambda(2) & = \frac{\pi^{2}}{6} \cdot \pi^{-2/2} \Gamma(2/2), \\
		\Lambda(4) & = \frac{\pi^{4}}{90} \cdot \pi^{-4/2} \Gamma(4/2).
	\end{align*}
	This gives the result.
\end{proof}

\subsection{Beilinson's conjecture for the adjoint $L$-value}

Next, we give an explicit version of Beilinson's conjecture for the symmetric square $L$-function in terms of a non-zero motivic cohomology class $\alpha \in H^{2k-1}_{\mathcal M}(M(f, \Sym^2)_\Z, \Q(k))$ and a natural generator $\delta'$, as in Definition~\ref{def:natural_gen}.

\begin{theorem}\label{thm:Beilinson_for_Sym2}
	Let $f$ be a holomorphic Siegel modular form of even weight $(k,2)$ with trivial central character, defined over $\Q$, and let $M(f)$ be the associated motive over $\Q$. Fix a polarization pairing on $M(f)$ and write $\langle -, - \rangle_{\pol}$ for the induced pairing on $M(f, \Sym^2)$. For $\alpha \in H^{2k-1}_{\cal M}(M(f, \Sym^2)_\Z, \Q(k))$, Beilinson's conjecture for the adjoint $L$-function is equivalent to the equation:
	\begin{align*}
		L'(f, \Sym^2, k-1) & \sim_{\Q^\times}  \pi^{-2(k-1)} \cdot c^+(M(f)(k-1)) \cdot c^-(M(f)(k-1)) \cdot \langle  \delta', r_{\cal D}(\alpha) \rangle_{\mathrm{pol}} 
	\end{align*}
	for a natural generator $\delta'$ (Definition~\ref{def:natural_gen}).
	
\end{theorem}
\begin{proof}	
	We first introduce the notation relevant to the motive $M(f)$ and recall Lemma~\ref{lemma:c+-(M(f))} which computed $c^\pm(M(f))$. Recall that:
	$$H^{k-1}_B(M(f), \Q) \iso H^{k-1}_B(M(f), \Q)^+ \oplus H^{k-1}_B(M(f), \Q)^-$$
	and we fix a basis $v_i^\pm$. In other words, $F_\infty$ acts by $v_i^+ \mapsto v_i^+$ and $v_i^- \mapsto - v_i^-$ for $i = 1, 2$. Note that the polarization pairing descends to a skew-symmetric pairing on $H^{k-1}_B(M(f), \Q)^\pm$ and we may assume that that basis $v_1^+, v_2^+, v_1^-, v_2^-$ is chosen so that its matrix is
	\begin{equation}\label{eqn:polarization_matrix}
		(2 \pi i)^{-(k-1)} \begin{pmatrix}
			& & 1 &   \\
			& & & 1 \\
			-1 & & &   \\
			& -1 & & \\
		\end{pmatrix}.
	\end{equation}
	
	For a basis $\omega_1, \omega_2$ of $F^{1} H^{k-1}_{\mathrm{dR}}(M(f))$, the comparison map is given by:
	\begin{align*}
		F^{1} H^{k-1}_{\mathrm{dR}}(M(f)) & \overset{\widetilde \pi_{0}}\to H^{k-1}_{B}(M(f)_\C, \R) \\
		\omega_1 & \mapsto c_{1,1}^+ v_1^+ + c_{1,2}^+ v_2^+ + c_{1,1}^- v_1^- + c_{1,2}^- v_2^-,\\
		\omega_2 & \mapsto c_{2,1}^+ v_1^+ + c_{2,2}^+ v_2^+ + c_{2,1}^- v_1^- + c_{2,2}^- v_2^-.
	\end{align*}
	Then:
	\begin{equation}\label{eqn:cpm_vs_cpm(M(f))}
		c^\pm := \det (c_{i,j}^\pm) = c^\pm(M(f)(k-1))
	\end{equation}
	as in Lemma~\ref{lemma:c+-(M(f))}.

	Throughout the rest of the proof, we will use the shorthand 
	$$M := M(f,  \Sym^2).$$ 
	Consider the Beilinson short exact sequence associated to the motive $M$:
	\begin{equation*}\label{eqn:Beilinson_ses_Thm}
		0 \to F^{k} H_{\mathrm{dR}}^{2k-2} (M_\R)  \overset{\tilde \pi_{k-1}}\to H^{2k-2}_B (M_\R, \R(k-1))  \to  H^{2k-1}_{\cal D} (M_\R, \R(k)) \to 0.
	\end{equation*}
	Beilinson's Conjecture~\ref{conj:Beilinson} says that
	\begin{equation}\label{eqn:Beilinson}
		r_{\cal D} ( H^{2k-1}_{\cal M}(M_\Z, \Q(k))) = L'(M, k-1) \cdot \frac{\det H^{2k-2}_B(M_\R, \Q(k-1))}{\det \tilde \pi_{k-1}( F^{k} H^{2k-2}_{\mathrm{dR}}(M_\Q) )} \quad \text{in }H^{2k-1}_{\cal D}(M_\R, \R(k)).
	\end{equation}
	To make this explicit, we choose bases for the various spaces. 
	
	We have that:
	\begin{equation*}
		F^{k} H_{\mathrm{dR}}^{2k-2}(M) = \Q \omega_{1,1} \oplus \Q \omega_{2,2} \oplus \Q \omega_{1,2}
	\end{equation*}
	where
	\begin{align*}
		\omega_{i,i} & = \omega_i \otimes \omega_i, & i=1,2, \\
		\omega_{1, 2} &  = \omega_1 \otimes \omega_2 + \omega_2 \otimes \omega_1.
	\end{align*}
	We identify the 10-dimensional space $H_B^{2k-2}(M_\C, \Q)$ with $\Sym^2 H^{k-1}_B(M(f)_\C, \Q)$, a quotient of $H^{k-1}_B(M(f)_\C, \Q)  \otimes H^{k-1}_B(M(f)_\C, \Q)$. The space $H^{2k-2}_B(M_\C, \Q)^{(-1)^{k-1}}$ is four-dimensional, spanned by
	\begin{align*}
		u_{1,1} & =  v_1^+ \otimes v_1^- + v_1^- \otimes v_1^+,  \\
		u_{1,2} & = v_1^+ \otimes v_2^- + v_2^- \otimes v_1^+,  \\
		u_{2,1}& =  v_2^+ \otimes v_1^- + v_1^- \otimes v_2^+, \\
		u_{2,2} & =  v_2^+ \otimes v_2^- + v_2^- \otimes v_2^+, 
	\end{align*}
	and hence $u_{i,j}(k-1) = (2 \pi i)^{k-1} u_{i,j}$ is a basis of $H^{2k-2}_B(M_\C, \Q(k-1))^+$. The polarization pairing~\eqref{eqn:polarization_matrix} induces a polarization pairing on $M$, denoted $\langle - , - \rangle_{\pol}$. Then:
	\begin{align*}
		\langle u_{i,j}, u_{k, \ell} \rangle_{\pol} & = \langle v_i^+ \otimes v_j^- + v_j^- \otimes v_i^+, v_k^+ \otimes v_\ell^- + v_\ell^- \otimes v_k^+ \rangle \\
		& = -2 (2 \pi i)^{-2(k-1)} \delta_{i\ell} \delta_{jk},
	\end{align*}
	i.e.\
	\begin{align}\label{eqn:pairing_matrix_uij}
		(\langle u_{i,j}, u_{k, \ell} \rangle_{\pol})_{(i, j), (k,\ell)} & = -2 (2 \pi i)^{-2(k-1)} \begin{pmatrix}
			1 & & & \\
			& & 1 \\
			& 1 \\
			& & & 1
		\end{pmatrix}
	\end{align}
	
	Finally, $H^{2k-1}_{\cal D} (M_\R, \R(k))$ is one-dimensional, spanned by $r_{\cal D}(\alpha)$ for an element $\alpha \in H^{2k-1}_{\mathcal M}(M, \Q(k))$.
	
	In these bases, the map $\widetilde \pi_{k-1}$ can be described as follows: 
	\begin{align}
		F^{k} H_{\mathrm{dR}}^{2k-2} (M_\Q) & \to H^{2k-2}_B(M_\R, \R(k-1))  \nonumber \\[0.3cm]
		\omega_{1,1} = \omega_1 \otimes \omega_1 & \mapsto  c_{1,1}^+ c_{1,1}^- u_{1,1} + c_{1,1}^+ c_{1,2}^- u_{1,2} + c_{1,2}^+c_{1,1}^- u_{2,1} + c_{1,2}^+ c_{1,2}^-u_{2,2}, \label{eqn:v1_uij} \\[0.3cm]
		\omega_{2,2} = \omega_2 \otimes \omega_2 & \mapsto  c_{2,1}^+ c_{2,1}^- u_{1,1} + c_{2,1}^+ c_{2,2}^- u_{1,2} + c_{2,2}^+c_{2,1}^- u_{2,1} + c_{2,2}^+ c_{2,2}^-u_{2,2}, \\
		\omega_{1,2} = \omega_1 \otimes \omega_2 + \omega_2 \otimes \omega_1  &\mapsto  (c_{1,1}^+ c_{2,1}^- + c_{2,1}^+ c_{1,1}^-) u_{1,1} +  (c_{1,1}^+ c_{2,2}^- + c_{2,1}^+ c_{1,2}^-) u_{1,2} \\
		&~  \qquad +(c_{1,2}^+ c_{2,1}^- + c_{2,2}^+ c_{1,1}^-)u_{2,1} + (c_{1,2}^+ c_{2,2}^- + c_{2,2}^+ c_{1,2}^-) u_{2,2} . \nonumber
	\end{align}
	
	In other words, in the chosen bases, the matrix of this transformation is:
	$$
	\begin{pmatrix}
		c_{1,1}^+ c_{1,1}^-   &   c_{2,1}^+ c_{2,1}^- &   c_{1,1}^+ c_{2,1}^- + c_{2,1}^+ c_{1,1}^-  \\
		c_{1,1}^+ c_{1,2}^-  & c_{2,1}^+ c_{2,2}^-   &  c_{1,1}^+ c_{2,2}^- + c_{2,1}^+ c_{1,2}^- \\
		c_{1,2}^+ c_{1,1}^-   & c_{2,2}^+ c_{2,1}^-    &  c_{1,2}^+ c_{2,1}^- + c_{2,2}^+ c_{1,1}^- \\
		c_{1,2}^+ c_{1,2}^-  &	c_{2,2}^+ c_{2,2}^- &  c_{1,2}^+ c_{2,2}^- + c_{2,2}^+ c_{1,2}^-
	\end{pmatrix}.
	$$
	
	Let $v_1 = \widetilde \pi_{k-1}(\omega_{1,1}), v_2 =  \widetilde \pi_{k-1}(\omega_{1,2}), v_3 =  \widetilde \pi_{k-1} (\omega_{2,2})$, i.e.\ the columns of the above matrix. Presumably, all of the basis vectors $u_{i,j}$ are linearly independent from $v_1, v_2, v_3$, but at this point we only know that one of them is. Let us assume it is $u_{1,2}$; the computation with any other of the vectors is entirely analogous.
	
	We then complete $v_1, v_2, v_3$ to a basis by choosing $v_4 = (2 \pi i)^{4(k-1)}  \dfrac{u_{1,2}}{c}$, where we choose $c \in \R$ to satisfy:
	$$v_1 \wedge v_2 \wedge v_3 \wedge v_4 = u_{1,1}(k-1) \wedge u_{1,2}(k-1) \wedge u_{2,1}(k-1) \wedge u_{2,2}(k-1)$$
	so that $v_1 \wedge v_2 \wedge v_3 \wedge v_4$ is a basis for $\det H^{2k-2}_B(M_\R, \Q(k-1))$. A computation of the determinant shows that:
	\begin{equation*}
		c = \det
		\begin{pmatrix}
			c_{1,1}^+ c_{1,1}^-   &   c_{2,1}^+ c_{2,1}^- &   c_{1,1}^+ c_{2,1}^- + c_{2,1}^+ c_{1,1}^- & 0  \\
			c_{1,1}^+ c_{1,2}^-  & c_{2,1}^+ c_{2,2}^-   &  c_{1,1}^+ c_{2,2}^- + c_{2,1}^+ c_{1,2}^- & 1 \\
			c_{1,2}^+ c_{1,1}^-   & c_{2,2}^+ c_{2,1}^-    &  c_{1,2}^+ c_{2,1}^- + c_{2,2}^+ c_{1,1}^- & 0 \\
			c_{1,2}^+ c_{1,2}^-  &	c_{2,2}^+ c_{2,2}^- &  c_{1,2}^+ c_{2,2}^- + c_{2,2}^+ c_{1,2}^- & 0
		\end{pmatrix} = c^+  c^- (c_{2,2}^+ c_{1,1}^- - c_{1,2}^+ c_{2,1}^-).
	\end{equation*}
	
	Finally, suppose that $r_{\cal D}(\alpha) \in H^{2k-1}_{\cal D}(M_\R, \R(k))$ lifts to an element 
	$$a_1 v_1 + \cdots + a_4 v_4 \in H^{2k-2}_B(M_\R, \R(k-1))$$
	Then Beilinson's conjecture~\eqref{eqn:Beilinson} amounts to the equation:
	\begin{align}
		a_4 & \sim_{\Q^\times } L'(M ,k-1). \label{eqn:Beilinson_a_4_L}
	\end{align}
	
	In order to pick out $a_4$, we consider a natural generator:
	$$\delta' = \delta'(\omega_1, \omega_2) = (\omega_1 \otimes \overline{\omega_2} + \overline{\omega_2} \otimes \omega_1) - (\omega_2 \otimes \overline{\omega_1} + \overline{\omega_1} \otimes \omega_2)  \in H^{2k-1}_{\mathcal D}(M_\R, \R(k))$$
	as in Definition~\ref{def:natural_gen}.
	
	We first compute $\delta'$ in terms of the basis $u_{i,j}$:
		\begin{align}
				\omega_1 \otimes \overline{\omega_2}  + 	\overline{\omega_2} \otimes \omega_1 
				&  \mapsto 
				( c_{1,1}^+ v_1^+ + c_{1,2}^+ v_2^+ + c_{1,1}^- v_1^- + c_{1,2}^- v_2^-) 
				\otimes 
				( c_{2,1}^+ v_1^+ + c_{2,2}^+ v_2^+ - c_{2,1}^- v_1^- - c_{2,2}^- v_2^-) \nonumber   \\
				& \quad +   (  c_{2,1}^+ v_1^+ + c_{2,2}^+ v_2^+ - c_{2,1}^- v_1^- - c_{2,2}^- v_2^-) 
				\otimes 
				( c_{1,1}^+ v_1^+ + c_{1,2}^+ v_2^+ + c_{1,1}^- v_1^- + c_{1,2}^- v_2^-) \nonumber \\
				\omega_2 \otimes \overline{\omega_1} + \overline{\omega_1} \otimes \omega_2 
				& \mapsto (c_{2,1}^+ v_1^+ + c_{2,2}^+ v_2^+ + c_{2,1}^- v_1^- + c_{2,2}^- v_2^- ) 
				\otimes
				(c_{1,1}^+ v_1^+ + c_{1,2}^+ v_2^+ - c_{1,1}^- v_1^- - c_{1,2}^- v_2^-) \nonumber \\
				& \quad + (c_{1,1}^+ v_1^+ + c_{1,2}^+ v_2^+ - c_{1,1}^- v_1^- - c_{1,2}^- v_2^-)
				\otimes 
				(c_{2,1}^+ v_1^+ + c_{2,2}^+ v_2^+ + c_{2,1}^- v_1^- + c_{2,2}^- v_2^- ) \nonumber \\
				\delta' & \mapsto 2(- c_{1,1}^+ c_{2,1}^- + c_{2,1}^+ c_{1,1}^- ) u_{1,1} 
				+ 2(- c_{1,1}^+ c_{2,2}^- + c_{2,1}^+ c_{1,2}^-) u_{1,2}  \label{eqn:delta_uij} \\
				& \quad + 2(- c_{1,2}^+ c_{2,1}^- + c_{2,2}^+ c_{1,1}^-) u_{2,1} 
				+ 2(- c_{1,2}^+ c_{2,2}^- + c_{2,2}^+ c_{1,2}^- ) u_{2,2}. \nonumber
			\end{align}
		
			We will check that:
			\begin{align}
				\langle \delta', v_i \rangle_{\pol} & = 0 & \text{for }i = 1,2,3, \label{eqn:delta_v_i_pairing} \\
				\langle \delta', v_4 \rangle_{\pol} & = 4 \frac{(2 \pi i)^{2(k-1)} }{c^+ c^-} \label{eqn:delta_v_4_pairing}.
			\end{align}
			The first is immediate, because $v_i \in H^{2(k-1), 0} \oplus H^{0, 2(k-1)}$ as they are images of elements $F^k H^{2k - 2}_{\dR}(M_\R)$, while $\delta' \in H^{k-1, k-1}$. As a sanity check, we may also verify equation~\eqref{eqn:delta_v_i_pairing} by direct computation; for example, for $i = 1$, we have:
			\begin{align*}
				\langle  \delta', v_1 \rangle & = \langle \delta', c_{1,1}^+ c_{1,1}^- u_{1,1} + c_{1,1}^+ c_{1,2}^- u_{1,2} + c_{1,2}^+c_{1,1}^- u_{2,1} + c_{1,2}^+ c_{1,2}^-u_{2,2} \rangle & \text{\eqref{eqn:v1_uij}} \\
				&  = 2 \Bigl\langle (- c_{1,1}^+ c_{2,1}^- + c_{2,1}^+ c_{1,1}^- ) u_{1,1} 
				+ (- c_{1,1}^+ c_{2,2}^- + c_{2,1}^+ c_{1,2}^-) u_{1,2}  \\
				& \qquad \quad + (- c_{1,2}^+ c_{2,1}^- + c_{2,2}^+ c_{1,1}^-) u_{2,1} 
				+ (- c_{1,2}^+ c_{2,2}^- + c_{2,2}^+ c_{1,2}^- ) u_{2,2},  \\
				& \qquad \qquad \quad c_{1,1}^+ c_{1,1}^- u_{1,1} + c_{1,1}^+ c_{1,2}^- u_{1,2} + c_{1,2}^+c_{1,1}^- u_{2,1} + c_{1,2}^+ c_{1,2}^-u_{2,2} \Bigr\rangle & \text{\eqref{eqn:delta_uij}} \\
				& =  -4 (2 \pi i)^{-2(k-1)} \Bigl(   (- c_{1,1}^+ c_{2,1}^- + c_{2,1}^+ c_{1,1}^- ) (c_{1,1}^+ c_{1,1}^- )
				+ (- c_{1,1}^+ c_{2,2}^- + c_{2,1}^+ c_{1,2}^-) (c_{1,2}^+ c_{1,1}^-)  \\
				& \qquad \quad + (- c_{1,2}^+ c_{2,1}^- + c_{2,2}^+ c_{1,1}^-)(c_{1,1}^+ c_{1,2}^-)
				+ (- c_{1,2}^+ c_{2,2}^- + c_{2,2}^+ c_{1,2}^- ) (c_{1,2}^+ c_{1,2}^-)    \Bigr) & \text{\eqref{eqn:pairing_matrix_uij}} \\
				& \qquad = 4 (2 \pi i)^{-2(k-1)}(c_{1,1}^+ c_{1,1}^- + c_{1,2}^+ c_{1,2}^- ) \cdot \underbrace{( c_{1,1}^+ c_{2,1}^- + c_{1,2}^+ c_{2,2}^- - c_{2,1}^+ c_{1,1}^- - c_{2,2}^+ c_{1,2}^- )}_{\langle \omega_1, \omega_2 \rangle = 0} \\
				& = 0.
			\end{align*}
			The computations for $i = 2, 3$ are similar.
			
			Next, equations~\eqref{eqn:delta_uij} and~\eqref{eqn:pairing_matrix_uij} give:
			\begin{align*}
			\langle \delta', u_{1,2} \rangle_{\pol} & =  -4 (2 \pi i)^{-2(k-1)} (-c_{1,2}^+ c_{2,1}^- + c_{2,2}^+ c_{1,1}^-).
			\end{align*}
			Therefore:
			\begin{align*}
			\langle \delta', v_4 \rangle_{\pol} & = \frac{-4 (-2 \pi i)^{-2(k-1)} (-c_{1,2}^+ c_{2,1}^- + c_{2,2}^+ c_{1,1}^-)}{(2 \pi i)^{-4(k-1)} c} \\
			& =  \frac{-4 (2 \pi i)^{2(k-1)} (-c_{1,2}^+ c_{2,1}^- + c_{2,2}^+ c_{1,1}^-)}{ c^+ c^- (c_{2,2}^+ c_{1,1}^- - c_{1,2}^+ c_{2,1}^-)} \\
			& = 4 \frac{(2 \pi i)^{2(k-1)} }{ c^+ c^-},
			\end{align*}
			proving equation~\eqref{eqn:delta_v_4_pairing}.
			
			Thus:
			\begin{align*}
				\left\langle \delta', r_{\mathcal D}(\alpha) \right\rangle_{\pol}  & = \left\langle \delta', \sum\limits_{i=1}^4 a_i v_i \right\rangle_{\pol}  \\
				& = a_4 \langle \delta',v_4 \rangle & \text{\eqref{eqn:delta_v_i_pairing}} \\
				& = \frac{4(2 \pi i)^{2(k-1)} a_4}{c^+ c^-}. & \text{\eqref{eqn:delta_v_4_pairing}}
			\end{align*}
		
			Therefore, Beilinson's conjecture is equivalent to:
			\begin{align*}
				L'(M, k-1) & \sim_{\Q^\times} a_4 &  \text{\eqref{eqn:Beilinson_a_4_L} } \\
				& \sim_{\Q^\times} \pi^{-2(k-1)} c^+ c^- \langle \delta', r_{\mathcal D}(\alpha) \rangle_{\pol} \\
				& \sim_{\Q^\times} \pi^{-2(k-1)} c^+(M(f)(k-1)) c^-(M(f)(k-1)) \langle \delta', r_{\mathcal D}(\alpha) \rangle_{\pol}. & \text{\eqref{eqn:cpm_vs_cpm(M(f))}}
			\end{align*}
			This completes the proof.
		\end{proof}

	\begin{remark}
		Theorem~\ref{thm:Beilinson_for_Sym2} is the non-critical analog of Yoshida's formulas~\eqref{eqn:Yoshida1}, \eqref{eqn:Yoshida2}. 
		
		The formulation of Beilinson's conjecture via a Poincar\'e duality pairing on Deligne cohomology is standard;  another example for non-critical values of spin $L$-functions can be found in~\cite{Cauchi_Lemma_Jacinto, Cauchi_Lemma_Jacinto_GSp(6)}.
	\end{remark}

\begin{example}\label{ex:endoscopic}
	In this extended example, we give an alternative proof of Theorem~\ref{thm:Beilinson_for_Sym2} in the endoscopic case, i.e. a Siegel modular form $f$ of weight $(k,2)$ associated with a pair $f_1$, $f_2$ of even weight $k$ modular forms with trivial central characters and rational Fourier coefficients. This also serves as a useful check and a prelude to the results of the next two sections.
	
	The alternative proof is based on the following factorization of motives:
	\begin{align*}
		M(f) & = M(f_1) \oplus M(f_2), \\
		M(f, \Sym^2) & = M(f_1, \Sym^2) \oplus M(f_2, \Sym^2) \oplus M(f_1) \otimes M(f_2),
	\end{align*}
	associated with the factorization of $L$-functions
	$$L'(f, \Sym^2, k-1) = L(f_1, \Sym^2, k-1) L(f_2, \Sym^2, k-1) L'(f_1 \times f_2, k-1).$$
	
	We give explicit forms of Beilinson's conjecture for $M(f_i)$,  $M(f_i, \Sym^2)$ and $M(f_1) \otimes M(f_2)$.
	
	\begin{enumerate}
		\item For $i = 1, 2$ we pick a basis $v_i^\pm$ of the one-dimensional space $H_B^{k-1}(M(f_i)_\C, \Q)^\pm$ and a basis $\omega_i$ of $F^{k-1} H_{\mathrm dR}^{k-1} M(f_i)$. Note that this agrees with the notation in the proof of Theorem~\ref{thm:Beilinson_for_Sym2}. Then:
		\begin{align*}
			F^{k-1} H_{\dR}^{k-1} (M(f_i)_\R) & \overset{\widetilde \pi_{k-2}}{\to} H^{k-1}_B(M(f_i)_\R, \R(k-2)) \\
			\omega_i & \mapsto c_i^+ v_i^+ + c_i^- v_i^-,
		\end{align*}
		i.e.\  $c_{i,j}^\pm = \delta_{i,j} c_i^\pm$. A rational basis of $H^{k-1}_B(M(f_i)_\R, \R(k-2))$ is given by $v_i^+(2 \pi i)^{k-2}$, and hence 
		\begin{align*}
			c^\pm(M(f_i)(1)) & = (2 \pi i)^{-(k-2)} c^\pm_i \\
			c^\pm(M(f_i)(k-1)) & = c^{\mp}_i.
		\end{align*}
		
		\item The short exact sequence~\eqref{eqn:Beilinson_ses} for $\Sym^2 M(f_i)$ is:
		\begin{align*}
			0 \to F^{k} H_{\dR}^{2k - 2}(\Sym^2 M(f_i))  & \to H_B^{2k-2}(M_\R, \R(k-1))  \to 0  \\ 
					\omega_i \otimes \omega_i & \mapsto c_i^+ c_i^- (v_i^+ \otimes v_i^- + v_i^- \otimes v_i^+),
		\end{align*}
		where we note that $H_B^{2k-2}(M_\R, \Q(k-1)) \iso H_B^{2k - 2}(M_\R, \Q)^-$ via multiplication by $(2 \pi i)^{-(k-1)}$, and the latter space is spanned by $u_{i,i} = (v_i^+ \otimes v_i^- + v_i^- \otimes v_i^+)$. Deligne's conjecture hence amounts to:
		\begin{equation*}
			L(M(f_i, \Sym^2), k-1) =  (2 \pi i)^{-(k-1)} c_i^+ c_i^-.
		\end{equation*}
	
		\item The short exact sequence~\eqref{eqn:Beilinson_ses} for $M^{1,2} := M(f_1) \otimes M(f_2)$ is:
		\begin{center}
			\begin{tikzcd}[row sep = 0.1em]
				0 \ar[r] & F^{k} H_{\mathrm{dR}}^{2k-2} (M_\R^{1,2}) \ar[r, "{\tilde \pi_{k-1}}"] & H^{2k-2}_B (M_\R^{1,2}, \R(k-1)) \ar[r] &   H^{2k-1}_{\cal D} (M_\R^{1,2}, \R(k)) \ar[r] &  0 \\
				& \omega_1 \otimes \omega_2 & u_{1,2} = v_1^+ \otimes v_2^- & \delta' := \omega_1 \otimes \overline{\omega_2} -  \overline{\omega_1} \otimes \omega_2 \\
				& 													& u_{2,1} = v_1^- \otimes v_2^+ \\
				
				& \omega_1 \otimes \omega_2 \ar[r] & c_1^+ c_2^- u_{1,2} + c_1^- c_2^+ u_{2,1}  \\
				& & 2(c_1^+ c_2^- u_{1,2} - c_1^- c_2^+ u_{2,1}) \ar[r] & \delta'
			\end{tikzcd}
		\end{center}
		We let $w_1 = c_1^+ c_2^- u_{1,2} + c_1^- c_2^+ u_{2,1}$ and $w_2 = - \frac{(2 \pi i)^{2(k-1)}}{{c_1^- c_2^+}} u_{1,2}$ 
		so that
		\begin{align*}
			w_1 \wedge w_2 = - \frac{(2 \pi i)^{2(k-1)}}{c_1^- c_2^+} c_1^- c_2^+ u_{2,1} \wedge u_{1,2} = u_{1,2}(k-1) \wedge u_{2,1}(k-1) \in \wedge^2 H^{2k-2}_B(M_\R, \Q(k-1)).
		\end{align*}
		
		Finally, let $\alpha \in H_{\mathcal M}^{2k-1}(M^{1,2}, \Q(k))$ and suppose that $r_{\mathcal D}(\alpha) \in H^{2k-1}_{\cal D} (M_\R^{1,2}, \R(k))$ lifts to $a_1 w_1 + a_2 w_2 \in H^{2k-2}_B (M_\R^{1,2}, \R(k-1))$. Then Beilinson's conjecture~\ref{conj:Beilinson} predicts that $a_2 \sim_{\Q^\times} L'(f_1 \times f_2, k-1)$. On the other hand,
		\begin{align*}
			\langle \delta', r_{\mathcal D}(\alpha) \rangle_{\pol} & = 2\langle c_1^+ c_2^- u_{1,2} - c_1^- c_2^+ u_{2,1}, a_1 v_1 + a_2 v_2 \rangle_\pol \\
			& = 2 a_2 \langle c_1^+ c_2^- u_{1,2} - c_1^- c_2^+ u_{2,1}, v_2 \rangle_\pol \\
			& = - 2 \frac{(2 \pi i)^{2(k-1)}}{c_1^- c_2^+} a_2 \langle c_1^+ c_2^- u_{1,2} - c_1^- c_2^+ u_{2,1}, u_{1,2} \rangle_\pol \\
			& \sim_{\Q^\times} a_2.
		\end{align*}
		Altogether, this shows that:
		\begin{equation*}
			L'(f_1 \times f_2, k-1) \sim_{\Q^\times} \langle \delta', r_{\mathcal D}(\alpha) \rangle_{\pol}.
		\end{equation*}
	\end{enumerate}

	The results of (1)--(3) altogether give the formula:
	\begin{align*}
		L'(f, \Sym^2, k-1) & = L(f_1, \Sym^2, k-1) L(f_2, \Sym^2, k-1) L'(f_1 \times f_2, k-1) \\
		& \sim_{\Q^\times} (2 \pi i)^{-2(k-1)} c_1^+ c_1^- c_2^+ c_2^- \langle \delta', r_{\mathcal D}(\alpha) \rangle_\pol \\ 
		& \sim_{\Q^\times}  \pi^{-2(k-1)} c^+(M(f)(k-1)) c^-(M(f)(k-1))  \langle \delta', r_{\mathcal D}(\alpha) \rangle_\pol,
	\end{align*}
	recovering the result of Theorem~\ref{thm:Beilinson_for_Sym2}.
	
	Although Theorem~\ref{thm:Beilinson_for_Sym2} is still true, the definition of the motivic action required the form $f$ to be non-endoscopic: see Remark~\ref{rmk:Yoshida_lifts} for a detailed discussion.
\end{example}

\subsection{Completing the proof}

We are now ready to complete the proof of Theorem~\ref{thm:main}.

\begin{proof}[Proof of Theorem~\ref{thm:main}]
	We just need to check that
	$$  \frac{\sqrt{\Delta_{\Ad(f)}}}{\pi^{2k}} \frac{[f^W]}{ \langle r_{\mathcal D}(\alpha), \delta \rangle_{\mathrm{pol}}}   \in H^2(X_\C, \cal E_{k,2})_f$$
	is a $\Q$-rational cohomology class. Recall from Definition~\ref{def:Whit_period} that
	$$\frac{[w_\infty f^W]}{c^W(f)} \in H^1(X, \cal E_{1,3-k})_f$$
	is a $\Q$-rational cohomology class. Since $H^2(X_\C, \cal E_{k,2})_f$ is one-dimensional, it is enough to check that:
	\begin{equation}\label{eqn:SD_rational}
		q :=  \frac{\sqrt{\Delta_{\Ad(f)}}}{\pi^{2k}}  \left \langle \frac{[f^W]}{ \langle r_{\cal D}(\alpha),\delta^\vee \rangle_{\mathrm{pol}}}, \frac{[w_\infty f^W]}{c^W(f)} \right \rangle_{\rm SD} \in \Q^\times.
	\end{equation}
	
	Recall the functional equation for the adjoint $L$-function gives:
	\begin{equation*}
		L_\infty^\ast(f, \Ad, 0) L'(f, \Ad, 0) = \sqrt{\Delta_{\Ad(f)}}  L_\infty(f, \Ad, 1) L(f,\Ad, 1)
	\end{equation*}
	where $\Delta_{\Ad(f)} \in \Q^\times$ is the adjoint conductor, i.e.\ the conductor of the adjoint Galois representation associated with $f$. Recalling that:
	$$L_\infty(f, \Ad, s) = \Gamma_\C(s+(k-1))^3 \Gamma_\C(s) \Gamma_\R(s+1)^2,$$
	we have that:
	\begin{equation}\label{eqn:functional_eqn_adjoint}
		\pi^{-3(k-1)} L'(f, \Ad, 0) \sim_{\Q^\times} \sqrt{\Delta_{\Ad (f)}} \cdot \Lambda(f, \Ad, 1)
	\end{equation}
	
	We observe that if $\alpha \in H^1_{\mathcal M}(M(f, \Ad), \Q(1))$, then $\alpha' = (2 \pi i)^{k-1} \alpha \in H^{2k-1}_{\mathcal M}(\Sym^2 M(f), \Q(k))$, and similarly we observe that $\delta' = (2 \pi i)^{(k-1)} \delta$, so
	\begin{equation}\label{eqn:reg_Ad_vs_Sym}
		\langle r_{\mathcal D}(\alpha), \delta \rangle_{\mathrm{pol}} = \langle r_{\mathcal D}(\alpha'), \delta' \rangle_{\mathrm{pol}}.
	\end{equation}

	We finally compute $q :=\frac{ \sqrt{\Delta_{\Ad(f)}}}{\pi^{2k}}  \left \langle \frac{[f^W]}{ \langle r_{\cal D}(\alpha),\delta \rangle_{\mathrm{pol}}}, \frac{[w_\infty f^W]}{c^W(f)} \right \rangle_{\rm SD}$ up to rational factors:
	\begin{align*}
		q & \sim_{\Q^\times} \pi^{-3} \frac{ \sqrt{\Delta_{\Ad(f)}}}{\pi^{2k}} \frac{\langle f^W, f^W \rangle}{  \langle r_{\cal D}(\alpha),\delta \rangle_{\mathrm{PD}}  c^W(f)} & \text{\eqref{eqn:SD}} \\
		& \sim_{\Q^\times} \frac{ \pi^{(k-1)} \sqrt{\Delta_{\Ad(f)} }  \cdot \Lambda(1, \Ad, f)}{ \langle r_{\cal D}(\alpha),\delta \rangle_{\rm pol} c^W(f)} & \text{Corollary~\ref{cor:Chen--Ichino}} \\
		& \sim_{\Q^\times}  \pi^{-2(k-1)}  \frac{L'(f, \Ad, 0)}{ \langle r_{\cal D}(\alpha),\delta \rangle_{\rm pol} c^W(f)} & \text{\eqref{eqn:functional_eqn_adjoint}} \\
		& \sim_{\Q^\times} \frac{\pi^{-4(k-1)} c^+(M(f)(k-1)) c^-(M(f)(k-1))  \langle r_{\mathcal D}(\alpha'), \delta' \rangle_{\mathrm{pol}}  }{ \langle r_{\cal D}(\alpha),\delta \rangle_{\rm pol} c^W(f)} & \text{Theorem~\ref{thm:Beilinson_for_Sym2}} \\
		& \sim_{\Q^\times} \frac{(2 \pi i)^{-4(k-1)} c^+(M(f)(k-1)) c^-(M(f)(k-1)) }{c^W(f)} & \text{\eqref{eqn:reg_Ad_vs_Sym}} \\
		& \sim_{\Q^\times} 1. & \text {Theorem~\ref{thm:cW(f)=c+c-}}
	\end{align*}
	This proves equation~\eqref{eqn:SD_rational} and hence the theorem. 
\end{proof}

\section{Yoshida lifts from real quadratic fields }\label{sec:Yoshida}


Recall from Section~\ref{sec:intro_explication} that Conjecture~\ref{conj:motivic_action} has an interpretation in terms of abelian surfaces conjecturally associated with Siegel modular forms of weight $(2,2)$, as in Section~\ref{subsec:abelian_surfaces}. In this section, we consider two special cases that arise when the abelian surface is obtained from an elliptic curve over a quadratic field. The corresponding Siegel modular forms can be obtained as theta lifts from orthogonal groups in four variables. Although we will eventually phrase our results in terms of the automorphic forms, we first explain the set up in terms of elliptic curves and abelian surfaces.

Let $F$ be a quadratic field and write $c \in \Gal(F/\Q)$ for the non-trivial automorphism. Consider an elliptic curve $E$ over $F$ such that $E^c$ is not isogenous to $E$, and the associated abelian surface $A = R_{F/\mathbb Q} E$, defined by Weil restriction of scalars.

One can then check that the motive of $A$ is identified with the restriction of scalars of the motive of $E$:
\begin{equation*}
	M := H^1(A) \iso R_{F/\Q} H^1(E).
\end{equation*}
Moreover, there is a factorization of motives:
\begin{align}\label{eqn:fact_for_dihedral}
	\Sym^2 H^1(A) \iso R_{F/\Q} \Sym^2 H^1(E) \oplus \Asai_{F/\Q} H^1(E),
\end{align}
where: 
\begin{itemize}
	\item $R_{F/\Q} \Sym^2 H^1(E)$ is the restriction of scalars of the motive $\Sym^2 H^1(E)$ over $F$ to $\Q$, realized within the disjoint union of $E \times E$ and $E^c \times E^c$, 
	\item $\Asai_{F/\Q} H^1(E)$ is the motive obtained by descending $H^1(E) \otimes H^1(E^c)$ to $\Q$ (cf.\ \cite[Sec.~4.1]{ghate1996critical}), realized within the disjoint union of $E \times E^c$ and $E^c \times E$.
\end{itemize}

At this point, the cases of real and imaginary quadratic fields diverge, as explained in Table~\ref{table:dichotomy_expanded}.
\begin{table}[h]
	\begin{tabular}{c|c|c|c}
		quadratic field $F$ & motive & value at $s = 1$ & motivic cohomology class \\ \hline & & & \\[-0.3cm]
		real & $\Sym^2 H^1(E)$  & critical & $C_i \subseteq E \times E^c$ and $\psi_i \colon C_i \to \mathbb P^1$   \\
		& $\Asai_{F/\Q} H^1(E)$ & non-critical & such that $\sum \mathrm{div}(\psi_i) = 0$ \\ \hline & & & \\[-0.3cm]
		imaginary & $\Sym^2 H^1(E)$ & non-critical & $C_i \subseteq E \times E$ and $\psi_i \colon C_i \to \mathbb P^1$  \\
		& $\Asai_{F/\Q} H^1(E)$ & critical &  such that $\sum \mathrm{div}(\psi_i) = 0$
	\end{tabular}
	\caption{
		For the motives within factorization~\eqref{eqn:fact_for_dihedral}, we indicate which have critical and non-critical $L$-values in the sense of Deligne~\cite{Deligne_Special_values}, depending on whether the quadratic field $F$ is real or imaginary quadratic. In each case, we give an explicit description of the relevant motivic cohomology class for the non-critical $L$-value in terms of higher Chow groups. }
	\label{table:dichotomy_expanded}
\end{table}
Interestingly, even though these two setups are quite different, our conjecture covers both cases simultaneously.

\begin{remark}
	Recall the explicit description of the motivic cohomology classes in terms of higher Chow groups from Section~\ref{sec:intro_explication}: $\alpha \in H^3_{\mathcal M}(A \times A, \Q(2))$ is given by a collection of irreducible divisors $D_i$ on $A \times A$ together with functions $\varphi_i$ on $D_i$ such that $\sum \mathrm{div}(\varphi_i) = 0$. In the special cases, we have that
	$$(A \times A)_F \iso E \times E^c \times E \times E^c,$$
	and there are natural motivic cohomology classes for $(A \times A)_F$ associated with the classes in Table~\ref{table:dichotomy_expanded}:
	\begin{enumerate}
		\item when $F$ is real quadratic, $D_i = C_i \times E \times E^c$ and $\varphi_i = \psi_i \circ \pi_{C_i}$,
		\item when $F$ is imaginary quadratic, $D_i = C_i \times E^c \times E^c$ and $\varphi_i = \psi_i \circ \pi_{C_i}$.
	\end{enumerate}
	Then, in each case, $\alpha$ is obtained from $\{(D_i, \varphi_i)\}$ by descent to $\Q$.
\end{remark}

We treat the case of real quadratic fields in this section and the case of imaginary quadratic fields in the next section. Henceforth, suppose $E$ is an elliptic curve over a real quadratic field~$F$. Then $E$ corresponds to a Hilbert modular form~$f_0$ of parallel weight $2$ by~\cite{freitas2015elliptic}, and the assumption that $E^c$ is not isogenous to $E$ amounts to $f_0^\sigma \neq f_0$. Therefore, we can identify $M(f_0)$ with $M = R_{F/\Q} H^1(E)$:
\begin{equation*}
	M(f_0) = R_{F/\Q} H^1(E).
\end{equation*}
Poincar\'e duality on $H^1(E)$ gives a pairing
$$\langle -, - \rangle_{\PD} \colon H^1(E)(1) \times H^1(E) \to \Q,$$
and hence there is a canonical polarization pairing on $M(f_0)$ given by:
$$\langle x, y \rangle_{\pol} = \langle x(1), y \rangle_{\PD}(-1).$$ 

The associated abelian surface $A = R_{F/\Q} E$ should correspond to a Siegel modular form $f$ of paramodular level according to Conjecture~\ref{conj:Brumer_Kramer}. This Siegel modular form was constructed in~\cite{Johnson_Schmidt}, building on the ideas of Yoshida~\cite{Yoshida:lifts, Yoshida:lifts2}. More precisely, Yoshida constructs an explicit Siegel modular form for the Siegel congruence subgroup, while Johnson-Leung and Roberts construct the desired modular form of paramodular level. They proved the desired equality of $L$-functions:
\begin{equation*}
	L(M(f), s) = L(f, s) = L(f_0, s) = L(M(f_0), s).
\end{equation*}
Under our running assumptions on motives, this gives the identification:
\begin{equation*}
	M(f) = M(f_0).
\end{equation*}

\begin{remark}
	Starting with a Hilbert modular form $f_0$ of weight $(2,2)$ with rational Fourier coefficients, we only know how to construct the associated elliptic curve $E$ over $F$ when $f_0$ transfers to a quaternion algebra over~$F$ split at a unique infinite place. For higher weight forms, the motive $M(f_0)$ was constructed by Blasius--Rogawski~\cite{Blasius_Rogawski} using other methods. On the other hand, the Asai motive $M(f_0, \Asai) = \Asai_{F/\Q} H^1(E)$ appears directly in the cohomology of the Hilbert modular surface $X_0$. Indeed, there is a Grothendieck motive $H^2(X_0)_f$, and under our running assumptions on motives, it may be lifted to a Chow motive $M(f_0, \Asai)$. To construct it in the category of Chow motives directly, one would have to prove that the Hecke idempotent associated with $f_0$ is an idempotent up to rational equivalence (which we currently only know up to homological equivalence).
\end{remark}

\begin{remark}\label{rmk:Yoshida_lifts}
	Yoshida also considers the split case $F = \Q \oplus \Q$, i.e.\ lifts a pair of classical modular forms to a Siegel modular form. However, as explained in Lemma~\ref{lemma:multiplicities}, there is no holomorphic paramodular level Siegel modular form associated with a pair of classical modular forms. More precisely, if we fix any level structure $K_f$, then lifts of a pair of modular forms will contribute to either $H^0$ or $H^1$ but not both. Therefore, any purported motivic action in these cases would need to not only change the representation at infinity but would also need to change the representation at some finite places. Since this seems to have a different nature than other motivic actions~\cite{Prasanna_Venkatesh, Horawa}, we decided not to pursue this case further here.
\end{remark}

Factorization~\eqref{eqn:fact_for_dihedral} gives the following equality of $L$-functions:
\begin{equation}
	\begin{aligned}
	L(f, \Sym^2, s) & = L(f_0, \Sym^2, s) \cdot L(f_0, \Asai, s), \\
	L(f, \Ad, s) & = L(f_0, \Ad, s) \cdot L(f_0, \Asai, s + 1). 
	\end{aligned}
	\label{eqn:ad_Asai}
\end{equation}

Using this factorization and Ramakrishnan's results~\cite{Ramakrishnan}, we can prove that the motivic action is rational without assuming part (2) of Beilinson's Conjecture~\ref{conj:Beilinson}.  Let $X_0$ be a toroidal compactification of the Hilbert modular surface over $\Q$ of level $\mathfrak N$.

\begin{theorem}\label{thm:real_quadratic}
	Let $\Pi$ be the $L$-packet on $\GSp_4$ associated with a weight $(2,2)$ Hilbert modular form $f_0$ of level $\mathfrak N$.  
	\begin{enumerate}
		\item There is a motivic cohomology class in $H^3_{\mathcal M}((X_0)_\Z, \Q(2))^\vee$ which acts rationally on $H^\ast(X, \mathcal E_{2,2})_\Pi$.
		\item Assuming Beilinson's filtration conjectures~\cite[\S2.1.10]{Prasanna_Venkatesh}, there is a motivic cohomology class in $H^1_{\mathcal M}(M(f, \Ad)_\Z, \Q(1))^\vee$ which acts rationally on $H^\ast(X, \mathcal E_{2,2})_\Pi$.
		\item Further assuming Hypothesis~\ref{hyp:reg_is_isom} that the Beilinson regulator is an isomorphism, Conjecture~\ref{conj:motivic_action} is true.
	\end{enumerate}
\end{theorem}

\begin{remark}
	As explained at the beginning of Section~\ref{sec:motives}, Beilinson's filtration conjectures~\cite[\S2.1.10]{Prasanna_Venkatesh} are a running assumption throughout the paper. For example, they are necessary for both the motive $M(f, \Ad)$ and the motivic cohomology group $H^1_{\mathcal M}(M(f, \Ad)_\Z, \Q(1))$ to be well-defined. 
	
	Nonetheless, we have stated Theorem~\ref{thm:real_quadratic} to contain the completely unconditional statement (1) and emphasized in part (2) that the filtration conjectures are needed to relate the classes in $H^3_{\mathcal M}((X_0)_\Z, \Q(2))^\vee$ to elements of $H^1_{\mathcal M}(M(f, \Ad)_\Z, \Q(1))^\vee$. Finally, in part (3), we also have to assume that the relevant motivic cohomology group has rank one to obtain the full statement of Conjecture~\ref{conj:motivic_action}.
\end{remark}

For the rest of the section, we will build up to the proof of this theorem by summarizing Ramakrishnan's results for completeness, following~\cite{kaye2016arithmetic, Ramakrishnan}. Recall that:
$$H^3_{\mathcal M}(X_0, \Q(2)) \iso \mathrm{CH}^{2,1}(X_0)$$
where the higher Chow group $\mathrm{CH}^{2,1}(X_0)$ is generated by formal $\Q$-rational sums $\sum_i a_i (C_i, \psi_i)$ where $C_i$ are closed irreducible curves on $X_{\overline \Q}$ and $\psi_i \in \O(C_i)^\times$ satisfy $\sum\limits_i a_i \mathrm{div}(\psi_i) = 0$, up to equivalence \cite[Def. III.4]{kaye2016arithmetic}. We also have an explicit definition of the Deligne cohomology group, as above:
$$H^3_{\mathcal D}(X_0, \R(2)) \iso H^{1,1}(X_{0}, \C) \cap H^2_B(X_{0, \C}, \R(1))^+.$$
Because $X_0$ is a surface, we have a natural pairing on $H^{1,1}(X_{0, \C}, \C)$ defined by
\begin{align*}
	\langle \omega_1, \omega_2 \rangle = \int\limits_{X(\C)} \omega_1 \wedge \omega_2.
\end{align*}
Then Beilinson's regulator~\eqref{eqn:Beilinson_reg} is the map
\begin{align*}
	r_{\mathcal D} \colon H^3_{\mathcal M}(X_0, \Q(2)) \otimes \R & \to H^3_{\mathcal D}(X_0, \R(2))
\end{align*}
defined by the property: for $\alpha = \sum a_i (C_i, \psi_i)$
$$\langle r_{\mathcal D}(\alpha), \omega \rangle_{\rm PD} = \frac{1}{2 \pi i} \sum_i a_i \int\limits_{C_i(\C)} \log|\psi_i| \cdot  \omega|_{C_i}.$$
Compare this to the explicit form of Theorem~\ref{thm:Beilinson_for_Sym2}.

\begin{remark}
	In fact, we will need to use Scholl's integral subspace $H^i_{\mathcal M}(M_\Z, \Q(j))$ of motivic cohomology, and accordingly, we should be considering an integral version of the Chow groups above. We have omitted this in the exposition so far, but we will return to it shortly.
\end{remark}

When $F = \Q \oplus \Q$, the Hilbert modular surface $X_0$ is a product of two modular curves, and the Asai $L$-function recovers the Rankin--Selberg $L$-function of a pair of weight two forms. Beilinson proved his conjecture in this case by using modular units attached to the product of their product~\cite{Beilinson} (c.f.\ \cite[Theorem III.8]{kaye2016arithmetic}). The motivic cohomology classes are called {\em Beilinson--Flach elements}. However, as explained in Remark~\ref{rmk:Yoshida_lifts}, this case does not fall under the scope of our conjecture.

Instead, we consider the case where $F/\Q$ is real quadratic and use Ramakrishnan's results. He defined a $\Q$-subspace $R$ of $\mathrm{CH}^{2,1}(X_0) \otimes \Q$ generated by sums $\sum a_i(C_i, \phi_i)$ where $C_i$ are Hirzebruch--Zagier divisors on $X_0$ and $\phi_i$ are modular units on the associated modular curves. It is natural to conjecture that these cycles generate the Chow group, but as far as we know this is not currently known. 

Assuming Beilinson's filtration conjectures~\cite[\S2.1.10]{Prasanna_Venkatesh}, we may further project $R$ onto the $f_0$-isotypic component of the motivic cohomology group $H^3_{\mathcal M}(H^2(X_0), \Q(2))$ to obtain a rank one subspace $R_{f_0} \subseteq H^3_{\mathcal M}(M(f_0, \Asai), \Q(2))$.

As in Theorem~\ref{thm:Beilinson_for_Sym2}, we would like to express Beilinson's conjecture in terms of a canonical generator. Recall the Hodge structure on $H^2(X_0)$ described by Oda~\cite[Ch.\ I]{Oda_Periods}. Given a Hilbert modular form $f_0$, there are two associated classes in $H^{1,1}(X_0)$ given by:
\begin{align*}
	\eta_{f_0,1} & = (2 \pi i)^2 f(\epsilon_1 z_1, \epsilon_2 \overline{z_2}) dz_1 \wedge d\overline{z_2}, \\
	\eta_{f_0,2} & = (2 \pi i)^2 f(\epsilon_2 \overline{z_1}, \epsilon_2 z_2) dz_1 \wedge d\overline{z_2}, 
\end{align*}
assuming that there exists a unit $\epsilon \in \O_F^\times$ such that $\epsilon_1 > 0$ and $\epsilon_2 < 0$. Even if such a unit does not exist, there are two similarly defined classes~\cite{Harris_periods_I}, but their classical description is more complicated.

\begin{definition}
	A {\em natural generator} of $H^{3}_{\mathcal D}(M(f_0, \Asai), \R(2))$ is:
	$$\eta = \omega^{\sigma} \otimes \overline{\omega^{\sigma^c}} - \overline{\omega^\sigma} \otimes \omega^{\sigma^c} \in H^{1,1}(M(f_0, \Asai))^-,$$
	where $\omega^\sigma \in F^1 H_{\mathrm{dR}}M(f_0)$ and $\omega^{\sigma^c} \in F^1 H_{\mathrm{dR}}M(f_0)^c$. Identifying $M(f_0, \Asai)$ with the motive $H^2(X_0)_f$ obtained from the Hilbert modular surface, we have that:
	$$\eta = \eta_{f_0, 2} - \eta_{f_0, 1}\in H^{1,1}(X_0)_f^{-}.$$
\end{definition}

\begin{theorem}[Ramakrishnan]\label{thm:Ramakrishnan}
	\leavevmode
	\begin{enumerate}
		\item The subspace $R$ belongs to $H^3_{\mathcal M}((X_0)_\Z, \Q(2))$.
		\item Let $f_0$ be a cuspidal Hilbert modular form form of weight $(2,2)$ such that $f_0 \neq f_0^c$ and let $\eta$ be the natural generator of $H^3_{\mathcal D}(M(f_0, \Asai)_\R, \R(2))$. Then for some $\alpha = \sum\limits_i a_i (C_i, \psi_i) \in R$, we have that: 
		$$L'(f_0, \Asai, 1) = \langle r_{\mathcal D}(\alpha), \eta \rangle_{\mathrm{PD}} = \frac{1}{2 \pi i} \sum_i a_i \int\limits_{C_i(\C)} \log|\psi_i| \cdot \eta|_{C_i}.$$
	\end{enumerate}
\end{theorem}
\begin{proof}
	Part (1) can be proved by picking an integral model of the Hilbert modular variety~$X_0$ following Rapoport~\cite{rapoport1978compactifications} and verifying that both the Hirzebruch--Zagier divisors and the modular units on them extend to this integral model. See also~\cite{Kings:Higher_reg} for higher weight Hilbert modular forms. 
	
	For a statement of part (2), see Ramakrishnan~{\cite[Prop. 12.30]{Ramakrishnan}}. Since the details of the proof have not appeared in the literature, see also the more modern and general treatment due to Lei--Loeffler--Zerbes~\cite[Sec.~5]{LLZ:Asai--Flach}.
\end{proof}

\begin{remark}
	In fact, Beilinson's Conjecture~\ref{conj:Beilinson} for the Asai $L$-function at $s = 1$ is the statement:
	$$r_{\cal D}(R_{f_0}) = L'(f_0, \Asai, 1) \cdot \mathcal R(M(f_0, \Asai), 2, 2) \text{ as subsets of } H^3_{\mathcal D}(M(f_0, \Asai)_\R, \R(2)).$$
	We explain how this implies the statement of Theorem \ref{thm:Ramakrishnan} in terms of a pairing with the natural generator, analogous to Example~\ref{ex:endoscopic}~(3) above. Recall that $\sigma, \sigma^c \colon F \hookrightarrow \R$ are the two real embeddings of $F$, and we choose bases
	\begin{align*}
		v_1^\pm & \in H_B^\sigma(M(f_0), \Q)^\pm,  \\
		v_2^\pm & \in H_B^{\sigma^c}(M(f_0), \Q)^\pm,  \\
		\omega & \in F^{1} H_{\dR}(M(f_0)), & \text{(basis of $F$-vector space)}.
	\end{align*}
	Under the comparison isomorphisms, $\omega^\sigma \mapsto c_1^+ v_1^+ + c_1^- v_1^-$ and $\omega^{\sigma^c} \mapsto c_2^+ v_2^+ + c_2^- v_2^-$.
	
	Beilinson's short exact sequence~\eqref{eqn:Beilinson_ses} for $M = M(f_0, \Asai)$ is then:
	\begin{center}
		\begin{tikzcd}[row sep = 0.1em]
			0 \ar[r] & F^{2} H_{\mathrm{dR}}^{2} (M_\R) \ar[r, "{\tilde \pi_{1}}"] & H^{2}_B (M_\R, \R(1)) \ar[r] &   H^{3}_{\cal D} (M, \R(2)) \ar[r] &  0 \\
			& \omega \otimes \omega^c & u_{1,2} = v_1^+ \otimes v_2^- & \eta :=  \omega^\sigma \otimes \overline{\omega^{\sigma^c}} -  \overline{\omega^\sigma} \otimes \omega^{\sigma^c}  \\
			& 													& u_{2,1} = v_1^- \otimes v_2^+ \\
			
			& \omega \otimes \omega^c \ar[r] & c_1^+ c_2^- u_{1,2} + c_1^- c_2^+ u_{2,1}  \\
			& & 2(c_1^+ c_2^- u_{1,2} - c_1^- c_2^+ u_{2,1}) \ar[r] & \eta.
		\end{tikzcd}
	\end{center}
	
	We let $w_1 = c_1^+ c_2^- u_{1,2} + c_1^- c_2^+ u_{2,1}$ and $w_2 = - \frac{(2 \pi i)^{2}}{{c_1^- c_2^+}} u_{1,2}$ 
	so that
	\begin{align*}
		w_1 \wedge w_2 = - \frac{(2 \pi i)^{2}}{c_1^- c_2^+} c_1^- c_2^+ u_{2,1} \wedge u_{1,2} = u_{1,2}(1) \wedge u_{2,1}(1) \in \wedge^2 H^{2}_B(M_\R, \Q(1)).
	\end{align*}

	For $\alpha \in H^3_{\mathcal M}(M_\Z, \Q(2))$ as in the statement, we consider its regulator $r_{\mathcal D}(\alpha) \in H^3_{\mathcal D}(M_\R, \R(2))$ and its lift to some $a_1 w_1 + a_2 w_2 \in H^{2}_B (M_\R, \R(1))$. Then Ramakrishnan's Theorem~\ref{thm:Ramakrishnan}~(2) shows that $a_2 \sim_{\Q^\times} L'(f_0, \Asai, 1)$. On the other hand:
	\begin{align*}
		\langle \eta, r_{\mathcal D}(\alpha) \rangle_{\pol} & = 2\langle c_1^+ c_2^- u_{1,2} - c_1^- c_2^+ u_{2,1}, a_1 v_1 + a_2 v_2 \rangle_\pol \\
		& = 2 a_2 \langle c_1^+ c_2^- u_{1,2} - c_1^- c_2^+ u_{2,1}, v_2 \rangle_\pol \\
		& = - 2 \frac{(2 \pi i)^{2}}{c_1^- c_2^+} a_2 \langle c_1^+ c_2^- u_{1,2} - c_1^- c_2^+ u_{2,1}, u_{1,2} \rangle_\pol \\
		& \sim_{\Q^\times} a_2.
	\end{align*}
	Therefore, we recover the statement of Theorem~\ref{thm:Ramakrishnan}.
\end{remark}

For completeness, we compare the natural generators of the Deligne cohomology groups.

\begin{lemma}
	Under the natural isomorphism:
	$$d \colon H^3_{\mathcal D}(M(f_0, \Asai)_\R, \R(2)) \to H^3_{\mathcal D}(\Sym^2(M(f)), \R(2))$$
	the natural generator $\eta \in H^3_{\mathcal D}(M(f_0, \Asai)_\R, \R(2))$ maps to
	$$d(\eta) = \sqrt{D}^{-1}\delta$$
	where $\delta \in H^1_{\mathcal D}(M(f, \Ad)_\R, \R(1))$ is a natural generator (Definition~\ref{def:natural_gen}).
\end{lemma}
\begin{proof}
	Recall that under the identification $M(f_0) = M(f)$, a choice of basis $\omega_1, \omega_2$ of $F^1 H^1_{\dR}(M(f)_\Q)$ is:
	\begin{align*}
		\omega_1 & = \omega^\sigma \\
		\omega_2 & = \sqrt{D} \omega^\sigma.
	\end{align*}
	Then, under the identification of $M(f_0, \Asai)$ as a submotive of $\Sym^2 M(f)$, 
	\begin{align*}
		\delta & = (\omega_1 \otimes \overline{\omega_2} + \overline{\omega_2} \otimes \omega_1) - (\omega_2 \otimes \overline{\omega_1} + \overline{\omega_1} \otimes \omega_2) \\
		& = (\omega^\sigma \otimes \overline{\sqrt{D} \omega^{\sigma^c}} + \overline{\sqrt{D} \omega^{\sigma^c}} \otimes \omega^\sigma) - (\sqrt{D} \omega^{\sigma^c} \otimes \overline{\omega^\sigma} + \overline{\omega^\sigma} \otimes \sqrt{D} \omega^{\sigma^c}) \\
		& = \sqrt{D} \eta,
	\end{align*}
	as claimed.
\end{proof}

Together with the factorization~\eqref{eqn:ad_Asai}, it is now clear that Theorem~\ref{thm:Ramakrishnan} is equivalent to our explicit form of Beilinson's conjecture in Theorem~\ref{thm:Beilinson_for_Sym2}. The ``critical part'' of the period is given by the Petersson norm of $f_0$:
\begin{equation*}
	c^W(f) \approx c^+(M(f)) c^-(M(f)) \approx \langle f_0, f_0 \rangle
\end{equation*}
where the notation $\approx$ indicates equality up to rational factors and powers of $\pi$.

We finally deduce Theorem~\ref{thm:real_quadratic}.

\begin{proof}[Proof of Theorem~\ref{thm:real_quadratic}]
	It is enough to prove (1), since (2) and (3) are formal consequences once we assume the relevant conjectures about motives and motivic cohomology. We have shown in Theorem~\ref{thm:main} that Conjecture~\ref{conj:motivic_action} is true under two assumptions:
	\begin{enumerate}
		\item Deligne's conjecture for $L(M(f), s)$ at the central point $s = 1$,
		\item Beilinson's conjecture for $L(M(f, \Ad), s)$ at the point $s = 1$.
	\end{enumerate}
	We need to check these two conditions for $M(f) = M(f_0)$ where $f$ is the paramodular Yoshida lift associated with $f_0$. 
	
	Part (1) is classical: see~\cite{Shimura_HMF, Harris_periods_I}. Therefore, it is enough to prove part (2). Thanks to factorization~\eqref{eqn:ad_Asai}, it is enough to prove Beilinson's conjecture for the adjoint $L$-function of $f_0$ and the Asai $L$-function of $f_0$:
	\begin{itemize}
		\item The $L$-value at $s = 1$ of $L(f_0,\ad, s)$ is critical and explicitly related to the Petersson inner product $\langle f_0, f_0 \rangle$ (e.g.~\cite[Prop. 6.6]{Ichino_Prasanna_periods}).
		\item The $L$-value at $s = 2$ of $L(f_0, \Asai, s)$ is non-critical and Beilinson's conjecture for this $L$-function was proved by Ramakrishnan~\cite{Ramakrishnan}; see Theorem~\ref{thm:Ramakrishnan} above.
	\end{itemize}
	This completes the proof.
\end{proof}

\section{Yoshida lifts from imaginary quadratic fields}\label{sec:imag_quad}


Suppose $E$ is a modular elliptic curve over an imaginary quadratic field. Under technical assumptions, Caraiani--Newton~\cite{caraiani2023modularity} prove that $E$ is modular, i.e. there is an associated Bianchi cusp form $f_0$ of weight~$2$, building on the potential modularity result in ~\cite{10author}.

As in the previous section, we assume that $E^c$ is not isogenous to $E$ where $\langle c \rangle = \Gal(F/\Q)$, i.e.\ $f_0^c \not\iso f_0$. The associated abelian surface $A = R_{F/\Q} E$ should correspond to a Siegel modular form $f$ of paramodular level according to Conjecture~\ref{conj:Brumer_Kramer}; these may be constructed explicitly using a Yoshida-type lifts from $O(3,1)$ to $\Sp(4)$ which we will discuss in the next section.

Factorization~\eqref{eqn:fact_for_dihedral} may be written as:
\begin{equation*}
	M(f, \Ad) \iso M(f_0, \Ad) \oplus M(f_0, \Asai)(1),
\end{equation*}
which gives a commutative diagram:
\begin{center}
	\begin{tikzcd}
		H^1_{\mathcal M}(M(f, \Ad)_\Z, \Q(1)) \ar[r, "\iso"] \ar[d, "r_{\mathcal D}"] & H^1_{\mathcal M}(M(f_0, \Ad)_\Z, \Q(1))  \oplus H^3_{\mathcal M}(M(f_0, \Asai)_\Z, \Q(2))  \ar[d, "r_{\mathcal D} \oplus r_{\mathcal D}"]  \\
		H^1_{\mathcal D}(M(f, \Ad)_\R, \R(1)) \ar[r, "\iso"] & H^1_{\mathcal D}(M(f_0, \Ad)_\R, \R(1))  \oplus H^3_{\mathcal D}(M(f_0, \Asai)_\R, \R(2)).
	\end{tikzcd}
\end{center}

A simple computation shows that:
\begin{equation}\label{eqn:H3=0_for_Bianchi_Asai}
	H^3_{\mathcal D}(M(f_0, \Asai)_\R, \R(2)) = 0.
\end{equation}
Therefore, we get an isomorphism:
\begin{equation}\label{eqn:iso_Deligne_Bianchi}
	d \colon H^1_{\mathcal D}(M(f, \Ad)_\R, \R(1)) \overset\iso\to H^1_{\mathcal D}(M(f_0, \Ad)_\R, \R(1)).
\end{equation}
Assuming Hypothesis~\ref{hyp:reg_is_isom} that the regulator map is an isomorphism, equation~\eqref{eqn:H3=0_for_Bianchi_Asai} implies that
\begin{equation*}
	H^3_{\mathcal M}(M(f_0, \Asai)_\Z, \Q(2)) = 0.
\end{equation*}
Under this assumption, we hence get an isomorphism
\begin{equation}\label{eqn:iso_motivic_Bianchi}
	m \colon H^1_{\mathcal M}(M(f, \Ad)_\Z, \Q(1)) \overset\iso\to H^1_{\mathcal M}(M(f_0, \Ad)_\Z, \Q(1)).
\end{equation}

Let $X_0 = \Gamma_0(\mathfrak N) \backslash \mathcal H_3$ be the Bianchi threefold of level $\mathfrak N \subseteq \O_K$. Our target theorem is the following.

\begin{theorem}\label{thm:HP_implies_PV}
	Let $f_0$ be a Bianchi modular form of weight $(2, 2)$, level $\mathfrak N$, with trivial character and rational Fourier coefficients, and let $f$ be the associated Siegel modular form. Under Hypothesis~\ref{hyp:reg_is_isom}, our Conjecture~\ref{conj:motivic_action} implies Conjecture~\ref{conj:PV_for_Bianchi} which is an explicit form of the conjecture of Prasanna--Venkatesh~\cite{Prasanna_Venkatesh}. More precisely, we have the following two statements.
	
	\begin{enumerate}
		\item For $i = 1, 2$, there is a rational map
		\begin{align*}
			\theta_i \colon H^i(X_0, \Q)_{f_0} & \to H^{i-1}(X, \E_{2,2})_f
		\end{align*}
		such that, under the natural isomorphism~\eqref{eqn:iso_Deligne_Bianchi} of dual Deligne cohomology groups:
		\begin{align*}
			d^\vee \colon H^1_{\mathcal D}(M(f_0, \Ad), \R(1))^\vee \to H^1_{\mathcal D}(M(f, \Ad), \R(1))^\vee,
		\end{align*}
		the diagram
		\begin{center}
			\begin{tikzcd}
				H^1(X_0, \Q)_{f_0} \otimes \R \ar[d, "\eta \ast"] \ar[r, "\theta_1"] & H^0(X, \E_{2,2})_f \otimes \R \ar[d, "d^\vee(\eta) \ast"] \\
				H^2(X_0, \Q)_{f_0} \otimes \R \ar[r, "\theta_2"] & H^1(X, \E_{2,2})_f \otimes \R
			\end{tikzcd}
		\end{center}
		commutes for any $\eta \in H^1_{\mathcal D}(M(f_0, \Ad), \R(1))$. 
		
		\item Assuming Hypothesis~\ref{hyp:reg_is_isom}, under the natural isomorphism~\eqref{eqn:iso_motivic_Bianchi} of dual motivic cohomology groups:
		\begin{align*}
			m^\vee \colon H^1_{\mathcal M}(M(f_0, \Ad), \Q(1))^\vee \to H^1_{\mathcal M}(M(f, \Ad), \Q(1))^\vee,
		\end{align*}
		we have a commutative diagram:
		\begin{center}
			\begin{tikzcd}
				H^1(X_0, \Q)_{f_0} \ar[d, "\alpha \ast"] \ar[r, "\theta_1"] & H^0(X, \E_{2,2})_f  \ar[d, "m^\vee(\alpha)\ast"] \\
				H^2(X_0, \Q)_{f_0} \otimes \R \ar[r, "\theta_2"] & H^1(X, \E_{2,2})_f \otimes \R
			\end{tikzcd}
		\end{center}
		for any $\alpha \in H^1_{\mathcal M}(M(f_0, \Ad), \Q(1))^\vee$. Therefore, the rationality of the action of $\alpha$ is equivalent to the rationality of the action of $m^\vee(\alpha)$.
	\end{enumerate}
\end{theorem}

\begin{remark}
	The key point is that Theorem~\ref{thm:HP_implies_PV} is proved without assuming Beilinson's Conjecture~\ref{conj:Beilinson}. Part (1) is a statement only about the action of the Deligne cohomology groups, so it makes sense to formulate it without even assuming that the motivic cohomology groups have rank 1 (i.e.\ without Hypothesis~\ref{hyp:reg_is_isom}). Part (2) then follows from part (1) after imposing this hypothesis, but still without assuming the rationality statement of Beilinson's Conjecture~\ref{conj:Beilinson}~(2). While both conjectures are implied by the rationality statement in Beilinson's conjecture for $L(f, \Ad, 1) = L(f_0, \Ad, 1) L(f_0, \Asai, 2)$, we instead give a direct relationship between the conjectures.
\end{remark}

\begin{remark}
	A similar theorem should be true for Bianchi modular forms $f_0$ of any parallel weight $(k,k)$. We decided to not pursue this here, because the conjecture in~\cite{Prasanna_Venkatesh} is not stated for cohomology of local systems.
\end{remark}

\subsection{Theta lifts from $\GO(3,1)$ to $\GSp_4$}

We follow~\cite{HST, BergerDembelePacettiSengun} to construct both a holomorphic Siegel modular form $f$ and a Whittaker normalized Siegel modular form $f^W$ of paramodular level associated with a Bianchi cusp form $f_0$. 

\begin{theorem}[{Harris--Soudry--Taylor~\cite{HST}, Berger--Demb{\'e}l{\'e}--Pacetti--{\c{S}}eng{\"u}n \cite[Theorem 4.1]{BergerDembelePacettiSengun}}]\label{thm:theta_lifts_im_quad}
	Let $F/\Q$ be an imaginary quadratic field of discriminant $D$ and let $\mathfrak N$ be an ideal of $\mathcal O_F$. Let $f_0$ be a Bianchi modular form of level $\mathfrak N$ and weight $(k, k)$ for some $k \geq 2$ even, which is not Galois invariant. Then there exists a holomorphic Siegel modular form $f$ of weight $(k,2)$ and paramodular level $N = D^2 N_{F/\Q} \mathfrak N$ with Hecke eigenvalues, epsilon factor, and spinor L-function determined
	explicitly by $f_0$.
\end{theorem}

Let $\Pi$ be the $L$-packet containing the automorphic representation generated by $f$. Then $\Pi$ also contains a generic representation, generated by a Whittaker-normalized vector $f^W$. In fact, both $f$ and $f^W$ can be constructed in the following uniform way using theta lifting, following \cite[Section 4]{BergerDembelePacettiSengun} for details.

Let $\sigma_F$ be an automorphic representation of $\GL_2(\A_F)$ associated with a weight $(k,k)$ Bianchi modular form $f_0$ with trivial Nebentypus. There is a quadratic space $X$ such that:
$$\mathrm{GSO}(X) \iso \GL_{2, F} \times_{\GL_{1, F}} \GL_{1, \Q}$$
and hence a representation of ${\rm GSO}(X)$ corresponds to a representation of $\GL_{2, F}$ together with an extension $\widetilde \omega$ of its central character. By choosing the trivial extension of the central character, we get an automorphic representation $\widetilde{\sigma_F}$ of $\mathrm{GSO}(X)$ associated with $\sigma_F$. Since the theta correspondence is defined for representations of $\mathrm{GO}(X)$, we further extend the representation there. We consider two extensions:
\begin{itemize}
	\item $\hat \sigma_F^+ = \bigotimes\limits_{v \text{ finite}} \sigma_v^+ \otimes \sigma_\infty^+$,
	\item $\hat \sigma_F^- = \bigotimes\limits_{v \text{ finite}} \sigma_v^+ \otimes \sigma_\infty^-$.
\end{itemize}
(see~\cite[pp. 361]{BergerDembelePacettiSengun} for the notation). We apply the theta correspondence to these representations. We know that:
\begin{itemize}
	\item for finite $v$, $\theta(\sigma_v^+)$ is the unique generic representation of $\GSp_4(\Q_v)$ with the given $L$-parameter~\cite{Roberts_Global_L-packets_GSp(2)_theta_lifts},
	\item $\theta(\sigma_\infty^-) = X_\lambda^1$ for $\lambda = (k-1,0)$,
	\item $\theta(\sigma_\infty^+) = X_{\overline \lambda}^2$ for $\lambda = (k-1, 0)$.
\end{itemize}
Moreover, we know that $\Theta(\hat \sigma_F^\pm)$ is non-vanishing by the local non-vanishing together with~\cite[Theorem 1.2]{takeda2009some}, and cuspidal by \cite[Theorem 1.3(b)]{takeda2009some}. By choosing vectors (see~\cite{BergerDembelePacettiSengun} and the references therein), we obtain Theorem~\ref{thm:theta_lifts_im_quad}.

We now discuss the contributions of the above modular forms to cohomology and hence maps induced on cohomology via the theta lifts. Given a Bianchi modular form $f_0$ of weight $(k,k)$, there are two natural contributions to the singular cohomology of a local system $\mathcal V_{k,k}^F$ on the Bianchi modular threefold $X_0$:
$$\omega^i(f_0) \in H^i_!(X_0, \mathcal V_{k,k}^F(\R))_{f_0};$$
see, for example, the explicit description in~\cite[Section 5.1]{tilouine2022integral}. Because the right-hand side is one-dimensional, we may use the natural rational structure on the right-hand side to define periods $u^i(f_0) \in \R^\times$ such that:
\begin{align*}
	\frac{\omega^i(f_0)}{u^i(f_0)} & \in H^i_!(X_0, \mathcal V_{k,k}^F(\Q))_{f_0} & i = 1,2.
\end{align*}
Moreover, define the period $d^W(f) \in \R^\times$ by normalizing $[f^W]$ to be rational:
\begin{align*}
	\frac{[f^W]}{d^W(f)} \in H^1(X_\Q, \E_{k,2})_f.
\end{align*}

\begin{remark}
	When $k = 2$, the conjecture of Prasanna-Venkatesh~\cite{Prasanna_Venkatesh} amounts to a relationship between $u^1(f_0)$, $u^2(f_0)$ and a Beilinson regulator; see Proposition~\ref{prop:Beilinson_for_AdM(f_0)} and its corollary for details. Similarly, our Conjecture~\ref{conj:motivic_action} is equivalent to $d^W(f) \sim \langle r_{\cal D}(\alpha), \delta \rangle_\pol$ for a natural generator $\delta$, which we proved is equivalent to Beilinson's Conjecture in Theorem~\ref{thm:Beilinson_for_Sym2} (under some assumptions).
\end{remark}

\begin{definition}
	We define the {\em cohomologically-normalized theta lifts} to be the maps:
	\begin{align}\label{eqn:theta}
		\theta \colon H^\ast(X_0, \mathcal V_{k,k}^F(\Q))_{f_0} & \to H^{\ast-1}(X_\Q, \E_{k,2})_f
	\end{align}
	induced by Theorem~\ref{thm:theta_lifts_im_quad} and normalized rationally, i.e.\ explicitly:
	\begin{align}
		\theta \left(  \frac{\omega^1(f_0)}{u^1(f_0)} \right) & =  [f], \label{eqn:theta_1}\\
		\theta\left(\frac{\omega^2(f_0)}{u^2(f_0)} \right) & = \frac{[f^W]}{d^W(f)}. \label{eqn:theta_2}
	\end{align}
	They are well-defined up to $\Q^\times$ (or more generally $E^\times$ if both $f$ and $f_0$ have coefficients in $E$.)
\end{definition}

\begin{remark}
	Note that the domain of our map $\theta$  in~\eqref{eqn:theta} is the singular cohomology of the Bianchi threefold, while the codomain is the (coherent) sheaf cohomology of $\mathcal E_{k, 2}$ on the Shimura variety $X$. Therefore, it is difficult to interpret $\theta$ geometrically.
\end{remark}

\subsection{Explication of the conjecture of Prasanna--Venkatesh~\cite{Prasanna_Venkatesh}}

In order to prove Theorem~\ref{thm:HP_implies_PV}, we explicate the conjecture of Prasanna--Venkatesh~\cite{Prasanna_Venkatesh} in the case of motives associated to elliptic curves over imaginary quadratic fields.

\begin{definition}\label{def:nat_gen_for_Bianchi}
	A {\em natural generator} of $H^3_{\mathcal D}(M(f_0, \Sym^2), \R(2))$ is:
	$$\eta' = i (\omega^\sigma \otimes \overline{\omega^\sigma} + \overline{\omega^\sigma} \otimes \omega^\sigma).$$
	Therefore, a  {\em natural generator} of $H^1_{\mathcal D}(M(f_0, \Ad), \R(1))$ is given by:
	$$\eta = -(2\pi) (\omega^\sigma \otimes \overline{\omega^\sigma} + \overline{\omega^\sigma} \otimes \omega^\sigma).$$
	
	Similarly to Definition~\ref{def:dual_natural_generator}, we also define the dual natural generator $\eta^\vee$ of $H^1_{\mathcal D}(M(f_0, \Ad), \R(1))$ by:
	$$\eta^\vee := \frac{\pi^2}{\sqrt{\Delta_{\Ad(f_0)}}} \langle \eta, - \rangle_{\PD}$$
\end{definition}

\begin{definition}\label{def:mot_action_Bianchi}
	The action of the dual natural generator $\eta^\vee \in H^1_{\mathcal D}(M(f_0, \Ad)_\R, \R(1))^\vee$ on singular cohomology is defined by:
	\begin{align*}
		H^1(X_0, \C)_{f_0} & \mapsto H^2(X_0, \C)_{f_0} \\
		\frac{\omega^1(f_0)}{u^1(f_0)} & \mapsto \omega^2(f_0).
	\end{align*}
\end{definition}

\begin{remark}\label{rmk:adjoint_cond_square}
	As explained in \cite[Remark 5.10]{Horawa}, we expect that the adjoint conductor $\Delta_{\Ad(f_0)}$ is a square. This is proved in \cite[Proposition 5.8]{Horawa} except when the local component of the automorphic representation of $f_0$ at a place dividing~2 is the theta lift from a ramified quadratic extension. Therefore, it is at least true that $\sqrt{\Delta_{\Ad(f_0)}} \in \Q(\sqrt{2})^\times$, and we expect it to be rational. In particular, the factor of $\sqrt{\Delta_{\Ad(f_0)}}$ is not strictly necessary when defining the natural generator, but we keep it here for consistency with Definition~\ref{def:dual_natural_generator}.
\end{remark}

\begin{conjecture}[{Prasanna--Venkatesh~\cite{Prasanna_Venkatesh}}]\label{conj:PV_for_Bianchi}
	Via the dual Beilinson regulator, the resulting action of $H^1_{\mathcal M}(M(f_0, \Ad), \Q(1))^\vee$ preserves the rational structure $H^\ast(X_0, \Q)_{f_0}$.
\end{conjecture}

\begin{remark}
	As stated, this conjecture is actually an amalgamation of the conjecture of Prasanna--Venkatesh~\cite{Prasanna_Venkatesh} and the conjecture and Cremona--Whitley~\cite{cremona1994periods}.
	
	Indeed, the original phrasing of Prasanna--Venkatesh is slightly different. Rather than considering a dual natural generator $\eta^\vee$ in Definition~\ref{def:nat_gen_for_Bianchi}, they consider the normalized element:
	$$\widetilde \eta^\vee = \frac{\eta^\vee}{h(\omega^\sigma)},$$
	where $h(\omega) = \langle \omega^\sigma, \omega^\sigma \rangle_{\mathrm{PD}}$ is Faltings' height. Cremona--Whitley~\cite{cremona1994periods} conjectured that Faltings' height is explicitly related to $u^1(f_0)$ and provided numerical evidence. Granting this relationship, the action in Definition~\ref{def:mot_action_Bianchi} has the more familiar form
	\begin{equation}\label{eqn:PV_OG}
		\widetilde \eta^\vee (\omega^1(f_0)) = \omega^2(f_0).
	\end{equation}
	The advantage of this phrasing is that $\widetilde \eta^\vee$ does not depend on the choice of N\'eron differential $\omega^\sigma$. 
	
	We offer the alternative phrasing above here to bring it closer in line with our conjecture in the case of Siegel modular forms. The reason we cannot phrase our motivic action similarly to~\eqref{eqn:PV_OG} is that the holomorphic Siegel modular form does not have a Whittaker model, and hence we are forced to normalize it using coherent cohomology directly. Therefore, the analog of $[f]$ is actually $\frac{\omega^1(f_0)}{u^1(f_0)}$ and not just $\omega^1(f_0)$. Both Conjectures~\ref{conj:motivic_action} and \ref{conj:PV_for_Bianchi} then take the form:
	$$\text{(natural generator)} \ast \text{(rational class)} = \text{(Whittaker class)}.$$
\end{remark}

We next check that this conjecture is consistent with Beilinson's conjecture for $M(f_0, \Ad)$.

\begin{proposition}\label{prop:Beilinson_for_AdM(f_0)}
	Beilinson's conjecture for $M(f_0, \Sym^2)$ is equivalent to:
	$$L'(f_0, \Sym^2, 1) \sim_{\Q^\times} \pi \cdot \langle r_{\mathcal D}(\alpha'),  \eta' \rangle_{\mathrm{PD}} \cdot u^1(f_0)$$
	for $\alpha' \in H^3_{\mathcal M}(M(f_0, \Sym^2), \Q(2))$ and a natural generator $\eta'$ from Definition~\ref{def:nat_gen_for_Bianchi}.
\end{proposition}

The proof is a formal computation with Beilinson's conjecture, similar to the proof of Theorem~\ref{thm:Beilinson_for_Sym2}, and will occupy the rest of this subsection. Before that, we state a corollary which is implicit in \cite{urban1995formes, Prasanna_Venkatesh}.

\begin{corollary}
	Beilinson's conjecture for $M(f_0, \Ad)$ implies Conjecture~\ref{conj:PV_for_Bianchi}.
\end{corollary}

\begin{proof}
	To check Conjecture~\ref{conj:PV_for_Bianchi}, we must check that for $\alpha \in H^1_{\mathcal M}(M(f_0, \Ad), \Q(1))$, the cohomology class
	$$\frac{\sqrt{\Delta_{\Ad(f_0)}}}{\pi^2} \frac{\omega^2(f_0)}{\langle r_{\mathcal D}(\alpha), \eta \rangle_{\mathrm{PD}}} \in H^2(X_0, \C)$$
	is rational. By pairing the above cohomology class with rational cohomology class $\frac{\omega^1(f_0)}{u^1(f_0)}$ under Poincar\'e duality, it is enough to verify that
	$$\langle f_0, f_0 \rangle \sim_{\Q^\times} \sqrt{\Delta_{\Ad(f_0)}}^{-1} \pi^2 u^1(f_0) \cdot \langle r_{\mathcal D}(\alpha), \eta \rangle_{\mathrm{PD}}.$$
	
	We have that:
	\begin{align*}
		\langle f_0, f_0 \rangle & \sim_{\Q^\times} \pi^{-2} L(f_0, \Ad, 1) & \text{\cite[Prop. 7.1]{urban1995formes}} \\
		& \sim_{\Q^\times} \sqrt{\Delta_{\Ad(f_0)}}^{-1} \pi L'(f_0, \Ad, 0) & \text{functional equation} \\
		& \sim_{\Q^\times} \sqrt{\Delta_{\Ad(f_0)}}^{-1}  \pi^2  \langle r_{\mathcal D}(\alpha'),  \eta'^\vee \rangle_{\mathrm{PD}} \cdot u^1(f_0) & \text{ Proposition~\ref{prop:Beilinson_for_AdM(f_0)}}, \\
		& \sim_{\Q^\times} \sqrt{\Delta_{\Ad(f_0)}}^{-1} \pi^2  \langle r_{\mathcal D}(\alpha),  \eta \rangle_{\mathrm{pol}} \cdot u^1(f_0) 
	\end{align*}
	as claimed.
\end{proof}

In the next two sections, we will prove Proposition~\ref{prop:Beilinson_for_AdM(f_0)}.

\subsubsection{The motive  $M=\Res_{F/\Q} H^1(E)$} 
We consider $M=\Res_{F/\Q} H^1(E)$ and $n=1$, where $E$ is an elliptic curve over an imaginary quadratic field $F=\Q(\sqrt{-D})$. 

We keep the notation of the previous section and write $\sigmab$ for $\sigma^c$. Let $\{ \gamma_1^c, \gamma_2^c \}$ denote the basis of 
$H_1(E^\sigmab (\C),\Z)$ obtained by applying $c$ to the basis $\{ \gamma_1, \gamma_2 \}$, and let $\{ \gt_1^c, \gt_2^c \}$ denote the dual basis of $H^1 (E^\sigmab (\C), \Z)$. Thus 
$$ \langle \gamma_i^c , \omega^\sigmab \rangle = \overline{ \langle \gamma_i, \omega^\sigma \rangle }, $$
for $i=1,2$. 

The point $n=1$ is critical and the relevant map is:
$$ F^1 H^1_{\dR} (M_\R) \rightarrow H^1_B(M_\R, \R (0)) = H^1_B (M_{\C}, \R)^+.$$
Now,
$$ F^1 H^1_{\dR} (M_\Q) = F \omega,$$
while 
$$ H^1_B (M_\C,\R) = H^1_B( E^\sigma(\C), \R) \oplus H^1_B (E^\sigmab (\C), \R).$$

Suppose that 
$$ \langle \omega^\sigma, \gamma_1 \rangle = a = a_1 + a_2 i, \quad  \langle \omega^\sigma, \gamma_2 \rangle = b = b_1 + b_2 i,$$
where the $a_i, b_i$ are in $\R$. 
Via the comparison isomorphisms:
\begin{align*} \omega^\sigma &= (a_1 + a_2 i) \gt_1 + (b_1 +b_2 i) \gt_2,  \\
	\omega^\sigmab &= (a_1 - a_2 i) \gt_1^c + (b_1 - b_2 i) \gt_2^c. \\
	(\sqrt{-D} \omega)^\sigma &= \sqrt{D} \left[ (a_1  i - a_2 ) \gt_1 + (b_1 i  - b_2) \gt_2 \right].  \\
	(\sqrt{-D} \omega)^\sigmab &= -\sqrt{D} \left[ ( a_1  i + a_2 ) \gt_1^c + (b_1 i  + b_2) \gt_2^c \right].  
\end{align*}
Thus the map 
$$ F^1 H^1_{\dR} (M_\R) \rightarrow H^1_B(M_\C, \C)$$ 
sends
\begin{align*}\omega & \mapsto a_1 (\gt_1, \gt_1^c) + a_2 i (\gt_1, - \gt_1^c) + b_1 (\gt_2, \gt_2^c) + b_2 i (\gt_2, -\gt_2^c), \\
	\sqrt{-D} \omega & \mapsto \sqrt{D} \left[ a_1 i (\gt_1, - \gt_1^c) - a_2  (\gt_1,  \gt_1^c) + b_1 i (\gt_2, - \gt_2^c) - b_2  (\gt_2, \gt_2^c). \right]
\end{align*}
Consequently, the map 
$$ F^1 H^1_{\dR} (M_\R) \rightarrow H^1_B(M_\C, \R)$$ 
sends 
\begin{align*}\omega & \mapsto a_1 (\gt_1, \gt_1^c)  + b_1 (\gt_2, \gt_2^c) , \\
	\sqrt{-D} \omega & \mapsto -\sqrt{D} \left[  a_2  (\gt_1,  \gt_1^c)  + b_2  (\gt_2, \gt_2^c). \right]
\end{align*}
Clearly, this lands in $H^1_B(M_\C, \R)^+$ and taking the determinant with respect to the bases 
$ \{ \omega, \sqrt{-D} \omega \}$ and $\{ \e_1, \e_2 \}$
where 
$$ \e_1:= (\gt_1, \gt_1^c), \quad \e_2 := (\gt_2, \gt_2^c)$$
gives:
\begin{equation}\label{eqn:c^pm(EF)}
	c^+(M(1)) = -\sqrt{D} (a_1 b_2 - a_2 b_1).
\end{equation}
One can easily also check that
\begin{equation*}
	c^-(M(1)) = -\sqrt{D} (a_1 b_2 - a_2 b_1) = c^+(M(1))
\end{equation*}
by considering the map
$$ F^1 H^1_{\dR} (M_\R) \rightarrow H^1_B(M_\C, \R)^-$$ 
which sends 
\begin{align*}
	\omega & \mapsto a_2 i (\gt_1, - \gt_1^c) + b_2 i (\gt_2, -\gt_2^c) , \\
	\sqrt{-D} \omega & \mapsto \sqrt{D} \left[ a_1 i (\gt_1, - \gt_1^c) + b_1 i (\gt_2, - \gt_2^c). \right]
\end{align*}

Deligne's conjecture for Bianchi modular forms was studied by Cremona--Whitley and Hida.

\begin{theorem}[{\cite[(2.4)]{cremona1994periods}, \cite{hida1994critical}}]\label{thm:Cremona}
	Let $f_0$ be a Bianchi modular form of weight $(2,2)$, rational coefficients, and trivial central character. Then for $\chi = 1$ and for some quadratic characters $\chi$ (see \cite[pp.\ 417--418]{cremona1994periods} for details):
	$$L(f_0, \chi,  1) \sim_{\Q^\times} \pi^2 u^1(f_0).$$
	In particular, Deligne's conjecture for $M(f_0)$ is equivalent to:
	$$c^\pm(M(f_0)(1)) \sim_{\Q^\times} \pi^2 u^1(f_0).$$
\end{theorem}

\begin{remark}
	In our notation, $\int\limits_{E(\C)} \omega \wedge \overline{\omega}$ is a totally imaginary number, because:
	$$\overline{\int\limits_{E(\C)} \omega \wedge \overline{\omega}} = \int\limits_{E(\C)} \overline{\omega \wedge \overline{\omega}} = - \int\limits_{E(\C)} \omega \wedge \overline{\omega}.$$
	We caution the reader that Cremona--Whitley~\cite{cremona1994periods} use the notation $\int\limits_{E(\C)} \omega \wedge \overline{\omega}$ to denote its imaginary part which accounts for the factor of $i$ in our phrasing of their conjecture.
\end{remark}

\subsubsection{The motive $M = M(f_0, \Asai)$} 
We consider the Asai motive $M = M(f_0, \Asai)$ and the critical value $s = 2$. We refer to Ghate~\cite{ghate1996critical} for a detailed exposition of Deligne's conjecture for Asai motives of Bianchi modular forms. We record that:
\begin{align}
	c^+(M(2)) & \sim_{\Q^\times} c^+(M) \cdot (2\pi i)^6 & \text{\cite[(11)]{ghate1996critical}} \label{eqn:Asai_twist} \\
	c^+(M) & \sim_{\Q^\times} c^+(M(f_0)) \cdot (2 \pi i)^{-2}   & \text{\cite[Prop.\ 3 and Remark 3]{ghate1996critical}} \label{eqn:c^(Asai)andc^+(RE)}\\
	c^-(M) & \sim_{\Q^\times} \frac{-4}{\sqrt{-D}} c^+(M(f_0)) & \text{\cite[Prop.\ 3]{ghate1996critical}}
\end{align}
In particular, the $c^\pm(M(2))$ periods are determined in terms of the $c^+(M(f))$ period discussed above.

The main result of Ghate's thesis, strengthened by Loeffler--Williams is the following.

\begin{theorem}[{Ghate~\cite[Theorem 1]{ghate1996critical}, Loeffler--Williams~\cite[Corollary A.10]{loeffler2020p}}]\label{thm:Ghate_LW}
	Let $f_0$ be a Bianchi modular form of weight $(2,2)$, trivial central character. Then:
	$$L(f_0, \Asai, 2) \sim_{\Q^\times} (2 \pi i)^{4} u^1(f_0).$$
	Therefore, Deligne's conjecture for the Asai motive is equivalent to:
	$$c^+(M(f_0, \Asai)(2)) \sim_{\Q^\times} (2 \pi i)^{4} u^1(f_0).$$
\end{theorem}

\begin{remark}
	Note that this is consistent with Theorem~\ref{thm:Cremona} and equations~\eqref{eqn:Asai_twist}~\eqref{eqn:c^(Asai)andc^+(RE)}:
	$$u^1(f_0) =  c^+(M(f_0)(1)) \pi^{-2} = c^+(M(f_0)) = \pi^2 c^+(M(f_0, \Asai)) = \pi^{-4} c^+(M(f_0, \Asai)(2)).$$
\end{remark}

\subsubsection{The motive $M=\Sym^2 H^1(E)$.} 

We consider $M=\Sym^2 H^1(E)$, $n=1$, where $E$ is an elliptic curve over an imaginary quadratic field $F$. 
We keep the notation of the previous examples.
The Beilinson exact sequence~\eqref{eqn:Beilinson_ses} in this case is:
$$ 0 \rightarrow F^2 H^2_{\dR} (M_\R) \rightarrow H^2_B(M_\R, \R(1)) \rightarrow H^3_{\D}(M_{\R}, \R(2)) \rightarrow 0.$$
Now,
$$ H^2_B(M_\R, \R(1)) = H^2_B (M_{\C}, \R(1))^+ \stackrel{\cdot (2\pi i)^{-1} }{\simeq} H^2_B (M_\C, \R)^-.$$
We can realize $M$ as a submotive of $\Res_{F/\Q} (H^1(E) \otimes H^1(E))$. Then 
$$ F^2 H^2_{\dR} (M) = F \cdot \omega \otimes \omega,$$
while
$$ H^2_B(M_{\C}, \R)=\Sym^2 H^1_B( E^\sigma(\C), \R) \oplus \Sym^2 H^1_B (E^\sigmab (\C), \R).$$
Via the comparison isomorphisms:
\begin{align*} (\omega \otimes \omega)^\sigma &= \left( (a_1 +a_2 i) \gt_1 + (b_1+b_2 i)\gt_2\right) \otimes 
	\left((a_1+a_2 i) \gt_1 + (b_1 + b_2 i) \gt_2\right) \\
	&= (a_1+a_2i)^2 \gt_1 \otimes \gt_1 + (a_1+a_2 i) (b_1 + b_2 i) (\gt_1 \otimes \gt_2 + \gt_2 \otimes \gt_1) + 
	(b_1+b_2 i)^2 \gt_2 \otimes \gt_2. 
\end{align*}
Likewise,
\begin{align*} (\omega \otimes \omega)^\sigmab &= (a_1-a_2i)^2 \gt_1^c \otimes \gt_1^c + (a_1-a_2 i) (b_1 - b_2 i) (\gt_1^c \otimes \gt_2^c +  \gt_2^c \otimes \gt_1^c) + \\
	& \hspace{50mm} 
	(b_1-b_2 i)^2 \gt_2^c \otimes \gt_2^c, \\
	(\sqrt{-D} \cdot \omega \otimes \omega)^\sigma & =\sqrt{D} i \cdot [  (a_1+a_2i)^2 \gt_1 \otimes \gt_1 + (a_1+a_2 i) (b_1 + b_2 i) (\gt_1 \otimes \gt_2 + \gt_2 \otimes \gt_1) + \\
	& \hspace{50mm} 
	(b_1+b_2 i)^2 \gt_2 \otimes \gt_2], \\
	(\sqrt{-D} \cdot \omega \otimes \omega)^\sigmab &= -\sqrt{D} i \cdot [ (a_1-a_2i)^2 \gt_1^c \otimes \gt_1^c + (a_1-a_2 i) (b_1 - b_2 i) (\gt_1^c \otimes \gt_2^c +  \gt_2^c \otimes \gt_1^c) \\
	& \hspace{50mm} 
	(b_1-b_2 i)^2 \gt_2^c \otimes \gt_2^c. ]
\end{align*}
Let 
\begin{align*}
	\e_{11} &= (\gt_1 \otimes \gt_1 , \gt_1^c \otimes \gt_1^c),   & & \f_{11} = (\gt_1 \otimes \gt_1 , - \gt_1^c \otimes \gt_1^c), \\
	\e_{12} &= \frac{1}{2} (\gt_1 \otimes \gt_2 + \gt_2 \otimes \gt_1, \gt_1^c \otimes \gt_2^c + \gt_2^c \otimes \gt_1^c), & & \f_{12} =\frac{1}{2} (\gt_1 \otimes \gt_2 + \gt_2 \otimes \gt_1, - \gt_1^c \otimes \gt_2^c - \gt_2^c \otimes \gt_1^c),\\
	\e_{22} &= (\gt_2 \otimes \gt_2 , \gt_2^c \otimes \gt_2^c),   & & \f_{22} = (\gt_2 \otimes \gt_2 , - \gt_2^c \otimes \gt_2^c).
\end{align*}
Then the map  
$$ F^2 H^2_{\dR} (M_\R) \rightarrow H^2_B(M_\C, \C) $$ 
is given by:
\begin{align*} \omega \otimes \omega & \mapsto (a_1^2-a_2^2) \cdot \e_{11} + 2a_1 a_2 i \cdot \f_{11} +2 (a_1b_1-a_2b_2) \cdot \e_{12} \\
	& \hspace{20mm} + 2(a_1 b_2 + a_2 b_1) i \cdot \f_{12} +(b_1^2-b_2^2) \e_{22} + 2b_1 b_2 i \cdot \f_{22}. \\
	\sqrt{-D} \cdot \omega \otimes \omega & \mapsto \sqrt{D} \cdot \left[ (a_1^2 -a_2^2) i \cdot \f_{11} -2a_1 a_2 \e_{11} 
	+ 2(a_1 b_1 -a_2 b_2) i \cdot \f_{12} \right.\\
	& \hspace{20mm}  \left. -2 (a_1 b_2 + a_2 b_1)\cdot \e_{12} +(b_1^2 -b_2^2) i \cdot \f_{22} -2 b_1 b_2 \cdot \e_{22}
	\right].
\end{align*}
Thus the map 
$$ F^2 H^2_{\dR} (M_\R) \rightarrow H^2_B(M_\C, \R (1)) $$ 
is given by:
\begin{align*} 
	\omega \otimes \omega & \mapsto \vb_1:= 2a_1 a_2 i \cdot \f_{11} + 2 (a_1 b_2 + a_2 b_1) i \cdot \f_{12}  + 2 b_1 b_2 i \cdot \f_{22}, \\
	\sqrt{-D} \cdot \omega \otimes \omega & \mapsto \vb_2:=\sqrt{D} \left[ (a_1^2 -a_2^2) i \cdot \f_{11} + 2 (a_1 b_1 - a_2 b_2) i \cdot \f_{12} + (b_1^2 - b_2^2)i \cdot \f_{22} \right].
\end{align*}
Clearly, this lands in $H^2_B (M_{\R}, \R(1)) = H^2_B(M_{\C}, \R(1))^+$, an $\R$-basis for this space being $\{ \f_{11} (1) , \f_{12} (1), \f_{22} (1)  \}$. In fact, this is a $\Q$-basis for $H^2_B (M_{\R}, \Q(1))$.

We now pick $\vb_3$ in $H^2_B(M_{\C}, \R(1))^+$ such that 
$$ \vb_1 \wedge \vb_2 \wedge \vb_3 = \f_{11} (1) \wedge \f_{12} (1) \wedge \f_{22} (1).$$
There are two obvious good choices of such a $\vb_3$: we could take $\vb_3 = \frac{1}{\alpha} \f_{11}$ or $\vb_3 = \frac{1}{\beta} \f_{22}$, with $\alpha, \beta \in \R(1)$. 

Let us work through the case $\vb_3 = \frac{1}{\alpha} \f_{11}$ for instance. 
(The other case is similar.)
Then 
\begin{align*}
	\alpha &=\frac{1}{2\pi i} \cdot \frac{\sqrt{D}}{(2\pi)^2} \cdot \left((a_1 b_2 + a_2 b_1)(b_1^2-b_2^2) -  b_1 b_2 \cdot 2 (a_1 b_1 - a_2 b_2) \right)\\ 
	&= -\frac{1}{2\pi i} \cdot \frac{\sqrt{D}}{(2\pi)^2} \cdot (b_1^2 + b_2^2) (a_1 b_2 - a_2 b_1).
\end{align*}

The statement of the conjecture is then simply that $H^3_{\M} (M_{\Z}, \Q(2))$ is rank one, generated by $\ba$ say, and 
if 
$$\frac{1}{L'(M,1) } r_{\D} (\ba)$$ 
is lifted to an element of $H^2_B(M_{\R}, \R(1))$ and expanded in the basis $\vb_1, \vb_2, \vb_3$, then the coefficient of $\vb_3$ lies in $\Q$. 

As before, we make this explicit using Poincar\'{e} duality. We may assume that $\gamma_1, \gamma_2$ have been picked so that 
$$ \langle \gamma_1, \gamma_2 \rangle_{\PD} = 1.$$
Then
$$ \langle \gt_1, \gt_1 \rangle_{\PD} =  \langle \gt_2, \gt_2 \rangle_{\PD} =0, \quad \text{while} \qquad \langle \gt_1, \gt_2 \rangle_{\PD} =\frac{1}{2\pi i}.$$
Let 
$$ \eta' = i (\omega^\sigma \otimes \overline{\omega^\sigma} + \overline{\omega^\sigma} \otimes \omega^\sigma) \in \Sym^2 H^1 (E^\sigma (\C)).$$
Since $\eta$ is totally imaginary and of type $(1,1)$, we must have
$$ \langle P_1 (\vb_1), \eta' \rangle_{\PD} = \langle P_1 (\vb_2), \eta' \rangle_{\PD} = 0,$$
where $P_1$ is just the projection onto the first component. As a check, we verify this explicitly now. 

First note that 
\begin{align*}
	\omega^\sigma \otimes \overline{\omega^\sigma} &= \left( (a_1+a_2 i) \gt_1 + (b_1 + b_2 i) \gt_2 \right) \otimes \left((a_1-a_2 i) \gt_1 + (b_1 - b_2 i) \gt_2\right) \\
	&= (a_1^2+a_2^2) \gt_1 \otimes \gt_1 + (a_1 + a_2 i) (b_1 - b_2 i) \gt_1 \otimes \gt_2 + \\ & \hspace{20mm}  (a_1-a_2i)(b_1 + b_2 i) \gt_2 \otimes \gt_1 + (b_1^2 + b_2^2) \gt_2 \otimes \gt_2.
\end{align*}
Hence
$$ \eta' = 2i  [ (a_1^2+a_2^2) \gt_1 \otimes \gt_1 +(a_1 b_1 + a_2 b_2) (\gt_1 \otimes \gt_2 + \gt_2 \otimes \gt_1) +  (b_1^2 + b_2^2) \gt_2 \otimes \gt_2 ].$$
Thus (using $\langle \gt_1 \otimes \gt_2, \gt_2 \otimes \gt_1 \rangle_{\PD} = -1 /(2\pi i)^2$),
$$
\langle P_1 (\vb_1), \eta' \rangle_{\PD} = \frac{-2}{(2\pi i )^2} \cdot \left[ b_1 b_2 \cdot (a_1^2+a_2^2)   - (a_1b_2 + a_2 b_1) \cdot (a_1 b_1 + a_2 b_2) +  a_1 a_2 \cdot (b_1^2 + b_2^2) \right] = 0,$$
and 
$$ \langle P_1 (\vb_2), \eta' \rangle_{\PD} = \frac{-\sqrt{D} }{(2\pi i )^2} \cdot \left[ (a_1^2-a_2^2) (b_1^2+b_2^2) + 2 (a_1 b_1 - a_2 b_2)(a_1 b_1 + a_2 b_2) + (b_1^2- b_2^2) (a_1^2 + a_2^2) \right] =0, $$
as expected.  On the other hand,
$$ \langle P_1 (\vb_3), \eta' \rangle_{\PD} = \frac{i}{(2\pi i)^2} \cdot\frac{1}{\alpha} (b_1^2 + b_2^2) =  \frac{-2\pi}{\sqrt{D}(a_1  b_2 - a_2 b_1)}.$$
By equation~\eqref{eqn:c^pm(EF)}, we have that $-\sqrt{D}(a_1b_2 - a_2b_1) = c^+(H^1(E)(1))$, and hence Beilinson's conjecture is equivalent to:
$$ \langle r_{\mathcal D}(\ba), \eta' \rangle_{\mathrm{PD}}  \sim_{\Q^\times} L'(\Sym^2 H^1(E), 1) \frac{2\pi}{c^+(H^1(E)(1))},$$
i.e.
$$L'(\Sym^2 H^1(E), 1) = \frac{1}{2 \pi} \cdot \langle r_{\mathcal D}(\ba), \eta' \rangle_{\mathrm{PD}} \cdot c^+(H^1(E)(1)).$$
This proves Proposition~\ref{prop:Beilinson_for_AdM(f_0)}.

Explicitly, if $\ba$ is represented by $(C_i,f_i)$ on $E\times E$, Beilinson's conjecture is equivalent to 
\begin{equation}\label{eqn:explicit_Bianchi_regulator}
	\sum_i \int_{C_i,\C} \log |f_i| \cdot \frac{1}{2} \left(p_1^* \omega \wedge p_2^* \overline{\omega} + p_1^* \overline{\omega} \wedge p_2^* \omega \right) \in \Q \cdot \frac{4 \pi^2}{\sqrt{D}(a_1 b_2 - a_2 b_1)} \cdot L'(\Sym^2(E), 1).
\end{equation}
Here we write $\omega$ instead of $\omega^\sigma$ etc. 

\subsection{Completing the proof of Theorem~\ref{thm:HP_implies_PV}}

We start by relating the natural generators of the two Deligne cohomology groups.
\begin{lemma}\label{lemma:Deligne_Asai_case}
	Under the natural isomorphism     
	$$d \colon H^1_{\mathcal D}(M(f_0, \Ad), \R(1)) \to H^1_{\mathcal D}(M(f, \Ad), \R(1)),$$
	a natural generator $\eta \in H^1_{\mathcal D}(M(f_0, \Ad)_\R, \R(1))$ maps to 
	$$d(\eta) = -2 \sqrt{D}^{-1} \delta,$$
	where $\delta \in H^1_{\mathcal D}(M(f, \Ad)_\R, \R(1))$ is a natural generator (Definition~\ref{def:natural_gen}). 
	
	Therefore, the dual generator $\delta^\vee$ from Definition~\ref{def:dual_natural_generator} is identified with $-2 \pi^2\eta^\vee$.
\end{lemma}
\begin{proof}
	Recall that under the identification $M(f_0) = M(f)$, a choice of basis $\omega_1, \omega_2$ of $F^1H^1_{\dR}(M(f)_\Q)$ is:
	\begin{align*}
		\omega_1 & = \omega \\
		\omega_2 & = \sqrt{-D} \omega
	\end{align*}
	We then compute that:
	\begin{align*}
		\delta' & =  (\omega_1 \otimes \overline{\omega_2} + \overline{\omega_2} \otimes \omega_1) - (\omega_2 \otimes \overline{\omega_1} + \overline{\omega_1} \otimes \omega_2)  \\
		& =   (\omega^\sigma \otimes \overline{\sqrt{-D} \omega^\sigma} + \overline{\sqrt{-D} \omega^\sigma} \otimes \omega^\sigma) - (\sqrt{-D} \omega^\sigma \otimes \overline{\omega^\sigma} + \overline{\omega^\sigma} \otimes \sqrt{-D} \omega^\sigma) \\
		& =  -\sqrt{-D}(\omega^\sigma \otimes \overline{\omega^\sigma} - \overline{\omega^\sigma} \otimes \omega^\sigma) + \sqrt{-D}(\omega^\sigma \otimes \overline{\omega^\sigma} + \overline{\omega^\sigma} \otimes \omega^\sigma) \\
		& = -2 \sqrt{D} \eta'.
	\end{align*}
	Therefore, 
	\begin{align*}
		\delta & = (2 \pi i) \delta' \\
		& =  -2 \sqrt{D} (2 \pi i) \eta' \\
		& = -2 \sqrt{D} \eta
	\end{align*}
	Finally, for dual generators, we have that:
	\begin{align*}
		\delta^\vee & = \frac{\pi^{4}}{\sqrt{\Delta_{\Ad(f)}}} \langle \delta, - \rangle_{\pol} \\
			& =  -2 \frac{\pi^{4} \sqrt{D}}{\sqrt{\Delta_{\Ad(f)}}} \langle \eta, - \rangle_{\pol} \\
			& = -2 \frac{\pi^4}{\sqrt{\Delta_{\Ad(f_0)}}} \langle \eta, - \rangle_{\pol}  & \Delta_{\Ad(f)} = \Delta_{\Ad(f_0)} \cdot D \\
			& = -2 \pi^2 \eta^\vee,
	\end{align*}
	as claimed.
\end{proof}

We now return to the proof of Theorem~\ref{thm:HP_implies_PV}. It suffices to prove that the diagram
\begin{center}
	\begin{tikzcd}
		H^1(X_0, \Q)_{f_0} \otimes \C \ar[d, "d^\vee(\delta^\vee) \ast"] \ar[r, "\theta_1"] & H^0(X, \E_{2,2})_f \otimes \C  \ar[d, "\delta^\vee \ast"] \\
		H^2(X_0, \Q)_{f_0} \otimes \C \ar[r, "\theta_2"] & H^1(X, \E_{2,2})_f \otimes \C
	\end{tikzcd}
\end{center}
commutes, up to $\Q^\times$. By Lemma~\ref{lemma:Deligne_Asai_case} together with $\Delta(\Ad(f)) = \Delta(\Ad(f_0)) \cdot D$, we just need to check that:
$$\delta^\vee \ast \theta_1(\omega^1(f_0)/u^1(f_0)) \sim_{\Q^\times} \pi^2 \theta_2(\eta^\vee \ast \omega^1(f_0)/u^1(f_0)).$$

We compute both sides:
\begin{align*}
	\delta \ast \theta_1(\omega^1(f_0)/u^1(f_0)) & = \delta \ast [f] & \text{\eqref{eqn:theta_1}} \\
	& = [f^W] & \text{Definition~\ref{def:motivic_action}}
\end{align*}
and
\begin{align*}
	\theta_2(\eta \ast \omega^1(f_0)/u^1(f_0)) & =  \theta_2(\omega^2(f_0)) \\
	& = \frac{u^2(f_0)}{d^W(f)} \cdot [f^W] & \text{\eqref{eqn:theta_2}}.
\end{align*}

Therefore, Theorem~\ref{thm:HP_implies_PV} is equivalent to the following period identity.

\begin{theorem}\label{thm:Bianchi_case_non_critical_parts}
	We have that:
	$$d^W(f) \sim_{\Q^\times} \pi^2 u^2(f_0).$$
\end{theorem}

\begin{proof}
	First, we describe $d^W(f)$ and $u^2(f_0)$ in terms of the more accessible periods $c^W(f)$ and $u^1(f_0)$ and Petersson inner products. 
	
	As in the proof of Theorem~\ref{thm:main}, we have that:
	\begin{align*}
		L(f, \Ad, 1) & \sim_{\Q^\times} \pi^9 \cdot \Lambda(f, \Ad, 1) \\
		& \sim_{\Q^\times} \pi^9 \pi^{-8} \cdot \langle f^W, f^W \rangle & \text{Theorem~\ref{cor:Chen--Ichino}} \\
		& \sim_{\Q^\times} \pi^9 \pi^{-8} \pi^3 \cdot c^W(f) \cdot d^W(f) & \text{\eqref{eqn:SD}}
	\end{align*}
	
	Therefore:
	\begin{align}
		L(f, \Ad, 1) & \sim_{\Q^\times} \pi^4 \cdot c^W(f_0) \cdot d^W(f)  \label{eqn:adjoint_for_Siegel}\\
		L(f_0, \Ad, 1) & \sim_{\Q^\times}  \pi^{2} \cdot u^1(f_0) \cdot u^2(f_0) & \text{\cite[Prop. 7.1]{urban1995formes}} \label{eqn:adjoint_for_Bianchi}
	\end{align}
	by relating the adjoint $L$-value to the Petersson norm, and using Serre duality and Poincar\'e duality, respectively.
	
	Next, we use the factorization:
	$$L(f, \Ad, 1) = L(f_0, \Ad, 1) L(f_0, {\rm As}, 2)$$
	and the Ghate--Loeffler--Williams Theorem~\ref{thm:Ghate_LW}:
	$$L(f_0, {\rm As}, 2) \sim_{\Q^\times} (2 \pi i)^4 u^1(f_0)$$
	to conclude that
	$$L(f, \Ad, 1) \sim_{\Q^\times} \pi^4 \cdot L(f_0, \Ad, 1) \cdot u^1(f_0).$$
	Together with equations \eqref{eqn:adjoint_for_Siegel} and \eqref{eqn:adjoint_for_Bianchi}, this shows that:
	$$\pi^{4} \cdot c^W(f) \cdot d^W(f) \sim \pi^6 \cdot u^1(f_0)^2 \cdot u^2(f_0).$$
	Finally, we have that:
	\begin{align*}
		c^W(f) & \sim_{\Q^\times}  \Lambda(f, \psi_+, 1) \Lambda(f, \psi_-, 1) & \text{Theorem~\ref{thm:c^W_and_spin_L_values}} \\
		& \sim_{\Q^\times} \pi^{-4} L(f, \psi_+, 1) L(f, \psi_-, 1) \\
		& \sim_{\Q^\times} \pi^{-4} L(f_0, \psi_+, 1) L(f_0, \psi_-, 1) \\
		& \sim_{\Q^\times}  u^1(f_0)^2 & \text{Cremona--Whitley--Hida Theorem~\ref{thm:Cremona}}.
	\end{align*}
	This shows that $d^W(f_0) \sim_{\Q^\times} \pi^2 u^2(f_0)$, as claimed.
\end{proof}

%

\appendix

\section{Representation theory of $\GSp_4(\R)$}\label{app:GSp(4,R)}


For completeness, we include a summary of results from the representation theory of $\GSp_4(\R)$ in this appendix. We follow Schmidt~\cite{Schmidt} and use the notation therein. We also include some results of~\cite{Muic}.

\subsection{(Limits of) discrete series for $\Sp_4(\R)$ and $\GSp_4(\R)$}

Let $K^\circ \iso U(2)$ be the maximal compact subgroup of $\Sp_4( \R)$ and $\g^\circ$ be its Lie algebra. A {\em representation of $\Sp_4(\R)$} is a $(\frak g^\circ, K^\circ)$-module.

The {\em roots} are elements of $(\frak h_C)' = \Hom_\C(\frak h_\C, \C)$, where $\frak h_\C$ is the Cartan subgroup of the complexified Lie algebra $\frak g_\C^\circ$ of $\Sp_4( \R)$. The root space has a natural real subspace $E \iso \R^2$ and the {\em analytically integral elements} are $\Z^2 \subseteq E$. Explicitly, we have a maximal torus $S \subseteq U(2) \subseteq \Sp_{4, \R}$, given by:
\begin{equation*}\label{eqn:torus_S}
	S(\R) = \left\{ \begin{pmatrix}
		\cos \theta_1 & & \sin \theta_1 & \\
		& \cos \theta_2 & & \sin \theta_2 \\ 
		-\sin \theta_1 & & \cos \theta_1 & \\
		& -\sin\theta_2 & & \cos \theta_2
	\end{pmatrix}  \right\}
\end{equation*}
and $(n_1, n_2) \in \Z^2$ corresponds to the map sending the above matrix to $e^{i\theta_1 n_1} e^{i \theta_2 n_2}$.

The {\em walls} are spanned by the roots and the {\em Weyl group} $W$ is generated by the reflections over the walls. The {\em compact Weyl group} $W_K$ is generated by the reflection about the orange line labeled $\ell$ in Figure~\ref{fig:root_system}.

\begin{figure}[h]
	\begin{tikzpicture}
		\fill[fill=blue!10]    (0,0) -- (3.5, 0) -- (3.5, 3.5) -- (0,0);
		\draw (2.5, 1.5) node {I};
		
		\fill[fill=green!10]    (0,0) -- (3.5, 0) -- (3.5, -3.5) -- (0,0);
		\draw (2.5,-1.5) node {II};
		
		\fill[fill=darkgreen!10]    (0,0) -- (3.5, -3.5) -- (0, -3.5) -- (0,0);
		\draw (1.5, -2.5) node {III};
		
		\fill[fill=darkblue!10]    (0,0) -- (0, -3.5) -- (-3.5, -3.5) -- (0,0);
		\draw (-1.5, -2.5) node {IV};
		
		\draw (-3.5, 0) -- (3.5, 0) node[right] {$\lambda_1$};
		\draw (0, -3.5) -- (0, 3.5) node[above] {$\lambda_2$};
		
		\foreach \x in {-3, -2, ..., 3} {
			\foreach \y in {-3, -2, ..., 3} {
				\fill (\x, \y) circle(1pt);
		} }
		
		\draw[very thick, blue, ->] (0,0) -- (2,0);
		\draw[very thick, blue, ->] (0,0) -- (-2,0);
		\draw[very thick, blue, ->] (0,0) -- (0,2);
		\draw[very thick, blue, ->] (0,0) -- (0,-2);

		\draw[very thick, blue, ->] (0,0) -- (1,1);
		\draw[very thick, blue, ->] (0,0) -- (-1,-1);
		\draw[very thick, blue, ->] (0,0) -- (-1,1);
		\draw[very thick, blue, ->] (0,0) -- (1,-1);
		
		\draw[dotted, orange, very thick] (-3.5,-3.5) -- (3.5,3.5) node[above right] {$\ell$};

	\end{tikzpicture}
	\caption{A graph of $E \iso \R^2 \supseteq \Z^2$ and the root system for $\Sp_4( \R)$.}
	\label{fig:root_system}
\end{figure}

Each $\lambda = (\lambda_1, \lambda_2) \in \Z^2$ with $\lambda_1 \geq \lambda_2$ corresponds to a $K^\circ$-type $V_{\lambda}$ with highest weight $\lambda$.  For each $\lambda$ inside the regions I, II, III, IV indicated in Figure~\ref{fig:root_system}, there is a discrete series representation of $\Sp_4( \R)$ associated to it:
\begin{enumerate}
	\item[(I)] $X_\lambda^1$ {\em holomorphic discrete series} for $\lambda$ in region I
	\item[(II)] $X_\lambda^2$ {\em generic discrete series} for $\lambda$ in region II,
	\item[(III)] $X_\lambda^3$ {\em generic discrete series} for $\lambda$ in region III,
	\item[(IV)] $X_\lambda^4$ {\em antiholomorphic discrete series} for $\lambda$ in region IV.
\end{enumerate}

The element $\lambda$ is called the {\em Harish--Chandra parameter} of $X_\lambda$. In each case, the associated {\em Blattner parameter} $\Lambda = \lambda + \delta$ where $\delta$ is described in Table~\ref{table:delta} is the highest weight of the minimal $K^\circ$-type occurring in $X_\lambda$.

\begin{table}[h]\label{table:min_K-types}
	\begin{tabular}{|c|c|c|} 
		\hline
		Region & $\delta$ & $\Lambda$ \\ \hline
		I & $(1,2)$ & $(\lambda_1 + 1, \lambda_2 + 2)$ \\
		II & $(1,0)$ & $(\lambda_1 +1 , \lambda_2)$ \\
		III & $(0, -1)$ & $(\lambda_1, \lambda_2 - 1)$ \\
		IV & $(-2, -1)$ & $(\lambda_1 - 2, \lambda_2 - 1)$ \\ \hline
	\end{tabular}
	
	\caption{The invariant $\delta = \delta_\lambda^{\mathrm{nc}} - \delta_\lambda^{c}$ for $\lambda$ in various regions of the root space.}
	\label{table:delta}
\end{table}

The non-degenerate limit of discrete series representations of $\Sp_4( \R)$ lie on the border of the regions:
\begin{enumerate}
	\item[(I, II)] For $(\lambda_1, \lambda_2) = (p,0)$, we have 
	\begin{itemize}
		\item $X_\lambda^1$: {\em holomorphic limit of discrete series},
		\item $X_\lambda^2$: {\em generic limit of discrete series}.
	\end{itemize}
	\item[(II, III)] For $(\lambda_1, \lambda_2) = (p, -p)$, we have
	\begin{itemize}
		\item $X_\lambda^2$, $X_\lambda^3$: two {\em generic limits of discrete series}
	\end{itemize}
	\item[(III, IV)] For $(\lambda_1, \lambda_2) = (0,-p)$, we have 
	\begin{itemize}
		\item $X_\lambda^3$: {\em generic limit of discrete series}.
		\item $X_\lambda^4$: {\em antiholomorphic limit of discrete series},
	\end{itemize}
\end{enumerate}

We finally consider representations of $\GSp_4( \R)$. First, write
\begin{align*}
	\Sp_4(\R)^\pm &  = \{g \in \GSp_4(\R) \ | \ \nu(g) = \pm 1\} = \Sp_4(\R) \ltimes \mathrm{diag}(1,1,-1,-1), \\
	K^\pm & = K^\circ \ltimes \mathrm{diag}(1,1,-1,-1) .
\end{align*}
A {\em representation of $\Sp_4(\R)^\pm$} is a $(\frak g^\circ, K^\pm)$-module. Then:
$$\GSp_4( \R) = \R_{>0} \times \Sp_4( \R)^\pm \text{ and } \frak g = \mathrm{Lie}(\GSp_4( \R)) = \R \oplus \frak g^\circ.$$ 
and a {\em representation of $\GSp_4(\R)$} is a $(\frak g, K^{\pm})$-module, i.e.\ a representation of $(\frak g^\circ, K^\pm)$ together with an integer $m \in \Z$ such that $m \equiv \lambda_1 + \lambda_2 + 1 \ (2)$ that accounts for the central character.

Conjugation by the element
$$w_\infty = \mathrm{diag}(1,1,-1,-1) \in \Sp_4( \R)^\pm$$
corresponds to the reflection $\lambda = (\lambda_1, \lambda_2) \mapsto \lambda' = (-\lambda_2, -\lambda_1)$. Therefore, for $\GSp_4(\R)$:
\begin{itemize}
	\item $X_\lambda^1$ and $X_{\lambda'}^4$ combine into a single representation, denoted $X_\lambda^1$; we call it a {\em holomorphic discrete series},
	\item $X_\lambda^2$ and $X_{\lambda'}^3$ combine into a single representation, denoted $X_\lambda^2$; we call it a {\em generic discrete series},
	\item there is an extra parameter $m \in \Z$ such that $m \equiv \lambda_1 + \lambda_2 + 1 \ (2)$ corresponding the central character of the representation.
\end{itemize}

We write $X_{\lambda; m}^i$ when we want to indicate the integer $m$ corresponding to the central character.

Consequently, the analogous statement is true for limits of discrete series:
\begin{itemize}
	\item for $\lambda = (p,0)$, there is a {\em holomorphic limit of discrete series} $X_\lambda^1$ and a {\em generic discrete series} $X_\lambda^2$,
	\item for $\lambda = (p,-p)$, there is one {\em generic discrete series} $X_\lambda^\times$. 
\end{itemize}

\subsection{Archimedean $L$-packets}

We identify the dual group of $\GSp_4$ with $\GSp_4( \R)$ as in \cite[Section 3]{Schmidt}. Langlands parameters are hence continuous homomorphisms
$$W_\R \to \GSp_4( \C).$$

For $\lambda = (\lambda_1, \lambda_2)$ inside region I, let $\overline \lambda = (\lambda_1, - \lambda_2)$ which is in region II, we have discrete series representation $X_\lambda^1$ (holomorphic) and $X_{\overline \lambda}^2$ (generic). Their common $L$-parameter $\varphi$ is:
\begin{align*}
	re^{i \theta} & \mapsto \mathrm{diag}(e^{i(\lambda_1 + \lambda_2) \theta}, e^{i(\lambda_1 - \lambda_2) \theta}, e^{-i(\lambda_1 + \lambda_2) \theta}, e^{-i(\lambda_1 - \lambda_2) \theta}) \\
	j & \mapsto \begin{pmatrix}
		& & (-1)^{\epsilon} \\
		& & & (-1)^\epsilon \\
		1 \\
		& 1
	\end{pmatrix}
\end{align*}
where $\epsilon = \lambda_1 + \lambda_2$.

The component group of the $L$-parameter $\varphi$ has two elements, represented by the identity and $\mathrm{diag}(1,-1, 1, -1)$. This agrees with the size of the $L$-packet above.

The limits of discrete series $X_\lambda^1$ (holomorphic) and $X_{\overline \lambda}^2$ (generic) for $\lambda = \overline \lambda = (p, 0)$ with $p > 0$ have the common $L$-parameter:
\begin{align*}
	re^{i \theta} & \mapsto \mathrm{diag}(e^{ip \theta}, e^{ip \theta}, e^{-ip \theta}, e^{-ip \theta}) \\
	j & \mapsto \begin{pmatrix}
		& & (-1)^{p} \\
		& & & (-1)^p \\
		1 \\
		& 1
	\end{pmatrix}.
\end{align*}
Note that this matches the conjugated $L$-parameter of Schmidt~\cite[(3.14)]{Schmidt}. The component group of the  $L$-parameter has two elements, corresponding to the fact that we have a two-element $L$-packet $\{X_\lambda^1, X_{\overline \lambda}^2\}$.

Finally, for $\lambda = (p, -p)$ with $p > 0$, we have one limit of discrete series representation. The component group is trivial in this case and $X_\lambda^\times$ is the only element of the $L$-packet.

We summarize the above results in the following lemma.

\begin{lemma}\label{lemma:L-packet}
	Given $\lambda = (\lambda_1, \lambda_2) \in \Z^2$ with $\lambda_1 \geq \lambda_2 \geq 0$ (in region I), let
	\begin{align*}
		\overline \lambda & = (\lambda_1, -\lambda_2) & \text{in region II}, \\
		\overline \lambda' & = (\lambda_2, -\lambda_1) & \text{in region III}, \\
		\lambda' & = (-\lambda_2, - \lambda_1) & \text{in region IV}.
	\end{align*}
	The $L$-packets for $\GSp_4(\R)$ associated to $\lambda$ contains the holomorphic discrete series $X_\lambda^1$ and the generic discrete series $X_{\overline \lambda}^2$. Moreover,
	\begin{align*}
		X_\lambda^1|_{\Sp_4(\R)} & = X_\lambda^1 \oplus X_{\lambda'}^4 \\
		X_{\overline \lambda}^2|_{\Sp_4(\R)} & = X_{\overline \lambda}^2 \oplus X_{\overline \lambda'}^3.
	\end{align*}
\end{lemma}

See Figure~\ref{fig:coherent cohomology} for the location of these Harish--Chandra parameters and the associated Blattner parameters in the root space.

\subsection{Construction of the archimedean $L$-packets}

We discuss the construction of (limits of) discrete series using parabolic induction, following~\cite{Muic}. The group $\GSp_4(\R)$ has three conjugacy classes of parabolic subgroups: {\em Borel parabolic} $B$, {\em Siegel parabolic} $P$, {\em Klingen parabolic} $Q$:
\begin{align*}
	B & = \begin{bmatrix}
		\ast & & \ast & \ast \\
		\ast & \ast & \ast & \ast \\
		& & \ast & \ast \\
		& & & \ast
	\end{bmatrix}, &
	P & = \begin{bmatrix}
		\ast & \ast & \ast & \ast \\
		\ast & \ast & \ast & \ast \\
		& & \ast & \ast \\
		& & \ast & \ast
	\end{bmatrix}, &
	Q & = \begin{bmatrix}
		\ast & & \ast & \ast \\
		\ast & \ast & \ast & \ast \\
		\ast & & \ast & \ast \\
		& & & \ast
	\end{bmatrix} 
\end{align*}
with Levi subgroups:
\begin{align*}
	M_B & = \begin{bmatrix}
		\ast & &  &  \\
		& \ast &  &  \\
		& & \ast & \\
		& & & \ast
	\end{bmatrix}  &
	M_P & = \begin{bmatrix}
		\ast & \ast & & \\
		\ast & \ast & & \\
		& & \ast & \ast \\
		& & \ast & \ast
	\end{bmatrix} &
	M_Q & = \begin{bmatrix}
		\ast & & \ast & \\
		& \ast & & \\
		\ast & & \ast & \\
		& & & \ast
	\end{bmatrix}\\
	&  \iso \GL_1 \times \GL_1 \times \GL_1,
	& & \iso \GL_2 \times \GL_1,
	& & \iso \GL_1 \times \GL_2.  
\end{align*}

Given a representation of each Levi subgroup, inflated to the parabolics, we obtain a representation of $\GSp_4(\R)$ by normalized parabolic induction.

\begin{itemize}
	\item Given characters $\chi_1, \chi_2$, and $\sigma$ of $\R^\times$, we denote by $\chi_1 \times \chi_2 \rtimes \sigma$ the representation of $\GSp_4$ obtained by normalized parabolic induction from the character
	$$\begin{bmatrix}
		b & & \ast & \ast \\
		\ast & a & \ast & \ast \\
		& & cb^{-1} & \ast \\
		& & & ca^{-1}
	\end{bmatrix} \mapsto \chi_1(a) \chi_2(b) \sigma(c).$$
	
	\item Given an admissible representation $\pi$ of $\GL_2(\R)$ and a character $\sigma$ of $\R^\times$, we denote by $\pi \rtimes \sigma$ the representation of $\GSp_4(\R)$ obtained by normalized parabolic induction from the representation:
	$$
	\begin{bmatrix}
		A & \ast \\
		& c {^t A^{-1}}
	\end{bmatrix} \mapsto \sigma(c) \pi (A)
	$$
	of the Siegel parabolic $P$.
	
	\item Given an admissible representation $\pi$ of $\GL_2(\R)$ and a character $\chi$ of $\R^\times$, we denote by $\chi \rtimes \pi$ the representation of $\GSp_4(\R)$ obtained by normalized parabolic induction from the representation:
	$$
	\begin{bmatrix}
		a & & b & \ast \\
		\ast & t & \ast & \ast \\
		c & & d & \ast \\
		& & & t^{-1}(ad - bc)
	\end{bmatrix} \mapsto \chi(t) \pi \left(\begin{bmatrix}
		a & b \\
		c & d
	\end{bmatrix}  \right)
	$$
	of the Klingen parabolic $Q$.
\end{itemize}

Mui\'c describes the $K$-types occurring in each of the parabolic inductions~\cite[{Lemma 6.1}]{Muic} and uses them to describe the composition series~\cite[Sections 9, 10, 11]{Muic} in terms of the Langlands classification of irreducible representations (see, for example,~\cite{Knapp_book}). We recall the necessary facts below.

For $k \geq 1$, let $D_k^+$ (resp.\ $D_k^-$) be the holomorphic (resp.\ antiholomorphic) discrete series for $\SL_2(\R)$ of lowest (resp.\ highest)  $K$-type $(k+1)$ (resp. $-(k+1)$). When $k = 1$, we still use the notation $D_1^+$ and $D_1^-$ for the holomorphic and antiholomorphic limits of discrete series.

Moreover, for any $\ell$, we write $\sigma_\ell$ for the character $\sigma_\ell(c) = |c|^\ell \sgn(c)^\ell$ of $\R^\times$.

For $\lambda = (\lambda_1, \lambda_2)$ for integers $\lambda_1 \geq \lambda_2 \geq 0$, we write $\lambda'$, $
\overline \lambda$, $\overline{\lambda}'$ as in Lemma~\ref{lemma:L-packet}.

\begin{theorem}[{Mui\'c}]\label{thm:Muic_Klingen}
	Let $\lambda = (\lambda_1, \lambda_2)$ for integers $\lambda_1 \geq \lambda_2 \geq 0$. Parabolic induction from the Klingen parabolic $Q \cap \Sp_4(\R)$ to $\Sp_4(\R)$ have the following composition series:
	\begin{center}
		\begin{tikzcd}
			0 \ar[r] & X_{\lambda}^1 \oplus X_{\overline \lambda}^2 \ar[r] & \sigma_{\lambda_2} \rtimes D_{\lambda_1}^+ \ar[r] & \mathrm{Lang}(\sigma_{\lambda_2} \rtimes D_{\lambda_1}^+) \ar[r] & 0, \\[-0.5cm]
			0 \ar[r] & X_{\overline \lambda '}^3 \oplus X_{\lambda'}^4 \ar[r] & \sigma_{\lambda_2} \rtimes D_{\lambda_1}^- \ar[r] & \mathrm{Lang}(\sigma_{\lambda_2} \rtimes D_{\lambda_1}^-) \ar[r] & 0,
		\end{tikzcd}
	\end{center}
	where $\mathrm{Lang}(-)$ are Langlands quotient representations.
\end{theorem}

The above results rely on the description of $K^\circ$-types occurring in the parabolic induction~\cite[Lemma 6.1]{Muic}. We indicate the $K^\circ$-types and these short exact sequences in Figure~\ref{fig:K-types}.

\begin{figure}[h]
	\begin{tikzpicture}[scale = 0.7]
		\fill[fill=darkblue!10]    (10.5, 3) -- (4,3) -- (4, 4) -- (10.5, 10.5);
		
		\fill[fill=pink!40]  (4,-1) -- (8,3) -- (4, 3);
		
		\draw[darkblue, very thick]    (10, 3) -- (4,3) -- (4, 4) -- (10, 10);
		
		\draw[darkblue, dotted, very thick] (10, 10) -- (11, 11);
		\draw[darkblue, dotted, very thick] (10, 3) -- (11, 3);
		
		\fill[fill=darkgreen!60, opacity = 0.3]  (4,-3.5) -- (4, -1) -- (10.5, 5.5) -- (10.5, -3.5);
		
		\draw[darkgreen, very thick]  (4,-3) -- (4, -1) -- (10, 5);
		
		\draw[darkgreen, dotted, very thick]  (10, 5) -- (11, 6);
		\draw[darkgreen, dotted, very thick]  (4, -3) -- (4, -4);
		
		
		\draw [pink, very thick] (4, -1) -- (4, 3);
		
		\draw (-2.5, -2.5) -- (10.5, 10.5);
		
		\draw (-2.5, 0) -- (10.5, 0);
		
		\foreach \x in {-2, -1, ..., 10} {
			\foreach \y in {-3, -2, ..., 10} {
				\fill (\x, \y) circle(1pt);
		} }
		
		\draw[very thick, blue, ->] (0,0) -- (2,0);
		\draw[very thick, blue, ->] (0,0) -- (-2,0);
		\draw[very thick, blue, ->] (0,0) -- (0,2);
		\draw[very thick, blue, ->] (0,0) -- (0,-2);

		\draw[very thick, blue, ->] (0,0) -- (1,1);
		\draw[very thick, blue, ->] (0,0) -- (-1,-1);
		\draw[very thick, blue, ->] (0,0) -- (-1,1);
		\draw[very thick, blue, ->] (0,0) -- (1,-1);
		
		
		
		\fill[darkblue, very thick] (3,1) circle(3pt) node[left] {$\lambda$};
		
		\fill[darkblue, very thick] (4,3) circle(3pt) node[left] {$\Lambda$};
		
		\draw[darkblue, dotted, thick] (3,1) -- (4,1) -- (4,3);
		
		\fill[darkgreen, very thick] (3,-1) circle(3pt) node[left] {$\overline \lambda$};
		
		\fill[darkgreen, very thick] (4,-1) circle(3pt) node[above left] {$\overline \Lambda$};
		
		\draw[darkgreen, dotted, thick] (3,-1) -- (4,-1);
		
		\draw (11, 6.5) node[darkblue] {$X_\lambda^1$};
		
		\draw (11, -1.5) node[darkgreen] {$X_{\overline \lambda}^2$};
		
		\draw (5, 1.5) node [red] {$\mathrm{Lang}$};
		
	\end{tikzpicture}
	\caption{This shaded region shows the $K$-types occurring in the parabolic induction $\sigma_{\lambda_2} \rtimes D_{\lambda_1}^+$ from the Klingen parabolic $Q$, according to~\cite[Lemma~6.1]{Muic}. The central character determines the parity of the occurring $K$-types and we do not indicate this here. We also do not indicate the multiplicities. The subrepresentations $X_\lambda^1$ and $X_{\overline \lambda}^2$ are shown in blue and green, respectively, and the Langlands quotient is shown in pink.}
	\label{fig:K-types}
\end{figure}
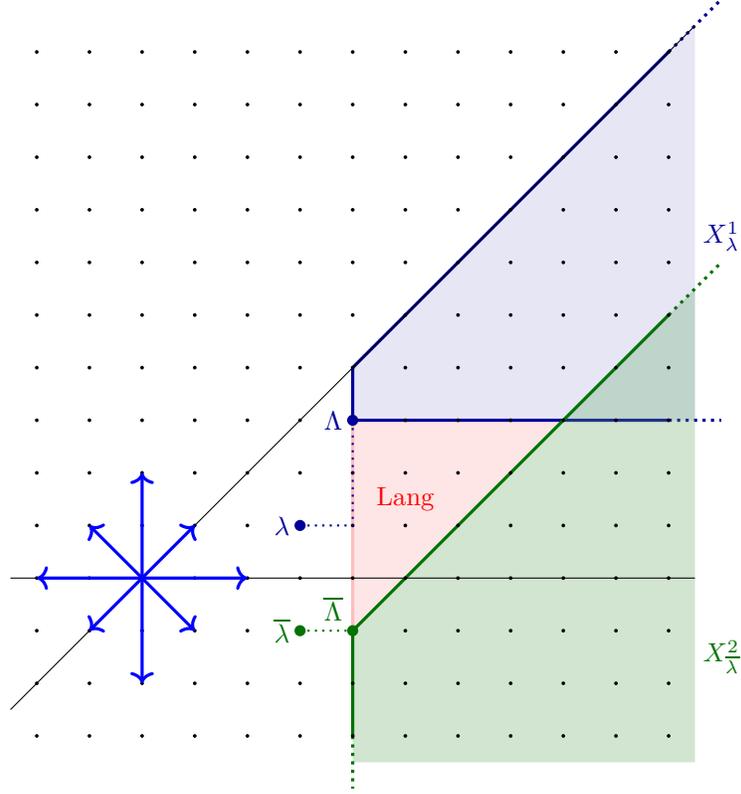

We now want to write down the analogous statement for representations of $\GSp_4(\R)$. Recall that for any character $\eta = |\cdot|^s \sgn^\epsilon$ for $s \in \C$, $\epsilon \in \{0,1\}$ and any $k \in \Z \setminus \{0\}$, we have a (limit of) discrete series representation $\delta(\eta, k)$ fitting in the short exact sequences:
\begin{center}
	\begin{tikzcd}
		0 \ar[r] & \delta(\eta, k) \ar[r] & \eta |\cdot|^{k/2} \sgn^{k+1} \times \eta |\cdot|^{-k/2} \ar[r] & \zeta(\eta, k) \ar[r] & 0 \\[-0.5cm]
		0 \ar[r] & \zeta(\eta, k) \ar[r] & \eta |\cdot|^{-k/2} \times \eta |\cdot|^{k/2} \sgn^{k+1} \ar[r] & \delta(\eta, k) \ar[r] & 0
	\end{tikzcd}
\end{center}
where $\chi_1 \times \chi_2$ denotes parabolic induction from the standard Borel of $\GL_2(\R)$ and $\zeta(\eta, k)$ is finite-dimensional (of dimension $k$ when $k > 0$). Moreover,
$$\delta(\eta, k)|_{\SL_2(\R)} \iso D_k^+ \oplus D_k^-.$$

\begin{corollary}
	Let $\lambda = (\lambda_1, \lambda_2)$ for integers $\lambda_1 \geq \lambda_2 \geq 0$. Parabolic induction from the Klingen parabolic $Q$ of $\GSp_4(\R)$ has the following composition series:
	\begin{center}
		\begin{tikzcd}
			0 \ar[r] & X_{\lambda}^1 \oplus X_{\overline \lambda}^2 \ar[r] & \sigma_{\lambda_2} \rtimes \delta(\eta, \lambda_1) \ar[r] & \mathrm{Lang}(\sigma_{\lambda_2} \rtimes \delta(\eta, \lambda_1)) \ar[r] & 0,
		\end{tikzcd}
	\end{center}
	where $\eta$ determines the central characters of $X_\lambda^1$ and $X_{\overline \lambda}^2$ and $ \mathrm{Lang}(\sigma_{\lambda_2} \rtimes \delta(\eta, \lambda_1))$ is a Langlands quotient representation.
\end{corollary}

The picture is analogous to Figure~\ref{fig:K-types} but also contains the $K$-types obtained by reflection about the $\lambda_2 = -\lambda_1$ diagonal.

Finally, we describe the representations occurring in the parabolic induction from the Siegel parabolic for completeness.

\begin{theorem}[Mui\'c]
	Let $\lambda = (\lambda_1, \lambda_2)$ for integers $\lambda_1 \geq \lambda_2 \geq 0$. Parabolic induction from the Siegel parabolic $P \cap \Sp_4(\R)$ of $\Sp_4(\R)$ has the composition series:
	\begin{center}
		\begin{tikzcd}
			0 \ar[r] & X_{\overline \lambda}^2 \oplus X_{\overline \lambda'}^3 \ar[r] & \delta( |\cdot|^{(\lambda_1 - \lambda_2)/2} \sgn^{\lambda_2}, \lambda_1 + \lambda_2 ) \rtimes \one \ar[r] & \mathrm{Lang} \ar[r] & 0
		\end{tikzcd}
	\end{center}
	where $\mathrm{Lang}$ is a Langlands quotient.
\end{theorem}

This leads to an analogous statement for $\GSp_4(\R)$ which we omit here.

\subsection{Weyl group action}\label{subsection:Weylgroup}

The contributions of automorphic forms to the cohomology of local systems~$V_\mu^\vee$ associated with a highest weight $\mu$ on Shimura varieties $X$ can be described in terms of the dual BGG decomposition~\cite[(2.3)]{Lan:Coh_of_AB}:
\begin{equation}\label{eqn:dualBGG}
	H^i_{\mathrm{dR}} (X, \underline V_{\mu}^\vee) \iso \bigoplus_{w \in W^M} H^{i - \ell(w)}(X^{\mathrm{tor}}, (\underline W_{w \cdot \mu}^{\mathrm{can}} )^\vee).
\end{equation}
We refer the reader to loc.\ cit.\ for the notation; importantly for us, $W^M$ is a set of minimal length representatives for the quotient $W_G/W_M$ of the Weyl group $W_G$ of $G$ by the Weil group $W_M$ of the Levi $M$ of a maximal parabolic of $G$.

For {\em low weights} which are not cohomological, it seems that the right-hand side of~\eqref{eqn:dualBGG} still gives all the contributions to cohomology of automorphic vector bundles of the $L$-packet of automorphic forms. For example, this is true for modular forms of weight one, Hilbert modular forms of partial weight one, and Siegel modular forms of weight $(k,2)$. It is then natural to ask whether the representatives $w \in W^M$ can be used to define maps between the various coherent cohomology groups.

This is the case for modular forms and Hilbert modular forms. As observed by Harris~\cite{Harris_periods_I}, one can use operators associated with the representatives $\begin{pmatrix} 1 & 0 \\ 0 & -1 \end{pmatrix} \in \GL_2(\R)$ of the Weyl group elements to define {\em partial complex conjugation maps} between coherent cohomology groups. For low weights, these maps can then be used to define a {\em motivic action} on coherent cohomology~\cite{Horawa}.

Returning to the case of $G = \GSp_4$, a natural candidate for the map
$$H^0(X_\C, \mathcal E_0) \to H^1(X_\C, \mathcal E_1)$$
is the Weyl operator $w_0 \in W_G(S)(\R)$ which acts on $K^\circ$-types by $(n_1, n_2) \mapsto (n_1, -n_2)$. In this section, we explain that this Weyl element does not lift to an element in $N_G(S)(\R)$ , and hence this na\"ive construction does not produce the desired operator. We thank Tasho Kaletha for explaining this to us.

Let $T = \{ \mathrm{diag}(t_1, t_2, t_1^{-1}, t_2^{-1}) \ | \ t_1, t_2 \in \R\}$ be the split maximal torus of $\Sp_{4, \R}$. We then have two short exact sequences:
\begin{align}
	1 \to T \to N_G(T) \to W_G(T) \to 1 \label{eqn:ses_WGT} \\ 
	1 \to S \to N_G(S) \to W_G(S) \to 1 \label{eqn:ses:WGS}
\end{align}
and both the Weyl groups $W_G(T)$ and $W_G(S)$ are isomorphic to the dihedral group of order 8 as algebraic groups over $\R$. 

The key observation is that the short exact sequence~\eqref{eqn:ses_WGT} remains exact on $\R$-points, i.e.\ every element of $W_G(T)(\R)$ has a representative in $N_G(T)(\R)$. Explicitly, the Weyl element $w_0 \in W_G(T)(\R)$ which sends $(n_1, n_2) \mapsto (n_1, -n_2)$ has the following representative:
$$g_0 = \begin{pmatrix}
	1 & & & \\
	& & & -1 \\
	& & 1 \\
	& 1
\end{pmatrix} \in N_G(T)(\R).$$

However, the short exact sequence~\eqref{eqn:ses_WGT} fails to be exact on $\R$-points. In fact, the elements of $W_G(S)(\R)$ that lift to $N_G(S)(\R)$ are exactly the elements that can be realized in $N_{K^\circ}(S)(\R)$. In other words:
$$\mathrm{Im}(N_G(S)(\R) \to W_G(S)(\R)) = \mathrm{Im}(W_{K^\circ}(S)(\R) \to W_{G}(S)(\R)).$$
Therefore, the element $w \in W_G(S)(\R)$ which sends $(n_1, n_2) \mapsto (n_1, -n_2)$ does not lift to $N_G(S)(\R)$, and hence does not allow us to define the desired operator $H^0(X_\C, \mathcal E_0) \to H^1(X_\C, \mathcal E_1)$. 

Explicitly, the point is that if $P^\circ \subseteq G^\circ = \Sp_4$ is the Siegel parabolic and $M_{P^\circ}$ is its Levi with torus $T$, then $g M_{P^\circ, \C} g^{-1} = K_{\mathbb C}^\circ$, but $g = \begin{pmatrix}
	-i I_2 & i I_2 \\
	I_2 & I_2
\end{pmatrix} \in \GSp_4(\C) \setminus \GSp_4(\R)$. Therefore, a representative over $\C$ of $w$ is given by:
$$g g_0 g^{-1} = \begin{pmatrix}
	1 \\
	& & & i \\
	& & 1 \\
	& i & & 
\end{pmatrix} \not\in \GSp_4(\R).$$

\bibliographystyle{amsalphaurl}
\bibliography{bibliography}

\end{document}